\documentclass[leqno,english,11pt]{amsart}
\usepackage{amsfonts,amssymb,amsmath,amsgen,amsthm}
\usepackage{bbm}
\usepackage{mathrsfs}
\usepackage{graphicx}
\usepackage{hyperref}
\usepackage{xcolor}
\newcommand{\msc}[2][2000]{%
  \let\@oldtitle\@title%
  \gdef\@title{\@oldtitle\footnotetext{#1 \emph{Mathematics subject
        classification.} #2}}%
}

\theoremstyle{plain}
\newtheorem{theorem}{Theorem} [section]

\newtheorem{lemma}[theorem]{Lemma}
\newtheorem{corollary}[theorem]{Corollary}
\newtheorem{proposition}[theorem]{Proposition}

\theoremstyle{remark}
\newtheorem{remark}[theorem]{Remark}

\def\R{{\mathbb R}}
\def\N{{\mathbb N}}
\def\O{\mathcal O}

\def\1{\mathbbm{1}}

\def\({\left(}
\def\){\right)}
\def\<{\left\langle}
\def\>{\right\rangle}

\def\Eq#1#2{\mathop{\sim}\limits_{#1\rightarrow#2}}
\def\Tend#1#2{\mathop{\longrightarrow}\limits_{#1\rightarrow#2}}

\def\d{{\partial}}
\def\eps{\varepsilon}

\def\si{{\sigma}}
\def\K{\mathcal{K}}
\def\ll{{\lambda_*}}

\newcommand{\enstq}[2]{\left\{#1~\middle|~#2\right\}}
\newcommand{\dd}{\mathrm{d}}

\numberwithin{equation}{section}

\begin{document}
\title[On NLS ground state]{On the ground state of the nonlinear Schr\"{o}dinger equation: asymptotic behavior at the endpoint powers} 

\author[R. Carles]{R\'emi Carles}
\address[R. Carles]{CNRS\\ IRMAR - UMR
  6625\\ F-35000 Rennes, France}
\email{Remi.Carles@math.cnrs.fr}

\author[Q. Chauleur]{Quentin Chauleur}
\address[Q. Chauleur]{Univ. Lille, CNRS, Inria\\ UMR 8524 - Laboratoire Paul Painlevé\\ F-59000 Lille, France}
\email{quentin.chauleur@inria.fr}

\author[G. Ferriere]{Guillaume Ferriere}
\address[G. Ferriere]{Univ. Lille, CNRS, Inria\\ UMR 8524 - Laboratoire Paul Painlevé\\ F-59000 Lille, France}
\email{guillaume.ferriere@inria.fr}

\author[D. E. Pelinovsky]{Dmitry E. Pelinovsky}
\address[D. E. Pelinovsky]{Department of Mathematics, McMaster University\\ Hamilton,
  Ontario\\ L8S 4K1, Canada} 
\email{pelinod@mcmaster.ca}

\begin{abstract}
We consider the ground states of the nonlinear Schr\"{o}dinger
equation, which stand for radially symmetric and exponentially decaying
solutions on the full space. We investigate their behaviors at both
endpoint powers of the nonlinearity, up to some rescaling to infer
non-trivial limits. One case corresponds to the limit towards a
Gaussian function called Gausson, which is the ground state of the
stationary logarithmic Schrödinger equation. The other case, for
dimension at least three, corresponds to the limit towards the
Aubin-Talenti algebraic soliton. We prove strong convergence with
explicit bounds for both cases, and provide detailed
asymptotics. These theoretical results are illustrated with numerical
approximations.  
\end{abstract}

\thanks{A CC-BY public
copyright license has been applied by the authors to the present
document and will be applied to all subsequent versions up to the
Author Accepted Manuscript arising from this submission.}
\maketitle

\section{Introduction}
\label{sec:intro}

We consider the ground states of the stationary nonlinear
Schr\"odinger equation
\begin{equation}\label{eq:phi}
  -\Delta \phi+\phi = |\phi|^{2\si}\phi,\quad x\in \R^d,
\end{equation}
with emphasis on the dependence of the solution upon the parameter
$\si>0$ in the nonlinearity. It has been known since the breakthrough
works \cite{BGK83,BL83a} that ground states, defined as a minimizer of the action, exist in $H^1(\R^d)$ provided that the nonlinearity is
$H^1$-subcritical: $0 < \si < \infty$ for $d = 1,2$ and 
$0<\si<\frac{2}{d-2}$ for $d\ge 3$. The uniqueness of such solutions,
up to translation and sign change, was established in \cite{BL83a} for
$d=1$, and completely settled in \cite{Kwong} for $d\ge 2$, after a
series of important steps, cited in \cite{Kwong}. The ground states are the (unique) positive, radially symmetric solutions to \eqref{eq:phi}. We recall
that $\phi\in \mathcal{C}^2(\R^d)$, and that $\phi,\nabla \phi$ decay
exponentially (see e.g. \cite[Theorem~8.1.1]{CazCourant}). 

In the present paper, we examine the behavior of the ground states when
the parameter $\sigma$ in the nonlinearity goes to the endpoint
values, $\si=0$ in any dimension, and $\si=\si_*(d):=\frac{2}{d-2}$ when
$d\ge 3$. In what follows, we omit the dependence on $d$ in $\sigma_*$.

For the limit $\si\to 0$, the Taylor expansion
\begin{equation}
\label{expansion-log}
  |\phi|^{2\si} = \exp\(\si \ln |\phi|^2\) = 1+\si\ln |\phi|^2 +\O(\si^2)
\end{equation}
suggests, in order to get a nontrivial limit, to consider, instead of
\eqref{eq:phi},
\begin{equation}
  \label{eq:u}
  \Delta u +\frac{1}{\si}\(|u|^{2\si}-1\)u=0.
\end{equation}
As we work on the whole space $\R^d$, this amounts to considering the
rescaling
\begin{equation}\label{eq:scaling-phi-u}
  u(x) = \phi\(\frac{x}{\sqrt\si}\). 
\end{equation}
Formally, when $\si$ goes to zero, the solution $u$ to \eqref{eq:u} is
expected to converge in some sense to a solution of the
stationary logarithmic Schr\"odinger equation,
\begin{equation}\label{eq:log-u}
  \Delta u +u\ln\(|u|^2\)=0. 
\end{equation}
Equation \eqref{eq:log-u} is the stationary counterpart of the time dependent logarithmic Schr\"odinger equation,
\begin{equation}\label{eq:logNLS}
  i\d_t \psi +\Delta \psi = \lambda \psi \ln\(|\psi|^2\),
\end{equation}
with $\lambda\in \R$ ($\lambda=-1$ here) initially introduced in
\cite{BiMy76}. It was remarked there that \eqref{eq:log-u} has explicit
ground states for any $d\in \N$, called Gaussons (see also
\cite{BiMy79}),
\begin{equation}
  \label{eq:Gaussian}
  u_0(x) = e^{\frac{d-|x|^2}{2}}. 
\end{equation}
The Cauchy problem for \eqref{eq:logNLS} in the case $\lambda<0$ was
studied initially in \cite{CaHa80}, and the orbital stability of the
Gaussons was proven in \cite{Caz83} in the radial case, and in \cite{Ar16}
for the general case. The fact that the Gausson \eqref{eq:Gaussian} is
the only (up to translation) positive, $\mathcal{C}^2$ solution of
\eqref{eq:log-u} vanishing at 
infinity,   was proven in \cite{dAMS14}. Uniqueness of
positive, radially symmetric solutions of \eqref{eq:log-u} vanishing at
infinity as well as their derivative  was 
established in  \cite{Troy2016}, 
for $1\le d\le 9$. Viewing ground states as solutions of 
 a  constrained minimization problem (minimization of the action on
 the Nehari manifold), uniqueness of ground states (up to translation and phase
 modification) was proven in \cite{Ar16}.

The convergence of ground states of
\eqref{eq:phi} to ground states for \eqref{eq:log-u} was considered
for the first time in \cite{WangZhang2019}. 
The scaled equation \eqref{eq:u} is also considered in \cite{GMS-p}, where
the limit $\si\to 0$ is addressed, for $x$ belonging to some bounded
and convex domain. In the case $x\in \R^d$, it is proven in
\cite{WangZhang2019} that ground states to \eqref{eq:u} converge to
ground states of \eqref{eq:log-u} in $H^1(\R^d)\cap
\mathcal{C}^{2,\alpha}(\R^d)$ for any $\alpha \in (0,1)$. In the present paper,
we revisit this convergence result by providing a rate of
convergence in $\O(\si)$ as suggested by \eqref{expansion-log}. 

In \cite{WangZhang2019}, based
on a result from \cite{FQTY08}, the authors infer that for any $\si\in
(0,\si_*)$, there is no positive solution to \eqref{eq:phi}  or, equivalently, to \eqref{eq:u}, in view of \eqref{eq:scaling-phi-u}, such
that $$\|\phi\|_{L^\infty}= \|u\|_{L^\infty}\le e^{d/2}.$$ However,
since in \cite{FQTY08}, the assumption $d\ge 3$ is made, one should be
cautious with low dimensions. Indeed, when $d=1$, ground states for
\eqref{eq:phi} are given explicitly by
\begin{equation*}
  \phi(x) =(1+\sigma)^{\frac{1}{2\sigma}} \cosh \left(
    \si x \right)^{-\frac{1}{\sigma}}. 
\end{equation*}
We note that for $\si>0$,
\begin{equation*}
  \|\phi\|_{L^\infty} = \phi(0) =
  (1+\sigma)^{\frac{1}{2\sigma}} =e^{\frac{1}{2\si}\ln (1+\si)}<
  e^{\frac{1}{2}},
\end{equation*}
so Theorem~1.3 in \cite{WangZhang2019} cannot be true for $d=1$. We
refer to Remark~\ref{rem:correction} for a more precise discussion.

In the $H^1$-critical case $\si=\si_*$ (for $d\ge 3$), the existence
of ground states goes back to \cite{Aubin76} and \cite{Talenti76}
independently. In view of Pohozaev identity (see
e.g. \cite{CazCourant}), nontrivial $\dot{H}^1 \cap L^{2\sigma_*+2}$ solutions satisfy the following equation, instead of \eqref{eq:phi},
\begin{equation}
\label{eq-phi-star}
  \Delta \phi_* +|\phi_*|^{2\si_*}\phi_*=0.
\end{equation}
Positive radially symmetric solutions to \eqref{eq-phi-star}
cease to be unique, due to a scaling invariance: 
if $\phi_*(x)$ is a solution to \eqref{eq-phi-star}, then so is
$\lambda^{1/\si_*}\phi_*(\lambda x)$ for any $\lambda>0$. Up to this
scaling invariance, the radially symmetric positive solutions are unique, given by
\begin{equation}
\label{alg-sol}
  \phi_*(x) = \frac{1}{(1+a |x|^2)^{(d-2)/2}},\quad a =
  \frac{\si_*^2}{4(1+\si_*)}= \frac{1}{d(d-2)}. 
\end{equation}
For any $d\ge 3$, $\phi_*$ belongs to the \emph{homogeneous} Sobolev
space $\dot H^1(\R^d)$ (that is, $\nabla\phi_*\in L^2(\R^d)$), but
$\phi_*\in L^2(\R^d)$ only if $d\ge 5$.
The limit $\si\to \si_*$
has  been considered previously, in \cite{FeGa04,GaPePuSe03,GaSe02},
where the Laplacian is replaced more generally by the $m$-Laplacian
($m=2$ corresponding to the regular Laplacian). Similar to the case
$\si\to 0$, where the
rescaling \eqref{eq:scaling-phi-u} was introduced in order to get a
nontrivial limit, the limit $\si\to \si_*$ requires a modification in
order to make the limit regular, and establish a connection with the
algebraic soliton \eqref{alg-sol}. This is discussed more precisely in
Section~\ref{sec:intro-si_*}, where we also compare our results with
those of \cite{FeGa04,GaPePuSe03,GaSe02}.

We conclude this introduction by illustrating in
Figure~\ref{fig:evolution_max_d_1_2_3_4_5.png} the dependence of the 
$L^{\infty}$-norm of the ground states of \eqref{eq:u} for
$d=1,\ldots,5$. The dependence is monotonically decreasing in dimensions $d=1$ and $d=2$, whereas it is monotonically 
increasing and diverges as $\sigma \to \sigma_*$ in dimensions $d=4$
and $d=5$, suggesting a renormalization in order to study the limit
towards $\phi_*$ in \eqref{alg-sol}, as evoked above. For $d=3$, the dependence first
decreases for small values of $\sigma$ and then increases and diverges
at $\sigma_*= 2$.  

\begin{figure}[htb!]
   \centering
      \includegraphics[width=0.9\textwidth,trim = 25cm 0cm 25cm 2cm, clip]{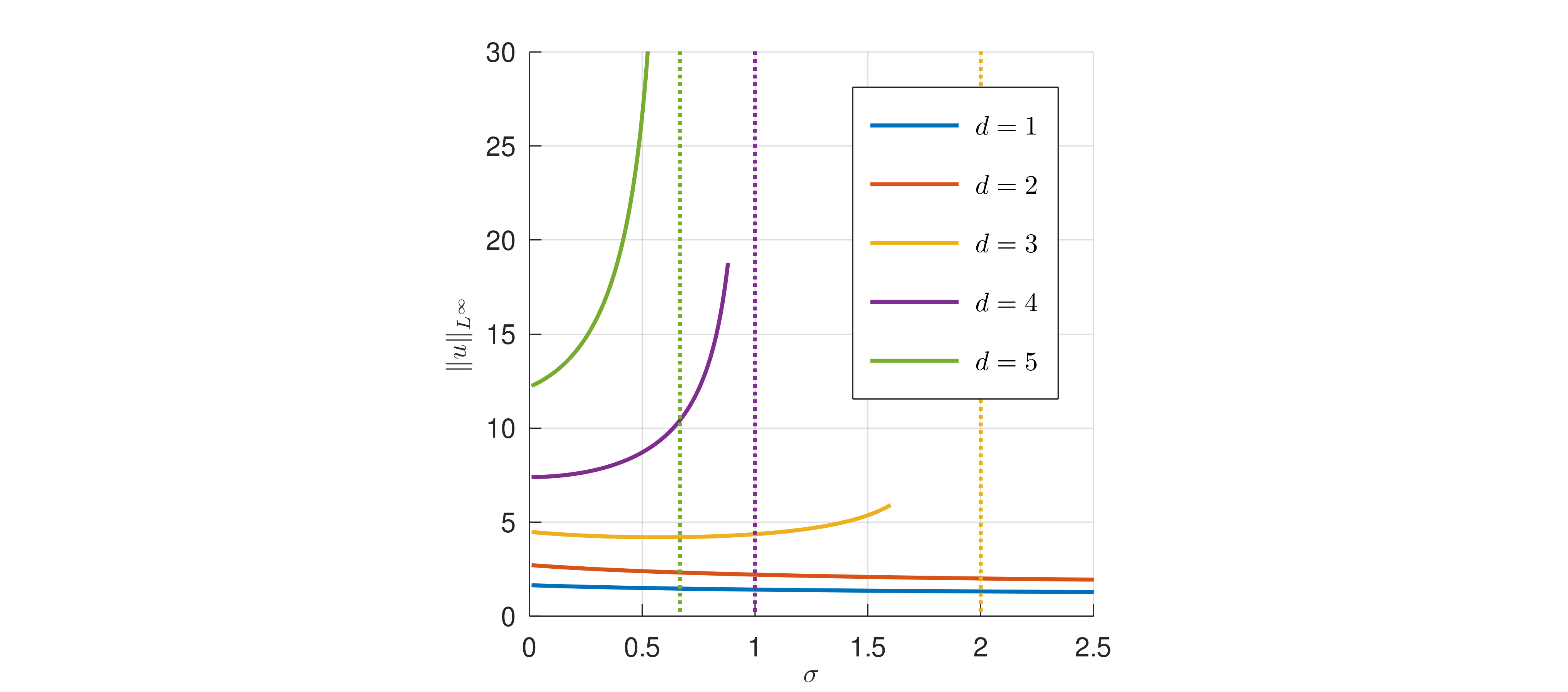}	 
	\caption{Maximum of the ground states $\|u\|_{L^{\infty}}$ versus $\sigma$ for $d=1,\ldots,5$ (see Remark \ref{rem:stiffness} for comments on the limit $\si \to \si_*$).}
	\label{fig:evolution_max_d_1_2_3_4_5.png}
\end{figure}

Precise statements of the main results on the asymptotic behavior of the ground states at the endpoint powers are given in Section~\ref{sec-main_results}. Section~\ref{sec:continuity} is dedicated to continuity properties 
with respect to the nonlinearity parameter $\si \in (0,\si_*)$. In Section~\ref{sec:log}, we consider the limit $\si\to 0$. The other
endpoint $\si\to \si_*$ is studied in
Section~\ref{sec-alg-limit}. Details about the numerical methods and
the numerical approximations are presented in Section~\ref{sec:num}.  
 
\smallskip

{\bf Notations.} The differential element is denoted by $\dd$ to avoid any confusion with the space dimension $d$. For radially symmetric functions and $1\le p<\infty$, we denote by $L^p_r=L^p_r(0,\infty)$ the 
set of functions $f=f(r)$ such that
\begin{equation*}
 \| f \|_{L^p_r}^p := \int_0^\infty r^{d-1}|f(r)|^p \dd r<\infty,
\end{equation*}
and by $H_r^k=H^k_r(0,\infty)$ for $k\in \N^*$ the set of functions
$f$ such that
\begin{equation*}
   \| f \|_{H^k_r}^2 := \int_0^\infty r^{d-1}
   \left(|f^{(k)}(r)|^2+\ldots+|f'(r)|^2+|f(r)|^2 \right) \dd
   r<\infty.  
\end{equation*}
These definitions discard the measure of the unit sphere in $\R^d$ to
lighten notations. This measure is only included in
Section~\ref{sec:num} for numerical computations. Finally we denote by 
\[ \langle f,g \rangle := \int_0^\infty r^{d-1} f(x) g(x) \dd x    \]
the scalar product on $L^2_r$.

\section{Main results} \label{sec-main_results}

In order to emphasize the dependence of the ground state profile upon the
nonlinearity parameter $\si$, we denote the
radially symmetric, positive, and monotonically decreasing solution to \eqref{eq:u} by $u_\si$, and make the standard abuse of notation in  $u_\si(x)=u_\si(r)$, $r=|x|$. 
In the radial coordinate $r$, studying \eqref{eq:u} amounts to considering the
family of solutions of the initial-value problem  
\begin{equation} 
\label{eq:ivp}
\left\{ \begin{aligned}
&u''(r) + \frac{d-1}{r} u'(r) + \frac{1}{\sigma}(|u(r)|^{2 \sigma} -1)
  u(r) = 0, \quad r >0, \\
&          u(0) = \alpha, \quad u'(0) = 0,
        \end{aligned} \right.
\end{equation}
for $\alpha > 0$ and $\sigma \in (0,\si_*)$. This family of solutions 
is denoted by $u(r;\alpha,\sigma)$. It is known (see \cite[Theorem
1.3]{FQTY08}) that for each $\sigma \in (0,\sigma_*)$,
there exists a unique value of $\alpha = \alpha(\sigma)$ such that
$u_{\sigma}(r) := u(r;\alpha(\sigma),\sigma)$ is a positive,
monotonically decreasing function in $\mathcal{C}^2(0,\infty) \cap 
L^{\infty}(0,\infty)$, with exponential decay as
$r \to \infty$. This is the ground state profile of \eqref{eq:u}. Due to its uniqueness (\cite{Kwong}), it 
coincides up to a scalar multiplication with a minimizer of the
variational problem  
\begin{equation}
\label{eq:variation}
\inf_{u \in H^1_r} \enstq{ \| u' \|_{L^2_r}^2 + \frac{1}{\si} \| u \|_{L^2_r}^2  }{ \ \| u \|_{L^{2\si+2}_r} = 1 }, 
\end{equation}
the existence of which was considered in \cite{BGK83,BL83a}. 

\subsection{Continuity with respect to $\si$}

Related to the ground state profile $u_{\sigma} \in \mathcal{C}^2(0,\infty) \cap
L^{\infty}(0,\infty)$, we also consider the linearized operator  
$\mathcal{L}_{\sigma} : H^2_r  \to L^2_r$ given by 
\begin{equation}
\label{eq:L}
\mathcal{L}_{\sigma} = - \frac{\dd^2}{\dd r^2} - \frac{d-1}{r} \frac{\dd}{\dd r}
+ \frac{1}{\sigma} \left(1 - u_{\sigma}^{2\sigma} \right) - 2
u_{\sigma}^{2\sigma}. 
\end{equation} 
In Section \ref{sec:continuity}, we prove the following theorem, which allows us to control the dependence of 
$\alpha(\sigma)$ upon $\sigma \in (0,\sigma_*)$.

\begin{theorem} \label{th-dependence}
The mapping $\si \mapsto u_\si$ is $\mathcal{C}^1$ in $(0, \si_*)$ with values in $H^1_r$. Moreover, $\frac{\dd u_\si}{\dd \si}=\chi_\si$ where $\chi_\si \in H^2_r \cap \mathcal{C}^2$ is the unique solution in $H^1_r$ to
\begin{equation} \label{eq:inhom}
\mathcal{L}_{\sigma} \chi_\si = \frac{1}{\sigma^2} (1 -
|u_{\sigma}|^{2\sigma}) u_{\sigma} + \frac{1}{\sigma} (\ln
 u_{\sigma}^2) |u_{\sigma}|^{2\sigma} u_{\sigma}. 
\end{equation}
The mapping $(0,\sigma_*) \ni \sigma \mapsto \alpha(\sigma) \in (0,\infty)$ 
is $\mathcal{C}^1$ and $\alpha'(\sigma) = \chi_\si(0)$.
\end{theorem}

\subsection{The limit $\si \to 0$}

 We gather the main results of Section \ref{sec:log} in the following statement:
\begin{theorem}\label{theo:log}
  Let $d\ge 1$. As $\si\to 0$, the ground state profile $u_\si$ of
  \eqref{eq:ivp} converges to the Gausson $u_0$ given by
  \eqref{eq:Gaussian} in $H^1_r \cap \mathcal{C}^{2, \alpha} \cap \mathcal{C}^\infty_{\textnormal{loc}}$ for any $0 < \alpha < 1$.
We have the asymptotic expansion
  \begin{equation*}
    u_\si = u_0 + \si \mu_0 + \si e_\si,
  \end{equation*}
  where
  \begin{equation*}
    \mu_0(r) = \frac{1}{12} \left[ d (d-4) + 4 (1-d) r^2 + r^4
  \right] u_0(r),
\end{equation*}
and for every $0\le s<1$,  $ e_\si$ goes to zero in
$H^s_r\cap \mathcal C^\infty_{\rm loc}$ as $\si\to 0$.
In particular, the mapping $\si \mapsto \alpha(\si)$ is continuously
differentiable as $\si\to 0$, with 
  \begin{equation*}
    \alpha(0) =e^{d/2} \quad \text{and} \quad
    \alpha'(0)=\frac{d(d-4)}{12}e^{d/2}. 
  \end{equation*}
\end{theorem}

The computation of $\alpha'(0)$ is new,
and illustrated numerically in Section~\ref{sec:num}. The computation
of the correcting term $\mu_0$ is new too, as well as the
corresponding error estimate for $e_{\si}$. 
This result shows that some statements from \cite{FQTY08} and
\cite{WangZhang2019} are flawed in the case $d\le 3$. See
Remark~\ref{rem:correction} for details. 

\subsection{The limit $\si \to \si_*$}
\label{sec:intro-si_*}
 The numerical data from
 Figure~\ref{fig:evolution_max_d_1_2_3_4_5.png} suggests  that
 $\alpha(\sigma) \to \infty$ as $\sigma \to \sigma_* = \frac{2}{d-2}$
 for $d \geq 3$. In order to get the asymptotic dependence of
 $\alpha(\sigma) \to \infty$, we use the scaling transformation  
\begin{equation}
\label{eq:transformation}
u(r) = \alpha w(\rho), \quad \rho = \frac{\alpha^{\sigma}
  r}{\sqrt{\sigma}}, \quad \alpha > 0, \quad \sigma \in (0,\sigma_*). 
\end{equation}
If $u$ satisfies \eqref{eq:ivp}, then $w$ satisfies the initial-value
problem: 
\begin{equation} 
\label{eq:ivp-w}
\left\{
  \begin{array}{l}
w''(\rho) + \frac{d-1}{\rho} w'(\rho) +|w(\rho)|^{2 \sigma} w(\rho) =
    \epsilon w(\rho), \\ 
    w(0) = 1, \quad w'(0) = 0,
  \end{array}
\right.
\end{equation}
where $\epsilon := \alpha^{-2\sigma}$. Every solution $u =
u(r;\alpha,\sigma)$ of \eqref{eq:ivp} is equivalent  
to the solution $w = w(\rho;\epsilon,\sigma)$ of
\eqref{eq:ivp-w}. Again, if $u_{\sigma}(r) = 
u(r;\alpha(\sigma),\sigma)$ for some $\alpha = \alpha(\sigma)$ is the
ground state (a positive, monotonically decreasing function in
$\mathcal{C}^2(0,\infty) \cap L^{\infty}(0,\infty)$, with the fast (exponential)
decay condition as $r \to \infty$), then  
$w_{\sigma}(\rho) = w(\rho;\epsilon(\sigma),\sigma)$ is the ground
state for  $\epsilon = \epsilon(\sigma)$, where 
\begin{equation}
\label{eq:eps-alpha}
\epsilon(\sigma) := [\alpha(\sigma)]^{-2\sigma}. 
\end{equation}
The limit $\alpha(\sigma) \to \infty$ corresponds now to the limit 
$\epsilon(\sigma) \to 0$, where the limiting ground state is
represented by  the Aubin--Talenti algebraic soliton \eqref{alg-sol},
rewritten as  
\begin{equation}
\label{eq:alg-soliton}
w_*(\rho) = \frac{1}{(1+a \rho^2)^{\frac{1}{\sigma_*}}}, \quad a :=
\frac{\sigma_*^2}{4 (1+\sigma_*)}, \quad \sigma_* = \frac{2}{d-2}. 
\end{equation}
We recall (\cite{Aubin76,Talenti76}) that the Aubin--Talenti algebraic
soliton \eqref{eq:alg-soliton}  coincides up to a scalar
multiplication with the unique minimizer of the variational problem  
\begin{equation}
\label{eq:variation-alg}
\mathcal{S} := \inf_{w \in D^{1,2}_r(0,\infty)} \enstq{ \|  w'
  \|_{L^2_r} }{ \| w \|_{L_r^{\frac{2d}{d-2}}} = 1 },  
\end{equation}
where $D^{1,2}_r(0,\infty)$ is the space of closure of
$\mathcal{C}^{\infty}_{0,r}(0,\infty)$ under the norm $\| \nabla \cdot
\|_{L^2_r}$. The minimizer of \eqref{eq:variation-alg} gives the best
constant of the Sobolev inequality  
\begin{equation}
\label{eq:Sob}
\| w \|_{L_r^{\frac{2d}{d-2}}} \leq \mathcal{S}^{-\frac{1}{2}} \|  w'
\|_{L^2_r}. 
\end{equation}
Furthermore, it is only degenerate due to the one-parameter scaling
transformation introduced before, $w_*(\rho) \mapsto
\lambda^{1/\si_*} w_*(\lambda \rho)$ with $\lambda > 0$ (see
\cite{CGS89} and 
the appendix  in \cite{BianchiEgnell1991}). Changing $u_{\si}$ satisfying
\eqref{eq:ivp} to $w_\si$ satisfying \eqref{eq:ivp-w} makes the limit
$\si\to \si_*$ regular, since 
it corresponds to $\epsilon \to 0 $ in \eqref{eq:ivp-w}. In
particular, the parameter $\lambda$ in the scaling invariance is
naturally $\lambda =1$, in view of the initial condition: $w_{\si}(0)
= 1$. However, as 
the expression of $\epsilon$ is implicit, the convergence $w_\si\to
w_*$ is quite delicate. In this direction, we prove in
Section~\ref{sec-alg-limit}   
the main result given by the following theorem.

\begin{theorem}\label{theo:algebraic}
Let $d \ge 3$. As $\si\to \si_*$, the ground state $w_\si$ of
\eqref{eq:ivp-w} converges to the Aubin-Talenti algebraic soliton
$w_*$ given by \eqref{eq:alg-soliton} in $ L_r^{\infty}\cap
W^{1,\infty}_{\rm loc}$. Moreover, if $d \ge 5$, we have 
\[ 
w_\si \rightarrow w_* \quad \text{in} \ H^1_r   
\]
and 
\[
\epsilon(\sigma) \Eq \si {\si_*} \frac{(1-\si_*) (\si_*-\si)}{2
	\si_*(1+\si_*)(2+\si_*)}=-  (\si_*-\si)\epsilon'(\si_*) .
\]
For $\alpha(\si)=u_\si(0)=\| u_\si \|_{L^{\infty}(\R^d)}$
where $u_\si$ is the ground state of \eqref{eq:u}, this yields 
the asymptotic behavior for $d \geq 5$:
\begin{equation*}
  \alpha(\si) \Eq \si {\si_*} |\epsilon'(\si_*)|^{-(d-2)/4}
  (\si_*-\si)^{-(d-2)/4},\quad 
\end{equation*}
where the explicit value of $\epsilon'(\si_*)$ appears in the above formula.
\end{theorem}
The uniform convergence $w_\si\to w_*$ was obtained initially in
\cite[Theorem~4]{GaSe02}, and the precise asymptotic behavior of
$\alpha(\si)$, as stated above,  was established in
\cite[Corollary~1]{GaPePuSe03}, by using fine ODE arguments. In
\cite[Theorem~2]{GaPePuSe03}, the convergences $w_\si\to 
w_*$ and $\nabla w_\si\to \nabla w_*$ are proved at every point, and
uniformly outside any neighborhood of the origin. The concentration of
the function $u_\si$ at the origin as $\si\to \si_*$ is measured
in terms of convergence of measures in \cite[Theorem~3]{FeGa04},
\begin{equation*}
  u_\si^{2^*}= u_\si^{d\si_*}\rightharpoonup \nu_d \delta_0,\quad
  |\nabla u_\si|^2\rightharpoonup \nu_d \delta_0 \quad\text{as }\si\to \si_*,
\end{equation*}
where $\nu_d$ is computed explicitly, and corresponds to the value
given by the algebraic soliton $w_*$, up to rescaling. In particular,
$\|\nabla w_\si\|_{L^2}\to \|\nabla w_*\|_{L^2}$. In view of the
uniform convergence $w_\si\to w_*$, this implies the strong
convergence $w_\si\to w_*$ in the homogeneous Sobolev space $\dot
H^1(\R^d)\hookrightarrow L^{2^*}(\R^d)$.  In
Proposition~\ref{prop:unif-cv}, we revisit the uniform convergence
$w_\si\to w_*$, and provide an elementary proof of the convergence in
$ L_r^{\infty}\cap W^{1,\infty}_{\rm loc}$. Our main contribution
compared to the results from \cite{FeGa04,GaPePuSe03,GaSe02} concerns
the convergence in $H^1$, which we establish by variational arguments
as opposed to ODE techniques there. In particular, the convergence in
$L^2$ seems to be completely new. 

\section{Continuity properties in $\sigma \in (0,\sigma_*)$}
\label{sec:continuity}

\subsection{Some properties of the linearized operator}
\label{sec:prop-lin}

The operator $\mathcal{L}_\si$, defined in \eqref{eq:L},  is a
self-adjoint operator in $L^2_r$. Due to the 
exponential decay $u_{\sigma}(r) \to 0$ as $r \to \infty$,  
the essential spectrum of $\mathcal{L}_{\sigma}$ is located on
$[\sigma^{-1},\infty)$  by Weyl's theorem. 

Since $u_{\sigma}$ is characterized variationally as a constrained minimizer of 
\eqref{eq:variation} with a single constraint, the Morse index of
$\mathcal{L}_{\sigma}$ (the number of negative eigenvalues in
$L^2_r$) is either $0$ or $1$, and as
\begin{equation*}
 \langle \mathcal{L}_{\sigma} u_{\sigma}, u_{\sigma} \rangle = -2
 \int_0^{\infty} r^{d-1} |u_{\sigma}(r)|^{2 \sigma + 2} \dd r < 0,  
\end{equation*}
the Morse index is exactly one. Moreover, due to non-degeneracy of constrained 
minimizers of \eqref{eq:variation} (\cite{Kwong,Weinstein85}), the
kernel of $\mathcal{L}_{\sigma}$ is trivial and the rest of its
spectrum in $L^2_r$ is strictly positive and bounded away
from~$0$. By Sturm's theorem, the uniquely defined solution  
$v \in \mathcal{C}^2(0,\infty)$ of the initial-value problem  
\begin{equation} 
\label{eq:ivp-lin}
\left\{ \begin{aligned}
&v''(r) + \frac{d-1}{r} v'(r) + \frac{1}{\sigma}(|u_{\sigma}(r)|^{2
  \sigma} -1) v(r) + 2 |u_{\sigma}(r)|^{2\sigma} v(r) = 0, \\ 
& v(0) = 1, \quad v'(0) = 0, \end{aligned} \right.
\end{equation}
has a single node $r_0 \in (0,\infty)$ 
such that $v(r) > 0$ for $r \in [0,r_0)$ and $v(r) < 0$ for $r \in
(r_0,\infty)$ with the divergence $v(r) \to -\infty$ as $r \to
\infty$.

\subsection{Proof of Theorem~\ref{th-dependence}}

The existence and uniqueness of  $\chi_{\sigma} =
\mathcal{L}_{\sigma}^{-1} h \in H^2_r$ solution to \eqref{eq:inhom} with 
 \begin{equation*}
  h(r) = \frac{1}{\sigma^2} (1 - |u_{\sigma}(r)|^{2\sigma})
  u_{\sigma}(r) + \frac{1}{\sigma} (\ln u_{\sigma}(r)^2)
  |u_{\sigma}(r)|^{2\sigma} u_{\sigma}(r) \in L^2_r,         
\end{equation*}
follows by the spectral theory since ${\rm Ker}(\mathcal{L}_{\sigma}) = \{0\}$.
Moreover, bootstrapping yields $\chi_{\sigma} \in
\mathcal{C}^2(0,\infty) \cap 
L^{\infty}(0,\infty)$ with the fast (exponential)
decay $\chi_{\sigma}(r) \to 0$ as $r \to \infty$.  	
The nonlinear operator function
\begin{equation*}
 F(u,\sigma) : H^2_r \times (0,\sigma_*) \to
 L^2_r, \quad F(u,\sigma) = -\Delta_r u + \frac{1}{\sigma}
 (1 - |u|^{2\sigma}) u,   
\end{equation*}
is $\mathcal{C}^1$ in $(u,\sigma)$, and by definition
$F(u_{\sigma},\sigma) = 0$. As $u_\si$ is positive and exponentially
decreasing at infinity, the Jacobian $\mathcal{L}_{\sigma} = D_u
F(u_{\sigma},\sigma)$ maps $H^2_r$ to $L^2_r$.  
In view of Section~\ref{sec:prop-lin}, this Jacobian is
invertible. The implicit 
function theorem then implies that the mapping $(0,\sigma_*) \ni \sigma \mapsto u_{\sigma} \in
H^2_r$ is $\mathcal C^1$. From Peano's Theorem (see
e.g. \cite[Chapter~V]{Hartman}), the derivative $\frac{\dd u_\si}{\dd \si}$
also belongs to $\mathcal{C}^2$, and satisfies the same 
equation as $\chi_\si$, that is \eqref{eq:inhom}. By uniqueness, we
conclude $\frac{\dd u_\si}{\dd \si}=\chi_\si$, 
hence Theorem~\ref{th-dependence}, since $\alpha(\si)=u_\si(0)$.

\subsection{Correspondence to earlier results}
Peano's Theorem also implies that the family of solutions of
the initial-value problem \eqref{eq:ivp} is $\mathcal{C}^1$ with respect to both
$\alpha$ and $\sigma$ with
\begin{equation*}
 v(r) := \partial_{\alpha} u(r;\alpha(\sigma),\sigma) \quad \mbox{\rm
   and} \quad \phi(r) := \partial_{\sigma} u(r;\alpha(\sigma),\sigma),  
\end{equation*}
where $v \in \mathcal{C}^2(0,\infty)$ solves  \eqref{eq:ivp-lin} 
and $\phi \in \mathcal{C}^2(0,\infty)$ solves the linear inhomogeneous
equation $\mathcal{L}_{\sigma} \phi = h$ with the initial condition
$\phi(0) = \phi'(0) = 0$. Considering $v$ goes back to
\cite{Kolodner1955}, with a first application in 
\cite{Coffman72} to prove uniqueness results, and considering $\phi$ goes back
to \cite{FelmerQuaas03}. In \cite{FQTY08}, both functions were used.
By the linear superposition principle, we have  
\begin{equation}
\label{eqs:phi}
	\phi(r) = \chi_{\sigma}(r) - \chi_{\sigma}(0) v(r),
\end{equation}
where $\chi_{\sigma} = \mathcal{L}_{\sigma}^{-1} h \in \mathcal{C}^2(0,\infty) \cap
L^{\infty}(0,\infty)$ was considered above.

The solution $\phi$ generally diverges as $r \to \infty$, if $\chi_{\sigma}(0) \neq 0$. We show that $\phi(r) < 0$ for small $r >  0$ in agreement with \cite[Lemma 3.1]{FQTY08}. Indeed, we  have $\phi''(0) = -d^{-1}
h(0)$ with
\begin{equation*}
  h(0) = \frac{1}{\sigma^2} (1 - \alpha^{2\sigma}) \alpha + \frac{1}{\sigma} (\ln \alpha^2) \alpha^{2\sigma} \alpha \equiv \mathfrak{h}(\alpha,\sigma).
\end{equation*}
Since
\begin{equation*}
  \lim_{\sigma \to 0} \mathfrak{h}(\alpha,\sigma) = \frac{1}{2} (\ln \alpha^2)^2 \alpha > 0,
\end{equation*}
and
\begin{equation*}
  \frac{\partial}{\partial \sigma} \sigma^2
  \mathfrak{h}(\alpha,\sigma) = \sigma (\ln \alpha^2)^2 \alpha > 0, 
\end{equation*}
we have $\mathfrak{h}(\alpha,\sigma) > 0$ for every $\sigma \in
(0,\sigma_*)$ and $\alpha > 0$. Therefore, $\phi''(0) < 0$ and
$\phi(r) < 0$ for small $r > 0$.  

\begin{remark}\label{rem:correction}
We will show in Section \ref{sec:lin-si-0} that 
$$
\alpha(0) = \alpha_0, \quad \alpha'(0) = \frac{d (d-4)}{12} \alpha_0,
$$
where $\alpha_0 = u_0(0) = e^{d/2}$. These results imply the following.
	\begin{itemize}
		\item The results of Theorems~1.1 and 1.2 in \cite{FQTY08} are
		incorrect for $d = 3$. Lemma~2.1 about $v(r)$ is correct and so are
		Lemmas~3.1--3.3 about $\phi(r)$. If $\alpha'(\sigma) = \phi_p(0) <
		0$, as for $d = 3$ and small $\sigma > 0$, then $\phi(r)$ stays
		negative for all $r > 0$ and diverges $\phi(r) \to -\infty$ as $r
		\to \infty$. If $\alpha'(\sigma) = \phi_p(0) > 0$, as for $d \ge 5$
		and small $\sigma > 0$, then $\phi(r)$ changes sign exactly once and
		diverges $\phi(r) \to +\infty$ as $r \to \infty$. The proofs of
		Theorems~1.1 and 1.2 in Sections~5-6 of \cite{FQTY08} are supposed
		to handle both cases; however, the outcome shows that the first case
		is mishandled.
		\item The result of Theorem~1.3 in  \cite{WangZhang2019}, based on the
		above mentioned result from \cite{FQTY08}, is incorrect
		for $1 \le d \le  3$: there are positive solutions to \eqref{eq:phi}
		such that $\|\phi\|_{L^\infty}\le e^{d/2}$ when $d\le 3$ and $\si\in
		(0,\si_*)$, given by $\phi_\si(x) = u_\si(x\sqrt\si)$. 
	\end{itemize}
\end{remark}

\section{The limit $\si\to 0$: convergence to the Gausson}
\label{sec:log}

In this section, we use the fact, proved in
\cite[Theorem~1.1]{WangZhang2019}, that for any $d\ge 1$, 
\begin{equation*}
  \|u_\si-u_0\|_{L^\infty_r}\Tend \si 0 0,
\end{equation*}
where the Gausson is given by
\begin{equation}
\label{Gausson}
u_0(r) = e^{\frac{d-r^2}{2}}.
\end{equation}
The main purpose of this section is to provide the proof of
Theorem~\ref{theo:log}. We first recall the main steps from the proof
of \cite[Theorem~1.1]{WangZhang2019}, and explain why the convergence
 also holds in $\mathcal C^{\infty}_{\rm loc}(\R^d)$.

\subsection{Leading order convergence}
\label{sec:u_si-u_0}

To prove \cite[Theorem~1.1]{WangZhang2019},
the authors establish  a variational characterization of the ground
states $u_\si$ and $u_0$, from which they infer the convergence
$u_\si\to u_0$ in $H^1(\R^d)$ and in $\mathcal{C}^{2, \alpha} (\mathbb{R}^d)$ for any $0 < \alpha < 1$ thanks to the following lemma:
\begin{lemma}[Lemma~2.1 in \cite{WangZhang2019}]
  (i) For any $\eta>0$, there exists $C_\eta>0$ such that
  \begin{equation*}
    \frac{x^{2\si}-1}{\si}\le C_\eta x^{2\eta}
  \end{equation*}
  holds for all $\si\in (0,\eta)$ and $x\ge 0$.\\
  (ii) Let $s>0$, $\delta>0$, then
  \begin{equation*}
    \frac{x^s(x^\delta-1)}{\delta}\Tend \delta 0 x^s \ln x\quad
    \text{in}\quad \mathcal C^{m,\alpha}_{\rm loc}[0,\infty),
  \end{equation*}
  where $m$ is the largest integer with $m<s$, and $\alpha\in
  (0,s-m)$. 
\end{lemma}
We note that the second convergence actually holds in $\mathcal
C^\infty_{\rm loc}(0,\infty)$: for any $0<a<b<\infty$, the convergence
holds uniformly on $[a,b]$ and the same is true for all derivatives,
as can be checked directly. It follows from
\cite[Corollary~2.1]{WangZhang2019} that the ground states $u_\si$ are
uniformly bounded in $L^\infty(\R^d)$,
\begin{equation*}
  \|u_\si\|_{L^\infty} = u_\si(0)\le C, \quad \forall \si\in
  \(0,\frac{2}{d}\),
\end{equation*}
where the bound $\si<2/d$ is here just to fix ideas. With these tools
in hand, standard $L^p$ estimates for elliptic equations (see
e.g. \cite[Theorem~9.11~\& ~9.19]{GilbargTrudinger}) and a bootstrap argument
imply the convergence $u_\si\to u_0$ as $\si\to 0$, in $W^{2k,p}_{\rm
  loc}$ for every 
integer $k\ge 0$ and every $p\in (1,\infty)$, hence in $\mathcal
C^\infty_{\rm loc}$ by Sobolev embedding.

\subsection{Computations of $\alpha'(0)$ and $\mu_0$}
\label{sec:lin-si-0}

The case $\si=0$ may be viewed as a limiting case of
Theorem~\ref{th-dependence}. Recall that the Gausson $u_0$
is the unique positive, radially symmetric, solution in
$\mathcal{C}^2(0,\infty) \cap 
L^{\infty}(0,\infty)$ of the limiting equation  
\begin{equation}
\label{eq:limiting}
u''(r) + \frac{d-1}{r} u'(r) + (\ln u(r)^2) u(r) = 0.
\end{equation}
The associated linearized operator
$\mathcal{L}_0 : {\rm Dom}(\mathcal{L}_0) \subset L^2_r \to
L^2_r$ given by  
\begin{equation}
\label{eq:L0}
\mathcal{L}_0 = - \frac{\dd^2}{\dd r^2} - \frac{d-1}{r} \frac{\dd}{\dd
  r} - \ln
u_0^2 - 2= - \frac{\dd^2}{\dd r^2} - \frac{d-1}{r} \frac{\dd}{\dd r} +r^2-d-2
\end{equation} 
is the (shifted) quantum harmonic Schr\"{o}dinger operator with
\begin{equation*}
  {\rm Dom}(\mathcal{L}_0) = \Sigma^2:=\{f\in H^2(\R^d),\ x\mapsto
  |x|^2 f(x)\in L^2(\R^d)\}
\end{equation*}
in $L^2_r$. By taking the limit $\si \to 0^+$ in \eqref{eq:inhom} and denoting $\mu_0 := \lim_{\si \to 0^+} \chi_{\si}$, we obtain 
the uniquely defined solution of the limiting problem $\mu_0 = \mathcal{L}_0^{-1} h_0$ with the limiting function
\begin{equation*}
  h_0 =\frac{1}{2} (\ln u_0^2)^2 u_0 \in L^2_r.
\end{equation*}
Note that both $\mu_0$ and $h_0$ decays faster (super-exponentially) for $\sigma = 0$ compared to the case $\sigma > 0$. 
Since for any $R>0$, $u_0(x)\ge e^{\frac{d-R^2}{2}}>0$ on the ball $B(0,R)$  of
radius $R$ in $\R^d$, the uniform convergence of $u_\si$ toward $u_0$ implies that $u_\si$ is bounded away from $0$ on  $B(0,R)$ for $\si\le \si(R)$ sufficiently small. The ODE theory implies that the mapping $(0,\sigma_*) \ni \sigma \mapsto  u_{\sigma} \in H^2_r(0,R)$ is also $\mathcal{C}^1$ in the limit $\sigma \to 0^+$.

We can thus consider the dependence $\alpha(\sigma)$ and the solution $\chi_{\si}(r)$ in the
limit $\sigma \to 0$. It follows from \eqref{Gausson} that
$\alpha_0 = u_0(0) = e^{d/2}$. Writing $\mathcal{L}_0 \mu_0 =
h_0$ explicitly, we obtain  
\begin{equation}
\label{eq:explicit}
-\mu_0''(r) - \frac{d-1}{r} \mu_0'(r) - (d-2) \mu_0(r) + r^2 \mu_0(r) = 
\frac{1}{2} (d-r^2)^2 e^{\frac{d-r^2}{2}}.
\end{equation}
Substitution $\mu_0(r) = \frac{1}{2} e^{\frac{d-r^2}{2}}
\tilde{\mu}_0(r)$ converts \eqref{eq:explicit} to the form
\begin{equation*}
  -\tilde{\mu}_0''(r) - \frac{d-1}{r} \tilde{\mu}_0'(r) 
+ 2 r \tilde{\mu}_0'(r) - 2 \tilde{\mu}_0(r) = 
(d-r^2)^2, 
\end{equation*}
polynomial solutions of which are available explicitly:
\begin{equation*}
\tilde{\mu}_0(r) = \frac{1}{6} \left[ d (d-4) + 4 (1-d) r^2 + r^4 \right]. 
\end{equation*}
This yields the expression
\begin{equation}
	\label{expr-mu}
\mu_0(r) = \frac{1}{12} \left[ d (d-4) + 4 (1-d) r^2 +
	r^4 \right]  u_0(r),
\end{equation}
which vanishes at the roots of the polynomial
	\begin{equation} 
	\label{eq:root_r_0}
	d (d-4) + 4 (1-d) r^2 + r^4 = 0 \quad \Leftrightarrow \quad r^2 = 2(d-1)
	\pm \sqrt{3 d^2-4 d + 4}. 
	\end{equation}
There is only one positive root of $r$ for $d \leq 4$ and 
two positive roots for $d \ge 5$. 

Assuming for the moment that $\alpha'(0) = \mu_0(0)$, we get
\begin{equation} \label{eq:slope_at_0}
\alpha'(0) = \frac{d (d-4)}{12} \alpha_0,
\end{equation}
which is negative for $d \leq 3$, zero at $d = 4$, and positive for $d
\ge 5$.  The relation $\alpha'(0)=\mu_0(0)$ is a direct consequence of
the property $e_\si\to 0$ in $\mathcal C^0_{\rm loc}$, which is a
particular case of the error estimate from Theorem~\ref{theo:log},
proven below.

\subsection{Explicit computations for $d = 1$}
\label{sec:1D}
In the one-dimensional case, the ground state is known explicitly, and
elementary computations can be carried out:

\begin{proposition}\label{prop:1D}
  Let $d=1$, 
  \begin{equation*}
  u_\si(x) =(1+\sigma)^{1/(2\sigma)} \cosh \left(
     x \sqrt\si\right)^{-1/\si}
 \end{equation*}
be the ground state associated to \eqref{eq:u}, and
$u_0(x)=e^{(1-x^2)/2}$ be the one-dimensional Gausson. Consider the
corrector
\begin{equation*}
  \mu_0(x) = e^{(1-x^2)/2}\( -\frac{1}{4}+\frac{x^4}{12}\)
  =\frac{1}{12}(x^4-3)u_0(x),  
\end{equation*}
in agreement with \eqref{expr-mu} for $d=1$. Then
\begin{equation*}
  \|u_\si - u_0-\si\mu_0\|_{L^\infty(\R)}+ \|u_\si -
  u_0-\si\mu_0\|_{L^1(\R)}=\O(\si^2).
\end{equation*}
\end{proposition}
In particular, the relation $\alpha'(0)=\mu_0(0)$ follows for $d=1$. 
\begin{remark}
  By interpolation, we also
  have, for any $p\in [1,\infty]$, 
  \begin{equation*}
  \|u_\si - u_0-\si\mu_0\|_{L^p(\R)}=\O(\si^2).
\end{equation*}
Similar estimates for momenta, $\|\<x\>^k(u_\si -
u_0-\si\mu_0)\|_{L^p(\R)}$, where $k>0$, follow easily by the same
argument as below. Controlling Sobolev norms of the error would
require more work though; we leave out this aspect, which is somehow
anecdotal. 
\end{remark}
\begin{proof}
We note that
\begin{equation*}
  \si \mapsto \|u_\si\|_{L^\infty} = u_\si(0) = (1+\sigma)^{1/(2\si)} 
\end{equation*}
is (strictly) decreasing on $\R_+$ (as can be checked by elementary
computations). We readily compute
\begin{equation}
\label{expansion-alpha}
\alpha(\sigma) = (1+\sigma)^{1/(2\si)} = e^{\frac{1}{2\si}\ln(1+\si)}=
 e^{1/2}\( 1 -\frac{1}{4} \si + \O(\si^2)\), 
\end{equation}
in agreement with \eqref{eq:slope_at_0} for $d = 1$, and we focus on
the remaining part defining $u_\si$.
\smallbreak
  
For $x,\si\ge 0$, let
\begin{equation*}
  g_x(\si) : = \ln \cosh \(x\sqrt{\si}\).
\end{equation*}
We have
\begin{equation*}
  g_x(0)=0,\quad \text{and for }\si>0,\quad g_x'(\si) =
  \frac{x}{2\sqrt\si}\tanh\(x\sqrt\si\). 
\end{equation*}
Since we have the expansion
\begin{equation*}
  \tanh(y) = y-\frac{y^3}{3}+\O(y^5)\quad\text{for }0\le y\le
  \frac{\pi}{2}, 
\end{equation*}
we infer in particular
\begin{equation*}
 g_x(\si) = \si\frac{x^2}{2}-\si^2\frac{x^4}{12}+\O\(\si^3x^6\), \quad
 0\le x\le \frac{1}{\sqrt\si}.  
\end{equation*}
Therefore, we have
\begin{equation}
\label{expansion-u}
\tilde u_\si(x):=\cosh \left(
     x \sqrt\si\right)^{-1/\si} = \exp\( \frac{-1}{\si}g_x(\si)\) =
   e^{-x^2/2}e^{\si\frac{x^4}{12}+ R(\si,x)}, 
\end{equation}
where there exists $C$ such that for all $0\le x\le 1/\sqrt\si$,
\begin{equation*}
  |R(\si,x)|\le C\si^2 x^6,
\end{equation*}
hence
\begin{equation*}
  \tilde u_\si(x)= e^{-x^2/2}\(1 +
  \si\frac{x^4}{12}+\O\(\si^2(x^6+x^8)\)\),\quad 0\le x\le 1/\sqrt\si.
\end{equation*}
On the other hand, for $x>1/\sqrt\si$, 
\begin{equation*}
  g_x(\si)\ge \ln \cosh (1),\quad\text{hence}\quad \tilde u_\si(x) \le
  e^{-\frac{\ln \cosh (1)}{\si}}. 
\end{equation*}
Let $\tilde u_0(x) = e^{-x^2/2}$.
 We obviously have $\tilde u_0(x)=\O(e^{-1/(2\si)})$ for
 $x>1/\sqrt\si$, so by symmetry, we infer
\begin{equation*}
 \tilde u_\si(x) = \tilde u_0(x)\(1+
 \si\frac{x^4}{12}\)+\O\(\si^2\)\quad\text{in }L^\infty(\R),
\end{equation*}
which, together with  \eqref{expansion-alpha},  yields the
$L^\infty$-estimate.

For the $L^1$-estimate, let $\tilde v_0(x) = \frac{x^4}{12}\tilde
u_0(x)$, we consider
\begin{equation*}
  \|\tilde u_\si-\tilde u_0-\si\tilde v_0\|_{L^1(\R)}=2\int_0^\infty
  e^{-x^2/2}\left| e^{\si\frac{x^4}{12}+
      R(\si,x)}-1-\si\frac{x^4}{12}\right|\dd x. 
\end{equation*}
Again, we distinguish the regions $0<x\le 1/\sqrt\si$ and
$x>1/\sqrt\si$. From the above Taylor expansion, 
on the first region,
\begin{equation*}
  e^{\si\frac{x^4}{12}+R(\si,x)} -1 = \si\frac{x^4}{12} + R_1(\si,x),
\end{equation*}
where there exists $C_1$ such that 
\begin{equation*}
  |R_1(\si,x)|\le C_1\( \si^2 x^8+\frac{1}{\si} \(x\sqrt\si\)^6\)= C_1
  \si^2 (x^8+x^6),\quad 0\le x\le \frac{1}{\sqrt\si}. 
\end{equation*}
This yields
\begin{align*}
  \int_0^{1/\sqrt\si}
  e^{-x^2/2}\left| e^{\si\frac{x^4}{12}+
  R(\si,x)}-1-\si\frac{x^4}{12}\right|\dd x
  &\lesssim
    \si^2\int_0^\infty e^{-x^2/2}\(1+x^6+x^8\)\dd x \\
  &\lesssim \si^2. 
\end{align*}
We next show that the tail of the integral is
actually much smaller. Changing variables,
\begin{equation*}
\int_{1/\sqrt\si}^\infty \tilde u_\si(x)\dd x=
\int_{1/\sqrt\si}^\infty \frac{\dd x}{(\cosh  
    (x\sqrt\si))^{1/\si}} = \frac{1}{\sqrt\si}\int_{1}^\infty \frac{\dd y}{(\cosh 
    (y))^{1/\si}}. 
\end{equation*}
Taylor formula for
$f(y) = \ln \cosh y $ yields
\begin{equation*}
  f(y)
  =f(1)+(y-1)f'(1)+(y-1)^2\int_0^1(1-\theta)f''(\theta
  y )\dd \theta. 
\end{equation*}
As
\begin{equation*}
  f'(y) = \tanh(y),\quad f''(y)=\frac{1}{\cosh^2 y}\ge 0,
\end{equation*}
we infer
\begin{equation*}
  f(y) \ge f(1)+(y-1)f'(1),
\end{equation*}
hence
\begin{equation*}
  \int_{1}^\infty \frac{\dd y}{(\cosh 
    (y))^{1/\si}}\le    \int_{y_\si}^\infty e^{-\frac{1}{\si}\(
  f(1)+(y-1)f'(1) \)}\dd y=\frac{\si}{f'(1)}
e^{-\frac{1}{\si}  f(1)}, 
\end{equation*}
which is $\O(\si^k)$ for all $k>0$.
Recalling the
asymptotic formula
\begin{equation*}
  \int_M^\infty e^{-x^2/2}\dd x\Eq M \infty \frac{1}{M} e^{-M^2/2}, 
\end{equation*}
we also have
\begin{equation*}
   \int_{1/\sqrt\si}^\infty \(\tilde u_0(x)+\si\tilde
   v_0(x)\)\dd x = \O(\si^k)\quad \text{for all }k>0,
 \end{equation*}
 hence the $L^1$-estimate of the proposition. 
\end{proof}

\subsection{Asymptotic expansions for general $d \geq 1$}

To complete the proof of Theorem~\ref{theo:log},
we describe $u_\si$ up to some $o(\si)$ in $H^s_r$ for $0 \leq s <
1$. The convergence in $\mathcal C^\infty_{\rm loc}$ follows from
rather classical arguments. 

\subsubsection{Derivation}
\label{sec:derivation}

For $z,\si>0$, we denote the nonlinearity in \eqref{eq:u} by
\begin{equation*}
  f(z,\si)=\(z^{2\si}-1\)z,
\end{equation*}
where we note that $f(z,0)=0$. We write an asymptotic expansion 
for $u_\si$ for small $\si > 0$ as
\begin{equation}
\label{eq:asymptotic}
  u_\si = u_0+\si \mu_0+\si e_\si=u_0+\si v_\si,
\end{equation}
where $u_0$ is the Gausson \eqref{Gausson} satisfying \eqref{eq:limiting}, 
$\mu_0$ is the first-order correction \eqref{expr-mu} satisfying \eqref{eq:explicit}, and the remainder term $e_\si$
is expected to vanish as $\si \to 0$ to ensure that $v_\si\to
\mu_0$. Plugging this expression into \eqref{eq:u}, and 
using \eqref{eq:limiting}, we obtain 
\begin{align*}
\si\(v_\si''+\frac{d-1}{r}v_\si'\)
  = -\frac{1}{\si}f\( u_0+\si v_\si,\si\)
  +u_0\ln u_0^2.  
\end{align*}
Consider the decomposition
\begin{equation*}
  f\( u_0+\si v_\si,\si\)= f\( u_0+\si v_\si,\si\)- f(u_0,\si)+f(u_0,\si).
\end{equation*}
Taylor formula yields, since $f(z,0)=0$,
\begin{equation*}
  f(u_0,\si) = \si\d_\si f(u_0,0) + \frac{\si^2}{2}\d^2_{\si\si} f(u_0,0) +
  \frac{\si^3}{2}\int_0^1(1-\theta)^2 \d^3_{\si\si\si}
  f(u_0,\theta\si)\dd\theta.   
\end{equation*}
We readily compute
\begin{equation*}
  \d_\si f(z,0) = z\ln z^2,\quad \d^2_{\si\si} f(z,0)= z\(\ln
  z^2\)^2, \quad\d^3_{\si\si\si} f(z,\si)= z^{2\si+1}\(\ln  z^2\)^3,
\end{equation*}
so
\begin{equation*}
   \frac{1}{\si^2}\(f\( u_0,\si\) -\si u_0\ln u_0^2\)\Big|_{\si=0}=
   \frac{1}{2}u_0 \(\ln
  u_0^2\)^2 =h_0. 
\end{equation*}
Next, we write
\begin{align*}
  f\( u_0+\si v_\si,\si\)-f(u_0,\si)
  &= \si v_\si
    \int_0^1 \d_z f\( u_0 +\theta \si v_\si,\si\)\dd \theta\\
  & = 2\si^2 v_\si  \int_0^1  \(u_0 +\theta \si
    v_\si\)^{2\si\theta}\dd \theta\\
&\quad    +\si v_\si \int_0^1  \(\(u_0 +\theta \si
    v_\si\)^{2\si\theta}-1\)\dd \theta.
\end{align*}
Assuming $v_\si\to \mu_0$ as $\si\to 0$, we get
\begin{equation*}
\frac{1}{\si^2}\(f\( u_0+\si v_\si,\si\)-f(u_0,\si)
\)\Tend \si 0   2\mu_0+\mu_0\ln u_0^2 = (d+2-r^2)\mu_0.
\end{equation*}
Reordering terms, we expect $\mu_0$ to solve $\mathcal L_0
\mu_0=h_0$, hence to be given explicitly by \eqref{expr-mu}.
The correction term $v_\si$ solves
\begin{align*}
  v_\si''+\frac{d-1}{r}v_\si' =
  & -\frac{1}{2}u_0\(\ln u_0^2\)^2 - \frac{\si}{2} \left(\ln u_0^2
    \right)^3 u_0 \int_0^1 (1-\theta)^2 u_0^{2\theta \si} \dd \theta\\
&\quad +2 v_\si  \int_0^1  \(u_0 +\theta \si
    v_\si\)^{2\si\theta}\dd \theta\\
&\quad    + v_\si \int_0^1  \frac{\(u_0 +\theta \si
    v_\si\)^{2\si\theta}-1}{\si}\dd \theta.
 \end{align*}
Recalling that $\si v_\si = u_\si-u_0$, if we denote the potential
\begin{equation}  \label{eq:V_si}
  \begin{aligned}
    V_\si(r)
    &:= - \int_0^1 \frac{\left((1-\theta)u_0(r) + \theta u_\si(r)
      \right)^{2\si} -1}{\si} \dd \theta \\
    &\quad-2  \int_0^1 \left( (1-\theta)u_0(r) + \theta u_\si(r) \right)^{2\si} \dd \theta,
  \end{aligned}
\end{equation}   
associated to the Schrödinger operator
\[ \widetilde{\mathcal{L}}_\si := - \frac{\dd^2}{\dd r^2} - \frac{d-1}{r} \frac{\dd}{\dd r} + V_\si, \]
then the equation on $v_\si$ writes
\begin{equation} \label{eq:v_si}
  \widetilde{\mathcal{L}}_\si v_\si =
   \frac{1}{2}u_0\(\ln u_0^2\)^2+ \frac{\si}{2} \left(\ln u_0^2
  \right)^3 u_0 \int_0^1 (1-\theta)^2 u_0^{2\theta \si}  \dd \theta =: h_\si.
\end{equation}

As we want to show that $e_\si$ vanishes as $\si \to 0$, we need to
invert the Schrödinger operator $\widetilde{\mathcal{L}}_\si$,
considering the right hand side of \eqref{eq:v_si} as a source
term. Unfortunately, such operator could have a zero eigenvalue. However, 
we prove that in the limit $\si \to 0$,
$\widetilde{\mathcal{L}}_\si$ is close in some sense to the shifted
harmonic oscillator $\mathcal{L}_0$, ruling out the aforementioned
scenario.

On a formal level,  not only the error term $e_\si$ in \eqref{eq:asymptotic} is expected to vanish as $\si\to 0$, but also it is likely to satisfy $e_\si=\O(\si)$. However, as can be observed in the case of $\mu_0$, every time a new term is derived in the asymptotic expansion in $\si$ of $u_\si$, it turns out
to be $u_0$ multiplied by a polynomial whose degree increases at every
step. This makes it delicate to prove a quantitative error bound, even
to show that $e_\si=\O(\si)$ in $L^2(\R^d)$ for $d \geq 2$. Also, to prove
$e_\si=\O(\si)$, we would have to expand $V_\si$ in powers of $\si$,
which would involve $u_0 \ln\(  (1-\theta)u_0 + \theta
u_\si\)$. Controlling this term in $L^2_r$ essentially requires to
know some uniform bound from below for $u_\si$, which we could not
derive. Therefore, we rely on the study of invertibility of the Schr\"{o}dinger operator $\widetilde{\mathcal{L}}_\si$.

\subsubsection{Spectrum of the radial shifted harmonic oscillator} 
Recall that the harmonic oscillator $H=-\Delta + |x|^2$ on $\R^d$ has its eigenvalues $(\Omega_n)_{n \in \N}$ and eigenfunctions $(f_n)_{n \in \N}$ satisfying
\[ \left| \begin{aligned}
& \Omega_n=(\omega_{n_1}+\ldots+\omega_{n_d}),  \\
& f_n = \psi_{n_1} \ldots \psi_{n_d}, \\
& n_1+\ldots+n_d=n,
\end{aligned} \right. \]
where $\omega_k=2k+1$ and $\psi_k$ denotes the $k$-th Hermite
function. Note that $(\omega_k)_{k \in \N}$ and $(\psi_k)_{k\in\N}$
are respectively eigenvalues and eigenfunctions of the one-dimensional
harmonic oscillator. For $k$ even (resp. odd), $\psi_k$ is even
(resp. odd).

Restricted on radial functions, the operator
\[ H_{\mathrm{rad}} = -\frac{ \dd^2 }{\dd r^2} - \frac{d-1}{r} \frac{\dd }{\dd r} + r^2  \]
then admits  eigenfunctions $(g_n)_{n\in\N}$ such that
\[ g_n=\psi_{n_1}\ldots \psi_{n_d}, \quad n_1=n_2=\ldots=n_d\text{ even},  \]
with sorted eigenvalues $\Lambda_n=\omega_{n_1}+\ldots+\omega_{n_d}$. In particular for $n=0$, $g_0$ denotes a radial Gaussian function associated to first eigenvalues $\Lambda_0=d$, while the second eigenfunction $\Lambda_1=5d$ corresponds to $n_1=\ldots=n_d=2$.

Recalling that radial shifted harmonic oscillator writes
$\mathcal{L}_0= H_{\mathrm{rad}}-d-2$ from \eqref{eq:L0}, and denoting
by $(\lambda_0^{(k)})_{k\in \N}$ and $(\varphi_0^{(k)})_{k\in\N}$ its sorted
eigenvalues and eigenfunctions, we thus infer that $\lambda_0^{(0)}=-2$ and
that $\lambda_0^{(1)}=4d-2$, so that $\lambda_0^{(k)}\ge 2$ for all $k\ge 1$ and all $d \geq 1$.

\subsubsection{Properties of the Schr\"{o}dinger operator $\widetilde{\mathcal{L}}_\si$}

\begin{lemma}
 Let $\si>0$. The potential $V_\si$, defined in \eqref{eq:V_si}, is
 radially symmetric, non-decreasing, and
 \begin{equation*}
   \lim_{r\to \infty}V_\si(r) =\frac{1}{\si}. 
 \end{equation*}
 Moreover, there exists $K>0$ such that for all $\si\in (0,2/d]$, $V_\si\ge
 -K$. As a consequence, $\widetilde{\mathcal{L}}_\si$ is a
 self-adjoint accretive operator such that
 $\si_c(\widetilde{\mathcal{L}}_\si)=[1/\si,\infty)$ and
 $\si_p(\widetilde{\mathcal{L}}_\si) \subset[-K,1/\si]$. 
\end{lemma}
In the above statement, the upper bound $\si\le 2/d$ is arbitrary, to avoid
to distinguish the case $d\le 2$ (where $\si$ has no upper bound otherwise) from
the general case. 
\begin{proof}
  As recalled in the beginning of Section~\ref{sec:log}, we know that
  $u_\si\to u_0$ in $L^\infty(\R^d)$, so there exists $C_\infty$ such
  that
  \begin{equation*}
    \|u_\si\|_{L^\infty}\le C_\infty,\quad \forall \si\in [0,2/d].
  \end{equation*}
  We infer
  \begin{equation*}
    \((1-\theta)u_0+\theta u_\si\)^{2\si}\le C_\infty^{2\si},\quad
    \forall \theta\in [0,1],
  \end{equation*}
  so
  \begin{equation*}
    V_\si\ge \frac{1-C_\infty^{2\si}}{\si}-2C_\infty^{2\si}\Tend \si 0
    -\ln C_\infty^2-2,
  \end{equation*}
  hence $V_\si\ge -K$ for some uniform $K>0$. The rest of the lemma
  follows easily. 
\end{proof}
Thus
there exists a set of sorted eigenvalues $\lambda_\si^{(0)} \le
\lambda_\si^{(1)} \le \ldots$, $(\lambda_\si^{(j)})_{j\in J}$, and
eigenvectors $(\varphi_\si^{(j)})_{j\in J}$, $J \subset \N$, such that 
\begin{equation} \label{eq:eigen_L_si}
(\widetilde{\mathcal{L}}_\si+ K) \varphi_\si^{(j)} = (\lambda_\si^{(j)} + K) \varphi_\si^{(j)},
\end{equation} 
with $\lambda_\si^{(j)} \in [-K,1/\si]$, and $\| \varphi_\si^{(j)}
\|_{L^2_r}=1$ for all $j \in J$.
\smallbreak

Our goal now is to prove that for
$\si>0$ sufficiently small, the point spectrum of
$\widetilde{\mathcal{L}}_\si$ is uniformly away from zero. We shall
argue by comparison with the limiting case  of the
shifted harmonic operator   $\mathcal L_0$, which will be made possible
thanks to compactness properties. 

\begin{lemma}\label{lem2}
  For all $\eps>0$, there exist $\si_0>0$ and $R>0$ such that if
  $0<\si\le\si_0$ and $r\ge R$,
  \begin{equation*}
    V_\si(r)+K\ge \frac{1}{\eps}. 
  \end{equation*}
\end{lemma}
\begin{proof}
  Let $\delta>0$ to be determined later. For any $z\le \delta$,
  $\frac{1-z^{2\si}}{\si}\ge \frac{1-\delta^{2\si}}{\si}$.

   On the other hand, for any $r\ge 0$ and $\theta\in [0,1]$,
   \begin{equation*}
    0< (1-\theta)u_0(r) + \theta u_\si(r) = u_0(r)
    +\theta\(u_\si(r)-u_0(r)\)\le u_0(r) + \|u_\si-u_0\|_{L^\infty}. 
  \end{equation*}
  Let $R>0$ such that for all $r\ge R$, $u_0(r)<\delta/2$, and let
  $\si_0>0$ such that for all $\si\le\si_0$,
  $\|u_\si-u_0\|_{L^\infty}<\delta/2$. For $r\ge R$ and $\si\le\si_0$,
  \begin{equation*}
    \int_0^1 \frac{1-\((1-\theta)u_0(r)+\theta u_\si(r)\)^{2\si}}{\si}\dd
    \theta \ge \frac{1-\delta^{2\si}}{\si}.
  \end{equation*}
  To control the other term defining $V_\si$, note that for $\si\le
  \si_0$ and $r\ge 0$,
  \begin{equation*}
    \int_0^1 \((1-\theta)u_0(r)+\theta u_\si(r)\)^{2\si}\le
    \(\|u_0\|_{L^\infty} +\frac{\delta}{2}\)^{2\si},
  \end{equation*}
  so we come up with
  \begin{equation*}
    V_\si(r)+ K \ge \frac{1-\delta^{2\si}}{\si}- 2\(\|u_0\|_{L^\infty}
    +\frac{\delta}{2}\)^{2\si}+K ,\quad \forall r\ge R,\ \forall
    \si\le \si_0. 
  \end{equation*}
  The right hand side goes to $K-\ln \delta^2$ as $\si$ goes to
  zero. Up to decreasing $\si_0>0$, we have
   \begin{equation*}
    V_\si(r)+ K \ge K-\frac{1}{2}\ln\delta^2 ,\quad \forall r\ge R,\ \forall
    \si\le \si_0,
  \end{equation*}
  and we conclude by picking $\delta>0$ such that
  $\eps^{-1}=K-\frac{1}{2}\ln\delta^2$. 
\end{proof}
We infer that weighted $L^2_r$ estimates involving $V_\si$ provide
compactness in $L^2_r$ of bounded family of $H^1_r$ functions:
\begin{lemma}\label{lem3}
  Let $(g_\si)_{\si>0}$ be a family in $H^1_r$ such that there
  exist $\si_1>0$ and $C>0$ with
  \begin{equation*}
    \|g_\si\|_{H^1_r}^2 +
    \int_0^\infty\(V_\si(r)+K\)g_\si(r)^2r^{d-1}\dd r \le C,\quad
    \forall \si\in (0,\si_1).
  \end{equation*}
  Then the family $(g_\si)_{\si>0}$ is relatively compact in $L^2_r$
  as $\si\to 0$. 
\end{lemma}
\begin{proof}
  We show that the Fr\'echet-Kolmogorov Theorem for radially symmetric
  functions can be applied, by proving the equitightness property,
  \begin{equation*}
    \lim_{R\to \infty}\limsup_{\si\to 0} \int_R^\infty
    g_\si(r)^2r^{d-1}\dd r =0. 
  \end{equation*}
  Let $\eps>0$, and consider $R,\si_0$ provided by Lemma~\ref{lem2}:
  for $\si\le \si_0$,
  \begin{equation*}
    \int_R^\infty   g_\si(r)^2r^{d-1}\dd r = \int_R^\infty
    \frac{V_\si(r)+K}{V_\si(r)+K}  g_\si(r)^2r^{d-1}\dd r \le C\eps,
  \end{equation*}
  hence the lemma. 
\end{proof}

\begin{lemma}\label{lem4}
 Let $(\varphi_\si)_{\si>0}$ be a sequence of eigenfunctions of
 $\widetilde{\mathcal{L}}_\si$, normalized in $L^2_r$, such that the
 related eigenvalue $\lambda_\si$ satisfies $-4\le \lambda_\si\le
 4$. Then there exists a subsequence $\si_n\to 0$ as $n\to \infty$
 such that $\varphi_{\si_n}\to \varphi_0$ in $L^2_r$ and
 $\lambda_{\si_n}\to \lambda$ for some $\varphi_0\in H^1_r$
 and $\lambda\in [-4,4]$. Moreover, $\varphi_0$ is an eigenfunction of
 $\mathcal L_0=\Delta_r+r^2-d-2$, normalized in $L^2_r$, with related
 eigenvalue $\lambda$. 
\end{lemma}
\begin{proof}
  We compute
  \begin{equation*}
    \|\varphi_\si'\|_{L^2_r}^2 +
    \int_0^\infty V_\si(r)\varphi_\si(r)^2r^{d-1}\dd r = \<
    \widetilde{\mathcal{L}}_\si \varphi_\si, \varphi_\si\> =
    \lambda_\si\< \varphi_\si,\varphi_\si\>=\lambda_\si.
  \end{equation*}
  Thus,
  \begin{equation*}
     \|\varphi_\si\|_{H^1_r}^2 +
    \int_0^\infty \(V_\si(r)+K\)\varphi_\si(r)^2r^{d-1}\dd r =
    1+\lambda_\si+K\le 5+K,
  \end{equation*}
  and we can invoke Lemma~\ref{lem3}. Up to a subsequence,
  $\varphi_{\si_n}\to \varphi_0$ in $L^2_r$ and 
 $\lambda_{\si_n}\to \lambda$ for some $\varphi_0\in H^1_r$
 and $\lambda\in [-4,4]$. It remains to show that $\mathcal L_0\varphi
 = \lambda\varphi$.

 We readily check the pointwise convergence $V_\si(r)\to V_0(r) := -\ln
 u_0(r)^2-2=r^2-d-2$ as $\si\to 0$. We thus have the convergences
 \begin{align*}
 &  \Delta \varphi_\si\Tend \si 0 \Delta \varphi\quad \text{in
   }H^{-2},\\
 & V_\si \varphi_\si  \Tend \si 0 V_0\varphi \quad \text{in }L^2_{\rm
   loc},\\
 & \lambda_\si \varphi_\si  \Tend \si 0 \lambda\varphi \quad \text{in }L^2. 
 \end{align*}
 Passing to the limit in the equation
 $\widetilde{\mathcal{L}}_\si\varphi_\si \equiv -\Delta
 \varphi_\si+V_\si\varphi_\si = \lambda_\si\varphi_\si$, we obtain
 \begin{equation*}
   -\Delta \varphi + V_0\varphi = \lambda\varphi\quad\text{in
   }H^{-1}_{\rm loc},
 \end{equation*}
 and $-\Delta +V_0=\mathcal L_0$, hence the lemma. 
\end{proof}
We infer the announced result:
\begin{proposition}\label{prop:spectre}
  There exists $\si_0>0$ such that for all $0<\si\le \si_0$,
  $\si_p(\widetilde{\mathcal{L}}_\si)\cap [-1,1]=\emptyset$. 
\end{proposition}
\begin{proof}
  We may assume that along some sequence $\si_n\to 0$,
  $\lambda_{\si_n}^{(0)}\le 2$ (the lowest eigenvalue of
  $\widetilde{\mathcal{L}}_{\si_n}$), for otherwise the result is
  straightforward. Up to a subsequence, $\lambda_{\si_n}^{(0)}\to
  \lambda_0$, and $\varphi_{\si_n}$ converges in $L^2_r$ to some
  normalized eigenfunction $\varphi_0$ of $\mathcal L_0$, associated
  to $\lambda_0$. Moreover, $\varphi_0$ is radially symmetric and
  nonincreasing (not an excited state), so necessarily
  $\varphi_0=\varphi_0^{(0)}$ and $\lambda_0=\lambda_0^{(0)}=-2$. As
  the limit is unique, no subsequence is needed.

  Consider now the second eigenvalue $\lambda_{\si}^{(1)}$, and
  suppose
  \begin{equation*}
    \nu_0=\liminf_{\si\to 0}\lambda_{\si}^{(1)}=\lim_{n\to \infty}
    \lambda_{\si_n}^{(1)}\le 4, 
  \end{equation*}
  for some sequence $\si_n\to 0$.
  Note that if $\nu_0>2$, the proposition is proven. Let
  $\varphi_{\si_n}^{(1)}$ be a normalized eigenfunction associated to
  $\lambda_{\si_n}^{(1)}$. Up to a subsequence,
  $\varphi_{\si_n}^{(1)}$ converges to an eigenfunction $\varphi_0$
  of $\mathcal L_0$, with eigenvalue $\lambda_0$, from
  Lemma~\ref{lem4}. In addition,
  \begin{equation*}
    0= \<\varphi_{\si_n}^{(0)},\varphi_{\si_n}^{(1)}\>\Tend n \infty
    \< \varphi_{0}^{(0)},\varphi_{0}\>.
  \end{equation*}
  Therefore, $\lambda_0>\lambda_0^{(0)}$, and $\lambda_0\ge
  \lambda_0^{(1)}=4d-2\ge 2$, hence the result. 
\end{proof}

\subsubsection{Convergence}

\begin{lemma}\label{lem6}
  There exists $C>0$ such that the following holds. Let $\si_0$ given
  by Proposition~\ref{prop:spectre}. For any $\si\le
  \si_0$, for any $\psi\in H^1_r$ such that
  $\widetilde{\mathcal{L}}_{\si}\psi=g\in L^2_r$,
  \begin{equation*}
    \|\psi\|_{H^1_r}^2 + \int_0^\infty  \(V_\si(r)+K\)\psi
    (r)^2r^{d-1}\dd r\le C\|g\|^2_{L^2_r},\quad \forall \si\le \si_0.
  \end{equation*}
\end{lemma}
\begin{proof}
  Let $U_\si$ be unitary operators on $L^2_r$ and $h_\si$ such
  that $\widetilde{\mathcal{L}}_{\si}= U_\si^{-1} h_\si (\rho)U_\si$ by the
  spectral theorem. We know that $h_\si(\rho)$ belongs to the
  spectrum of 
  $\widetilde{\mathcal{L}}_{\si}$ for almost all $\rho>0$, and so
  $|h_\si(\rho)|\ge 1$ for almost all $\rho>0$, as soon as $\si\le \si_0$
  from Proposition~\ref{prop:spectre}.
  \smallbreak

  Writing $g = U_\si^{-1} h_\si (\rho)U_\si\psi$, we have $\psi = U_\si^{-1}
  \frac{1}{h_\si (\rho)}U_\si g$, and thus
  \begin{equation*}
    \|\psi\|_{L^2_r}= \left\|U_\si^{-1}
  \frac{1}{h_\si }U_\si g\right\|_{L^2_r} = \left\|
  \frac{1}{h_\si }U_\si g\right\|_{L^2_r} \le \left\|
  U_\si g\right\|_{L^2_r}=\left\|   g\right\|_{L^2_r}.
\end{equation*}
For the remaining part to estimate,
\begin{align*}
   \|\psi'\|_{L^2_r}^2 + \int_0^\infty
  \(V_\si(r)+K\)\psi(r)^2r^{d-1}\dd r
  & = \< (\widetilde{\mathcal{L}}_{\si} +K)\psi,\psi\>\\
  &  = \< U_\si^{-1} \(h_\si +K\)U_\si\psi,\psi\>\\
  &= \< U_\si^{-1} \(h_\si +K\) \frac{1}{h_\si }U_\si g,U_\si^{-1}
    \frac{1}{h_\si }U_\si g\>\\
  & =\int_0^\infty \frac{h_\si(\rho) +K}{h_\si(\rho)^2} \(U_\si
    g (\rho)\)^2\rho^{d-1}\dd \rho.
\end{align*}
Since the map $z\mapsto \frac{z+K}{z^2}$ is bounded in $\R\setminus
[-1,1]$, we infer that there exists $C>0$ such that
\begin{equation*}
  \int_0^\infty \frac{h_\si(\rho) +K}{h_\si(\rho)^2} \(U_\si
    g (\rho)\)^2\rho^{d-1}\dd \rho\le C \int_0^\infty \(U_\si
    g (\rho)\)^2\rho^{d-1}\dd \rho= C \|g\|_{L^2_r}^2,
\end{equation*}
hence the result. 
\end{proof}

We can now prove the end of Theorem~\ref{theo:log}:
\begin{corollary}
  The family $(v_\si)_\si$ is bounded in $H^1_r$, and
  converges strongly in $L^2_r$ to $\mu_0$ given by
  \eqref{expr-mu}. By interpolation, the convergence holds in $H^s_r$ for all $0 \leq s<1$. 
\end{corollary}

\begin{proof}
  Denote by $h_\si$ the right hand side of \eqref{eq:v_si}. It is easy
  to check that
  \begin{equation*}
    h_\si\Tend \si 0 h_0=\frac{1}{2}u_0\(\ln u_0^2\)^2 \quad \text{in
    }L^2_r. 
  \end{equation*}
 Lemma~\ref{lem6} implies that $(v_\si)_\si$ is bounded in $H^1_r$, and together with  Lemma~\ref{lem3}, we infer that up to a
 subsequence, $v_\si$ converges strongly in $L^2_r$, to some $v\in
 H^1_r$. Passing to the limit in \eqref{eq:v_si}, which we rewrite as
 \begin{equation*}
   -\Delta v_\si + V_\si v_\si = h_\si,
 \end{equation*}
 and arguing like in the proof of Lemma~\ref{lem4},
  we come up with
 \begin{equation*}
   -\Delta v + V_0 v = h,\quad \text{that is}\quad \mathcal L_0v = h_0.
 \end{equation*}
 We infer that $v=\mu_0$, and by uniqueness of the limit, the whole
 sequence $(v_\si)_\si$ is converging to $\mu_0$ in $L^2_r$ as $\si\to
 0$, hence the result. 
\end{proof}

\begin{lemma}
	The family $(v_\si)_\si$ converges to $\mu_0$ in $\mathcal{C}^\infty_{\textnormal{loc}}$.
\end{lemma}

\begin{proof}
	We recall that $v_\si$ satisfies \eqref{eq:v_si}. Moreover, from the convergence of $u_\si$ to $u_0$ in $\mathcal{C}^\infty_{\textnormal{loc}}$ and using the fact that $u_0$ is strictly positive on every bounded set (and thus far from $0$), we deduce that $V_\si$ and $h_\si$ are bounded in $\mathcal{C}^\infty (K)$ uniformly in $\si$ for every compact set $K$.
	Thus, using regularity theory for elliptic equations (see for instance \cite[Theorems~9.11~\& ~9.19]{GilbargTrudinger}) and bootstraping (with the first step using that $v_\si$ is unformly bounded in $H^1 (\mathbb{R}^d)$), we deduce that $v_\si$ is also uniformly bounded in every $W^{m, p} (K)$ for any compact set $K$. This leads to the conclusion by Sobolev embeddings.
\end{proof}


\section{The limit $\sigma \to \sigma_*$: convergence to the algebraic soliton}
\label{sec-alg-limit}

\subsection{Some properties of the ground state}
\label{sec:extra}

As pointed out in Section~\ref{sec-main_results}, the parameter
$\epsilon(\si)$ in \eqref{eq:ivp-w} is implicit and is defined from the condition 
that $w_{\si}(\rho) = w(\rho;\epsilon(\si),\si)$ is positive and monotonically decreasing with the fast (exponential) decay condition as $\rho \to \infty$. Here,
we derive some estimates involving $\epsilon$ and $w_\si$, thanks to
Pohozaev identitites.

Multiplying \eqref{eq:ivp-w} by $\rho^{d-1} w_{\sigma}(\rho)$ and
integrating on $(0,\infty)$ by parts with
\begin{equation*}
  \rho^{d-1} w_{\sigma}(\rho) w_{\sigma}'(\rho) |_{\rho = 0}^{\rho \to \infty} = 0,
\end{equation*}
we obtain 
\begin{equation}
\label{Poh-1}
\| w_{\sigma} \|^{2\si + 2}_{L^{2\si+2}_r} = \epsilon(\sigma) \|
w_{\sigma} \|^2_{L^2_r} + \|  w_{\sigma}' \|^2_{L^2_r}. 
\end{equation}
Multiplying \eqref{eq:ivp-w} by $\rho^{d} w'_{\sigma}(\rho)$ and
integrating by parts on $(0,\infty)$,  with
\begin{equation*}
  \rho^{d} (w_{\sigma}'(\rho))^2 |_{\rho = 0}^{\rho \to \infty} = 
\rho^{d} w_{\sigma}^{2}(\rho) |_{\rho = 0}^{\rho \to \infty} = 0,
\end{equation*}
we get
\begin{equation}
\label{Poh-2}
\frac{d}{1 + \sigma} \| w_{\sigma} \|^{2\si + 2}_{L^{2\si+2}_r}
= d \epsilon(\sigma) \| w_{\sigma} \|^2_{L^2_r} + (d-2) 
\|  w_{\sigma}' \|^2_{L^2_r}. 
\end{equation}
In what follows, we can express $d \geq 3$ by using $\si_*$ due to $d = 2 + \frac{2}{\sigma_*}$.

Eliminating $\| w_{\sigma} \|^{2\si + 2}_{L^{2\si + 2}_r}$ from
\eqref{Poh-1}  and \eqref{Poh-2} yields
\begin{equation}\label{eps-expression}
\epsilon(\sigma) = \frac{(\sigma_* - \sigma)  \|w_{\sigma}'
  \|^2_{L^2_r}}{\sigma (1 + \sigma_*) \| w_{\sigma} \|^2_{L^2_r}}.
\end{equation}
On the other hand, eliminating $\| w_{\sigma}' \|^{2}_{L^{2}}$ from
\eqref{Poh-1}  and \eqref{Poh-2} yields
\begin{equation}
\label{alter-expr}
\| w_{\sigma} \|^{2 \si+ 2}_{L^{2\si + 2}_r} =
  \frac{\si_* (1+\si) \epsilon(\si)}{(\si_*-\si)} \| w_{\sigma} \|^2_{L^2_r} ,
\end{equation}
and since $\rho \mapsto w_\si(\rho)$ is nonincreasing on $(0,\infty)$,
interpolation yields
\begin{equation*}
  \| w_{\sigma} \|^{2 \si+ 2}_{L^{2\si + 2}_r}\le
  \|w_\si\|_{L^\infty_r}^{2\si} \|w_\si\|_{L^2_r}^2= \|w_\si\|_{L^2_r}^2.
\end{equation*}
We infer from \eqref{alter-expr} that 
\begin{equation}\label{eq:eps-sublin}
  0 < \epsilon(\si) \le \frac{(\si_*-\si)}{\si_*(1+\si)}, 
\end{equation}
hence $\epsilon(\si)=\O(\si_*-\si)$. Comparison \eqref{eps-expression} with \eqref{eq:eps-sublin} implies
that the ratio $\| w_{\sigma}' \|^2_{L^2_r}/\| w_{\sigma}
\|^2_{L^2_r} $ is uniformly bounded. Therefore, the  quantity
\begin{equation*}
   -\frac{\epsilon(\si)}{\si_*-\si}=
   \frac{\epsilon(\si_*)-\epsilon(\si)}{\si_*-\si} 
 \end{equation*}
 is bounded, and has converging subsequences.  We shall prove
 that $ \epsilon(\si)\sim c(d) (\si_*-\si)$ as $\si\to \si_*$,  for
 some explicit $c(d) > 0$, and no subsequence is needed.

\subsection{Convergence in $L_r^\infty\cap W_{\mathrm{loc}}^{1,\infty}$}

The first convergence result announced in Theorem~\ref{theo:algebraic}
is a direct consequence of the following property. 

\begin{proposition}\label{prop:unif-cv}
  Let $d\ge 3$ and $\si\in (0,\si_*)$. For $w_\si$ the solution to
  \eqref{eq:ivp-w}, with $\epsilon=\O(\si_*-\si)$, and $w_*$ given by
  \eqref{eq:alg-soliton}, there exists $C_0>0$ independent of $\si\in
  (0,\si_*)$ such that for every $R>0$,
  \begin{equation}\label{eq:Ehrenfest}
    \sup_{0\le \rho\le R}|w_\si(\rho)-w_*(\rho)|+
   \sup_{0\le \rho\le R}|w_\si'(\rho)-w_*'(\rho)| \le C_0 (\si_*-\si)e^{C_0R}.
  \end{equation}
  In addition, there exists $C>0$ independent of $\si\in
  (0,\si_*)$ such that
  \begin{equation}
  \label{L-infty-bound}
    \|w-w_*\|_{L^\infty_r}=\sup_{\rho\ge 0}|w_\si(\rho)-w_*(\rho)|\le C \(\ln
  \frac{1}{\si_*-\si}\)^{-2/\si_*}.
  \end{equation}
\end{proposition}

\begin{proof}
Let $\varphi_\si(\rho)= w_*(\rho)-w_\si(\rho)$ denote the error, and consider
$q_\si=\varphi_\si^2 + 
(\varphi_\si')^2$. We  compute
\begin{equation*}
  q_\si'(\rho) = 2\varphi_\si\varphi_\si' + 2\varphi_\si'\varphi_\si'',
\end{equation*}
so using \eqref{eq:ivp-w}
(and its limiting case $\si=\si_*$ for which
$\epsilon(\si_*)=0$),
\begin{equation*}
  q_\si'(\rho) = 2\varphi_\si\varphi_\si' + 2\varphi_\si'\(
  -\frac{d-1}{\rho}\varphi_\si'  
  -w_*^{2\si_*+1}+w_\si^{2\si+1}-\epsilon(\si)w_\si\). 
\end{equation*}
Young inequality implies $2\varphi_\si\varphi_\si'\le q_\si$. Recall
that we have the uniform estimates
\begin{equation*}
 0< w_\si(\rho),w_*(\rho)\le 1, 
\end{equation*}
and $\epsilon=\O(\si_*-\si)$. We decompose
\begin{equation*}
  w_*^{2\si_*+1}-w_\si^{2\si+1}=
  \underbrace{w_*^{2\si_*+1}-w_*^{2\si+1}}_{=:S_\si}
  +\underbrace{w_*^{2\si+1}-w_\si^{2\si+1}}_{=:G_\si}.
\end{equation*}
Taylor formula yields
\begin{equation*}
  G_\si =(2\si+1) \varphi_\si\int_0^1
  \(w_\si+\theta(w_*-w_\si)\)^{2\si}\dd\theta, 
\end{equation*}
and by the above uniform bounds,
\begin{equation*}
  |G_\si|\le (2\si+1)|\varphi_\si|.
\end{equation*}
Invoking Young inequality again, we infer, for some $C>0$ independent
of $\si\le \si_*$ and $\rho\ge 0$,
\begin{equation*}
  q_\si'(\rho)\le C q_\si(\rho) + S_\si(\rho)^2 +  C(\si_*-\si)^2 .
\end{equation*}
For the source term $S_\si$, Taylor formula for the map $z\mapsto
w_*^z$ yields
\begin{equation*}
  S_\si = w_*^{2\si+1}\(w_*^{2(\si_*-\si)}-1\) =
  2(\si_*-\si)w_*^{2\si+1}\ln w_*\int_0^1w_*^{2\theta(\si_*-\si)}\dd\theta,
\end{equation*}
hence, since $0<w_*\le 1$,
\begin{equation*}
  |S_\si(\rho)|\le 2(\si_*-\si)w_*(\rho)^{2\si+1}|\ln w_*(\rho)|=
  \O\(\si_*-\si\). 
\end{equation*}
As $q_\si(0)=0$, Gr\"onwall lemma implies that there exists $C_0,C_1>0$
independent of $\si\in [\si_*/2,\si_*]$ such
that for every $R>0$,
\begin{equation*}
  \sup_{0\le \rho\le R}q_\si(\rho)\le C_1(\si_*-\si)^2e^{2C_0R},
\end{equation*}
hence \eqref{eq:Ehrenfest}.
\smallbreak

Pick $R_\si = \frac{1}{2C_0} \ln \frac{1}{\si_*-\si}$. We have 
\begin{align*}
  |w_\si(R_\si)|\le |w_*(R_\si)|+ |w_\si(R_\si)-w_*(R_\si)|
  & \lesssim R_\si^{-2/\si_*} + (\si_*-\si) e^{C_0R_\si}\\
  &\lesssim \(\ln  \frac{1}{\si_*-\si}\)^{-2/\si_*}, 
\end{align*}
where we have used the explicit decay for $w_*$. 
Since $w_\si$ and $w_*$ are positive decreasing, for $\rho\ge R_\si$, we
have
\begin{equation*}
  |w_\si(\rho)-w_*(\rho)|\le w_\si(\rho)+ w_*(\rho)\le w_\si(R_\si)+
  w_*(R_\si)\lesssim \(\ln
  \frac{1}{\si_*-\si}\)^{-2/\si_*}. 
\end{equation*}
On the other hand, \eqref{eq:Ehrenfest} yields
\begin{equation*}
  \sup_{0\le \rho\le R_\si} |w_\si(\rho)-w_*(\rho)|\le C_0 (\si_*-\si)
  e^{C_0R_\si} \lesssim \sqrt{\si_*-\si}.
\end{equation*}
Combining the two bounds together yields \eqref{L-infty-bound}.
\end{proof}

\subsection{Convergence in $H^1_r$}

We now turn to the proof of the second convergence result of
Theorem~\ref{theo:algebraic}.  Let $d \ge 5$. 
We consider the minimizer $v_\si$ of the problem
\eqref{eq:variation}, and we denote 
\begin{equation*}
  \K_\si := \inf_{v\in H^1_r} \enstq{\|v'\|_{L^2_r}^2 +
 \frac{1}{\sigma} \| v\|_{L^2_r}^2}{ \ \|v\|_{L^{2\si+2}_r}^{2\si+2}=1},
\end{equation*}
so that $v_\si$ is solution of the equation 
\begin{equation} \label{eq:u_sigma}
v_{\si}'' + \frac{d-1}{r} v_{\si}'+ \K_\si v_{\si}^{1+2\si} = \frac{1}{\si} v_\si. 
\end{equation} 

\begin{lemma} \label{lem:lambda_sigma_bound}
There exists $C_1$, $C_2>0$ and $\eps>0$ such that for all $\si \in
\left[ \si_*-\eps/2,\si_*\right]$, we have $C_1 \le \K_\si \le C_2$.
\end{lemma}
\begin{remark}
  This lemma is reminiscent of continuity properties of the best
  constant in Gagliardo-Nirenberg inequalities,
  \begin{equation*}
    S_{\si,d}^{1/2} \|f\|_{L^{2\si+2}(\R^d)}\le
    \|f\|_{L^2(\R^d)}^{1-\theta}\|\nabla f\|_{L^2(\R^d)}^\theta,\quad
    \theta = \frac{d\si}{2\si+2} = \frac{\si (1+\si_*)}{\si_* (1+\si)}.
  \end{equation*}
  It follows from \cite[Section~2]{FGL21} (see also \cite[Lemma~2.50]{FLW23})
  that $\si\mapsto S_{\si,d}$ is continuous on $(0,\si_*]$. Note
  however that as $\si\to \si_*$, $\theta\to 1$, and the property
  $w_*\in L^2_r$ requires $d\ge 5$. Therefore, the bound from above in 
  Lemma~\ref{lem:lambda_sigma_bound} is not an immediate consequence
  of that continuity, this is why we prove it. 
\end{remark}

\begin{proof}
  By interpolation, for $v\in H^1_r$,
  \begin{equation*}
     \| v\|_{L^{2\si+2}_r}\le \|v\|_{L^2_r}^{1-\theta}
     \|v\|_{L^{2\si_*+2}_r}^\theta,\quad \theta =
     \frac{\si (1+\si_*)}{\si_* (1+\si)} \in (0,1),
   \end{equation*}
   and \eqref{eq:Sob} yields
   \begin{equation*}
     \|v\|_{L^{2\si_*+2}_r}^2\le \mathcal{S}^{-1}\|v'\|^2_{L^2_r} \le
     \mathcal{S}^{-1}\(\|v'\|^2_{L^2_r}+\frac{1}{\si} \|v\|_{L^2_r}^2\), 
   \end{equation*}
hence the bound from below. 

We now turn to the bound from above. Let $\lambda>0$ to be fixed, and denote
\begin{equation*}
  w_{*,\lambda}(r)=\lambda^{\frac{1}{\si_*}} \frac{w_*(\lambda r)}{\|
    w_*\|_{L^{2\si_*+2}_r}}
\end{equation*}
where $w_*$ denotes the Aubin-Talenti algebraic soliton
\eqref{eq:alg-soliton}. Then we compute 
\begin{equation*}
  \| w_{*,\lambda} \|_{L^{2\si+2}_r}^{2\si+2} =
  \lambda^{\frac{2\si+2}{\si_*}-d} \left( \frac{\| w_*
      \|_{L^{2\si+2}_r}}{\| w_* \|_{L^{2\si_*+2}_r}}
  \right)^{2\si_*+2}  = \lambda^{\frac{2 (\si-\si_*)}{\si_*}} \left(
    \frac{\| w_* \|_{L^{2\si+2}_r}}{\| w_* \|_{L^{2\si_*+2}_r}}
  \right)^{2\si_*+2} ,
\end{equation*}
thus we take
\begin{equation*}
  \lambda  = \left( \frac{\| w_* \|_{L^{2\si+2}_r}}{\| w_*
      \|_{L^{2\si_*+2}_r}} \right)^{\frac{(\si_*+1)\si_*}{\si_*-\si}},
\end{equation*}
so that $\|w_{*,\lambda} \|_{L^{2\si+2}_r}=1$. Moreover, we write
\begin{align*}
\| w_* \|_{L^{2\si+2}_r}^{2\si+2} - \| w_*
  \|_{L^{2\si_*+2}_r}^{2\si_*+2}
  &= \int_0^{\infty} \rho^{d-1} \left( w_*(\rho)^{2\si+2} -
    w_*(\rho)^{2\si_*+2}\right) \dd \rho \\ 
& = \int_0^{\infty} \rho^{d-1} w_*(\rho)^{2\si+2} \left(  1 -
  w_*(\rho)^{2(\si_*-\si)}\right) \dd \rho, 
\end{align*}
and we remark from Taylor expansion that
\begin{equation*}
  w_*(\rho)^{2(\si_*-\si)}-1 =2(\si_*-\si)
  \underbrace{\log w_*(\rho)}_{\le 0} \int_0^1
  \underbrace{w_*(\rho)^{2\theta(\si_* - \si)}}_{\in \left[0,1\right]}
  \dd \theta, 
\end{equation*}
hence
\begin{equation*}
  0 \geq w_*(\rho)^{2(\si_*-\si)} -1 \geq 2(\si_* - \si) \log w_*(\rho).
\end{equation*}
This enables to write, since $0<w_*(\rho)\le 1$,
\begin{align*}
0 \le \| w_* \|_{L^{2\si+2}_r}^{2\si+2} - \| w_*
  \|_{L^{2\si_*+2}_r}^{2\si_*+2}
  & \le 2(\si_* - \si) \int_0^{\infty} \rho^{d-1} | \log w_*(\rho)|
    w_*(\rho)^{2\si+2} \dd \rho \\ 
& \le K_1 (\si_* -\si) \| w_* \|_{L^{2\si+2-\eta}_r}^{2\si+2-\eta} \le
  K_2 (\si_* - \si), 
\end{align*}
for fixed $\eta>0$, $K_1=K_1(\eta)$, $K_2>0$ uniform in
$\si$. Moreover we have 
\begin{equation*}
  \| w_* \|_{L^{2\si+2}_r}^{2\si_*+2} - \| w_*
  \|_{L^{2\si+2}_r}^{2\si+2} = 2 \| w_* \|_{L^{2\si+2}_r}^{2\si+2}
  (\si_*-\si) \log \| w_*\|_{L^{2\si+2}_r} \int_0^1   \|
  w_*\|_{L^{2\si+2}_r}^{2\theta(\si_*-\si)} \dd \theta,
\end{equation*}
with $\| w_* \|_{L^{2\si+2}_r}$ both bounded from above and from below
by a positive constant, so  
\begin{equation*}
  \left| \| w_* \|_{L^{2\si+2}_r}^{2\si_*+2} - \| w_*
    \|_{L^{2\si+2}_r}^{2\si+2}  \right| \le K_3 (\si_* - \si),
\end{equation*}
for $K_3>0$ uniform in $\si$. Thus,
\begin{equation*}
  \underbrace{\left(1-K_3(\si_*-\si)
    \right)^{\frac{\si_*}{2(\si_*-\si)}}}_{\underset{\si \to
      \si_*}{\longrightarrow } \exp(-\si_*K_3/2)}  \le \lambda \le
  \underbrace{(1+(K_2 +
    K_3)(\si_*-\si))^{\frac{\si_*}{2(\si_*-\si)}}}_{\underset{\si \to
      \si_*}{\longrightarrow } \exp(\si_*(K_2+K_3)/2)},
\end{equation*}
so we get the estimate $ 0 < K_4 \le \lambda \le K_5$ for constants
$K_4$, $K_5$ uniform in $\si$. We then go back to the definition of
$w_{*,\lambda}$, and we write that 
\begin{equation*}
  \| w_{*,\lambda} \|_{L^2_r}^2 = \lambda^{2/\sigma_* -d} \frac{\| w_*
    \|_{L^2_r}^2}{\| w_*\|_{L^{2\si_*+2}_r}^2} = \lambda^{-2} \frac{\|
    w_* \|_{L^2_r}^2}{\| w_*\|_{L^{2\si_*+2}_r}^2} \le K_6,
\end{equation*}
and
\begin{equation*}
  \| w_{*,\lambda}' \|_{L^2_r}^2 = \lambda^{2/\si_*-d+2} \frac{\| w_{*}'
    \|_{L^2_r}^2}{\| w_{*}\|_{L^{2\si_*+2}_r}^2} = \frac{\| w_{*}'
    \|_{L^2_r}^2}{\| w_*\|_{L^{2\si_*+2}_r}^2} \le K_7,
\end{equation*}
for uniform constants $K_6$, $K_7>0$. This allows to conclude that
\begin{equation*}
  \K_\si \le \| w_{*,\lambda}' \|_{L^2_r}^2 + \frac{1}{\si} \|
  w_{*,\lambda} \|_{L^2_r}^2 \le K_8
\end{equation*}
for $K_8>0$ uniform in $\si$.
\end{proof}

Note that $(\si_*-\si)/(\si_*(\si+1))+1/\si_* = d/(2\si+2)$. For any
$\varphi \in H^1_r$ and $\lambda>0$, we then define
$\varphi_{\lambda}$ by  
\begin{equation*}
  \varphi_{\lambda}(r) := \lambda^{ d/(2\si+2)} \varphi(\lambda r).
\end{equation*}
We readily compute
\begin{align*}
 & \ \| \varphi_\lambda \|_{L^2_r} =
   \lambda^{(\si_*-\si)/(\si_*(\si+1))-1} \| \varphi \|_{L^2_r}, \\ 
 &  \ \| \varphi_{\lambda}' \|_{L^2_r} =
   \lambda^{(\si_*-\si)/(\si_*(\si+1))} \| \varphi' \|_{L^2_r}, \\ 
 & \ \| \varphi_{\lambda} \|_{L^{2\si+2}_r} =  \| \varphi \|_{L^{2\si+2}_r}.
\end{align*}

We then define the application 
\begin{equation*}
  T_\si :
  \left\{
    \begin{array}{rcccl}
      & \R_+^* \times \enstq{\psi \in H^1_r}{\| \psi\|_{L^2_r}=\|
        \psi\|_{L^{2\si+2}_r} =1} & \rightarrow & \enstq{\psi \in H^1_r}{ \|
 \psi\|_{L^{2\si+2}_r}=1} \\ 
& (\lambda,\varphi) & \mapsto & \varphi_{\lambda}.
    \end{array}
    \right.
\end{equation*}
It is straightforward to see that the application $T_\si$ is bijective.

\begin{lemma} \label{lem:expression_lambda_sigma}
Let 
\begin{equation*}
  \Lambda_\si := \inf_{\varphi \in H^1_r} \enstq{\|\varphi'\|_{L^2_r}^2}{
    \| \varphi\|_{L^2_r}=\| \varphi\|_{L^{2\si+2}_r} =1}.
\end{equation*}
We have
\begin{equation*}
  \K_\si = \left( \frac{1+\si_*}{\si_*-\si}
  \right)^{\frac{\si_*-\si}{\si_*(1+\si)}}
  \frac{\si_*(1+\si)}{\si(1+\si_*)}
  \Lambda_\si^{1+\frac{\si_*-\si}{\si_*(1+\si)}}.
\end{equation*}
Moreover, if $v_\si$ is a minimizer of \eqref{eq:variation}, then
\begin{equation*}
  \| v_\si \|_{L^2_r} = \sqrt{\frac{\si_*-\si}{1+\si_*}} \| v_{\si}'
  \|_{L^2_r}.
\end{equation*}
\end{lemma}
\begin{proof}
For any $v \in H^1_r$ such that $\| v\|_{L^{2\si+2}_r}=1$, there
exists a unique $\lambda>0$ and $\varphi \in H^1_r$ such that $\|
\varphi \|_{L^2_r}=\|\varphi \|_{L^{2\si+2}_r} = 1$ and
$v=\varphi_\lambda$. Therefore,  
\begin{equation*}
  \| v' \|_{L^2_r}^2 + \frac{1}{\si} \| v\|_{L^2_r}^2 =
  \lambda^{2(\sigma_*-\si)/(\si_*(1+\si))}  \left( \| \varphi'
    \|_{L^2_r}^2 + \frac{\lambda^{-2}}{\si} \underbrace{\| \varphi
      \|_{L^2_r}^{2}}_{=1} \right).
\end{equation*}
Minimizing the expression on the right hand side with respect to
$\lambda$, leading to
\begin{equation} \label{eq:alpha_0}
\lambda_0 = \sqrt{\frac{1+\si_*}{\| \varphi'\|_{L^2_r}^2 (\si_*-\si)}},
\end{equation}   
we then get 
\begin{equation*}
 \| v \|_{L^2_r}  = \lambda_0^{\frac{\si_*-\si}{\si_*(1+\si)}-1} =
  \lambda_0^{\frac{\si_*-\si}{\si_*(1+\si)}} \|  \varphi' \|_{L^2_r}
  \sqrt{\frac{\si_* -  \si}{1+\si_*}} = \| v' \|_{L^2_r}
  \sqrt{\frac{\si_* - \si}{1+\si_*}}.  
 \end{equation*}
 This is in particular the case when $v = v_\si$ which is a minimizer
 of \eqref{eq:variation} (and thus already a minimizer with respect
 to $\lambda$). Moreover, we then get that 
 \begin{align*}
   \| v' \|_{L^2_r}^2 + \frac{1}{\si} \| v\|_{L^2_r}^2
   & = \lambda_0^{(\si_*-\si)/(\si_*(1+\si))} \| \varphi' \|_{L^2_r}^2
     \left(1+ \frac{\si_*-\si}{\si_*(1+\si)} \right)\\ 
 & = \left(  \frac{1+\si_*}{\si_*-\si}
   \right)^{(\si_*-\si)/(\si_*(1+\si))}  \| \varphi'
   \|_{L^2_r}^{2(1-(\si_*-\si)/(\si_*(1+\si))} \frac{\si_*(1+\si)}{\si
   (1+\si_*)}. 
 \end{align*}
 The relation between $\K_\si$ and $\Lambda_\si$ follows from the
 minimization of the remaining expression on  
 \begin{equation*}
   \enstq{\varphi \in H^1_r}{ \| \varphi \|_{L^2_r} = \| \varphi
     \|_{L^{2\si+2}}=1}.
 \end{equation*}
\end{proof}

\begin{lemma}
Let $v_\si$ be a minimizer of \eqref{eq:variation}. We have the identities
\begin{equation*}
  \| v_{\si}' \|_{L^2_r}^2 = \K_\si
  \frac{\si(1+\si_*)}{\si_*(1+\si)} \quad \text{and} \quad \| v_{\si}
  \|_{L^2_r}^2 =  \K_\si \frac{\si(\si_*-\si)}{\si_*(1+\si)}.
\end{equation*}
\end{lemma}
\begin{proof}
These expressions are direct consequences of
Lemma~\ref{lem:expression_lambda_sigma} and the fact that 
\begin{equation*}
  \K_\si= \| v_{\si}' \|_{L^2}^2 + \frac{1}{\si} \| v_\si
  \|_{L^2}^2.
\end{equation*}
\end{proof}

\begin{lemma} \label{lem:a_si_b_si}
Let
\begin{equation} \label{eq:varphi_si}
\varphi_\si := \underset{\varphi \in H^1_r}{\mathrm{argmin}} \enstq{
  \| \varphi' \|_{L^2_r}^2}{ \ \| \varphi \|_{L^2_r}= \| \varphi
  \|_{L^{2\si+2}_r}=1}.  
\end{equation}  
Then the quantity
\begin{equation} \label{eq:mu_sigma}
   \Lambda_\si:=\| \varphi_{\si}' \|_{L^2_r}^2= \left[ \left(
      \frac{\si_*-\si}{1+\si_*} \right)^{\frac{\si_*-\si}{(\si+1)\si_*}}
    \frac{\si(1+\si_*)}{\si_*(1+\si)} \K_\si
  \right]^{\frac{1}{1+(\sigma_*-\si)/(\si_*(1+\si))}} 
\end{equation}  
is both bounded and bounded away from 0 as $\si$ varies. Moreover, denoting
\begin{equation*}
  a_\si := \frac{\si_*(1+\si)}{\si(1+\si_*)}
  \Lambda_{\si}^{1+\frac{2(\si_*-\si)}{\si_*(1+\si)}} \quad \text{and}
  \quad b_\si :=  \frac{\Lambda_\si}{1+\si_*},
\end{equation*}
$\varphi_\si$ is the only radially symmetric, positive
solution vanishing at infinity,  to
\begin{equation*}
  \varphi_{\si}'' + \frac{d-1}{\rho} \varphi_{\si}' + a_\si
  \varphi_{\si}^{1+2\si} = \frac{\si_*-\si}{\si} b_\si \varphi_\si.
\end{equation*}
\end{lemma}
\begin{proof}
The first part follows directly from
Lemma~\ref{lem:lambda_sigma_bound}. Recalling that 
\begin{equation*}
  v_\si(r) = \lambda_0^{(\si_*-\si)/(\si_*(1+\si))} \varphi_\si(\lambda_0
  r),
\end{equation*}
with $\lambda_0$ given in \eqref{eq:alpha_0}, and that $v_\si$
satisfies \eqref{eq:u_sigma}, we infer that $\varphi_\si$ satisfies 
\begin{equation*}
  \varphi_{\si}'' + \frac{d-1}{\rho} \varphi_{\si}' + \K_\si
  \lambda_0^{2\si((\si_*-\si)/(\si_*(1+\si)) -2} \varphi_{\si}^{1+2\si}
  = \frac{\lambda_0^{-2}}{\si} \varphi_\si,
\end{equation*}
with
\begin{equation*}
  \lambda_0^{-2} = \frac{\si_*-\si}{1+\si_*} \|
  \varphi_{\si}'\|_{L^2_r}^2 =  \frac{\si_*-\si}{1+\si_*} \Lambda_\si =
  (\sigma_*-\si) b_\si.
\end{equation*}
One can then directly compute that
\begin{align*}
  \lambda_0^{2\si((\si_*-\si)/(\si_*(1+\si)) -2}
  &= \lambda_0^{2
    (\si_*-\si) \left( \si/(\si_*(1+\si)) - 1/\si_*  \right)} =
    \lambda_0^{-2 \frac{\si_* - \si}{\si_*(1+\si)}} \\
  &= \left( \frac{\si_*-\si}{1+\si_*} \Lambda_\si
  \right)^{\frac{\si_*-\si}{\si_*(1+\si)}},
\end{align*}
and
\begin{equation*}
  \K_\si = \left( \frac{1+\si_*}{\si_* - \si}
  \right)^{(\si_*-\si)/(\si_*(1+\si))} \frac{\si_*(1+\si)}{\si
    (1+\si_*)} \Lambda_{\si}^{1+(\si_*-\si)/(\si_*(1+\si))},
\end{equation*}
so we get the identity
\begin{equation*}
  \K_\si \lambda_0^{2\si((\si_*-\si)/(\si_*(1+\si)) -2} =a_\si.
\end{equation*}
The fact that $\varphi_\si$ is positive and thus unique from
\cite{Kwong} follows from the same arguments as in the proof of
\cite[Theorem~B]{Weinstein83}.  
\end{proof}

\begin{lemma} \label{lem:inferior_L_infty_bound}
Let $\varphi_\si$ be defined by \eqref{eq:varphi_si}. We have $\|
\varphi_\si \|_{L^{\infty}_r} \geq 1$. 
\end{lemma}
\begin{proof}
This follows from interpolation, as 
\begin{equation*}
  1 = \| \varphi_\si \|_{L^{2\si+2}_r}^{2\si+2} \le \| \varphi_\si
  \|_{L^{\infty}_r}^{2\si} \| \varphi_\si \|_{L^2_r}^{2} = \|
  \varphi_\si \|_{L^{\infty}_r}^{2\si}.
\end{equation*}
\end{proof}

For our upcoming analysis, we will rely on the following result, which
is a direct application of \cite[Theorem 3]{Stein70} (as pointed out
after the statement in  \cite{Stein70}, the continuity of
$C_p$ follows from the proof). 

\begin{lemma} \label{lem:Stein_theorem}
For any $p \in (1,\infty )$, there exists $C_p>0$ such that for any $u
\in L^p(\R^d)$ and for any $\eps >0$, $v=\Delta(-\Delta+\eps)^{-1} u $
satisfies 
\begin{equation*}
  \| v \|_{L^p} \le C_p \| u \|_{L^p}.
\end{equation*}
Moreover, $C_p$ is locally bounded with respect to $p$.
\end{lemma}

\begin{lemma}  \label{lem:Hardy-Littlewood-Sobolev}
Let $u \in H^1(\R^d)$ with $d \ge 3$ such that $\Delta u \in
L^p(\R^d)$ for some $p \in \left[1,d/2 \right)$. Then there exists
$K_p>0$ (locally bounded with respect to $p$) such that 
\begin{equation*}
  \| u \|_{L^q(\R^d)} \le K_{p} \| \Delta u \|_{L^p(\R^d)} \quad
  \text{where} \ \frac{1}{q}=\frac{1}{p}-\frac{2}{d} \le 1.
\end{equation*}
\end{lemma}
\begin{proof}
Let $v=-\Delta u$, in particular one can write
$u=\frac{c_d}{|x|^{d-2}} \ast v$, and the conclusion follows from
Hardy-Littlewood-Sobolev inequality (see e.g. \cite{Stein70}).
\end{proof}

\begin{lemma} \label{lem:varphi_zeta}
For any $\zeta \in \left(1,d/2 \right)$ with $d \ge 3$, there exists
$C_{\zeta,d}>0$ (locally bounded with respect to $\zeta$) such that
for any $\si \in \left( \si_*/2,\si_* \right]$,  
\begin{equation*}
  \| \varphi_\si \|_{L^\gamma_r} \le C_{\zeta,d} \| \varphi_\si
  \|_{L^{(1+2\si)\zeta}_r}^{1+2\si} \quad \text{for} \
  \frac{1}{\gamma}=\frac{1}{\zeta}-\frac{2}{d}.
\end{equation*}
\end{lemma}
\begin{proof}
We prove the result on $\R^d$, which in particular implies the result
for radial functions. From Lemma \ref{lem:a_si_b_si} we have 
\begin{equation*}
  (-\Delta + (\si_*-\si)b_\si/\si) \varphi_\si = a_\si
  \varphi_{\si}^{1+2\si},
\end{equation*}
with $b_\si>0$. Thus, we can write 
\begin{equation*}
  \Delta \varphi_{\si} = \Delta(-\Delta + (\si_* -
  \si)b_\si/\si)^{-1}\( a_\si \varphi_{\si}^{1+2\si}\).
\end{equation*}
From Lemmas~\ref{lem:Stein_theorem} and
\ref{lem:Hardy-Littlewood-Sobolev}, along with the fact that $a_\si$ is
bounded uniformly with respect to $\si$, we infer  
\begin{equation*}
  \| \varphi_\si \|_{L^{\gamma}_r}  \le K_{\zeta}\|
  \Delta\varphi_\si \|_{L^{\zeta}_r} \le K_{\zeta} C_\zeta
  \| a_\si
  \varphi_{\si}^{1+2\si} \|_{L^{\zeta}_r} \le C_{\zeta,d} \|
  \varphi_{\si} \|_{L^{(1+2\si)\zeta}_r}^{1+2\si}.
\end{equation*}
\end{proof}

\begin{lemma}
For $d\ge 5$, there exists $\gamma \in (1,2)$ such that $\| \varphi
\|_{L^{\gamma}_r}$ is uniformly bounded in $\si$, for $\si$ close
enough to $\si_*$. 
\end{lemma}
\begin{proof}
Take
\begin{equation*}
  \zeta = \left\{
    \begin{aligned}
& 1+\eps & \text{if} \ d=5, \\
& \frac{2}{1+2(\sigma_*-\eps)} & \text{if} \ d\ge 6,
    \end{aligned}
  \right.
\end{equation*}
for some $\eps>0$ small enough, then $\zeta \in (1,2)$ as
\begin{equation*}
  \frac{2}{1+2\si_*} = \frac{2(d-2)}{d+2} = 2 - \frac{8}{d+2} \in
  \left[1,2 \right) \quad \text{for} \ d \ge 6.
\end{equation*}
Applying Lemma \ref{lem:varphi_zeta} we get for $d\ge 6$ that 
\begin{equation*}
  \| \varphi_\si \|_{L^{\gamma}_r}  \le C \| \varphi_\si
  \|_{L^{(1+2\si)\zeta}_r}^{1+2\si} \quad \text{with} \
  \frac{1}{\gamma}=\frac{1}{\zeta} - \frac{2}{d}.
\end{equation*}
First, since
\begin{equation*}
  \frac{1+2 \si_*}{2} - \frac{2}{d} = \frac12 \frac{d+2}{d-2} -
  \frac{2}{d} = \frac12 \frac{d(d+2)-4(d-2)}{d(d-2)} = \frac12
  \frac{d^2-2d+8}{d^2-2d} >\frac12,
\end{equation*}
we have that $\gamma <2$ for $\eps$ small enough. On the other hand,
we have $(1+2\si)\zeta \ge 2$ for $\si \ge \si_*-\eps$, and  
\begin{equation*}
  (1+2\si)\zeta \le 2 \frac{1+2\si_*}{1+2(\si_*-\eps)} < 2+2\si_*
\end{equation*}
for $\eps$ small enough. Thus, by interpolation and as $\| \varphi_\si
\|_{L^2_r}=\| \varphi_\si \|_{L^{2\si+2}_r}=1$, we have $\|
\varphi_\si \|_{L^{(1+2\si)\zeta}} \le 1$, so that $\| \varphi_\si
\|_{L^\gamma_r} \le C$. The case $d=5$ is performed similarly.  
\end{proof}

\begin{lemma} \label{lem:extraction_strong_convergence_L2}
Let $v_n \in H^1_r$ be a family of functions which are radially
symmetric decreasing and uniformly bounded in $H^1_r$ and in
$L^{\gamma}_r$ for some $\gamma \in \left[1,2 \right)$. Then there
exists a subsequence $v_{n_k}$ and $v \in H^1_r \cap L^{\gamma}_r$
a radially symmetric decreasing function such that $v_{n_k}
\underset{k \rightarrow \infty}{\longrightarrow} v$ in $L^2_r$. 
\end{lemma}
\begin{proof}
  Strauss lemma (\cite{Strauss77}, see also
  \cite[Proposition~1.7.1]{CazCourant}) implies that there exists a 
  subsequence $(v_{n_k})_{k\ge 0}$ and $v\in H^1_r$ such that
  $v_{n_k}\to v$ in $L^p_r$ for all $p\in (2,2^*)$, where 
  $2^*=\frac{2d}{d-2}$. Fatou's lemma yields $v\in L^{\gamma}_r$, and
  interpolation implies $v_{n_k}\to v$ in $L^2_r$. 
\end{proof}
The next lemma is a straightforward consequence of \cite{Aubin76} and
\cite{Talenti76}:  
\begin{lemma} \label{lem:equation_sigma_star}
Let $u \in H^1_r$ be a radially symmetric decreasing function satisfying 
\begin{equation*}
  \Delta u + a u ^{1+2\si_*} =0 \quad \text{in} \ \mathcal{D}'(\R^d),
\end{equation*}
for some $a>0$. Then there exists $\lambda \ge 0$ such that
$u(x)=\lambda w_*(x \sqrt{a} \lambda^{\si_*} )$. 
\end{lemma}

We will now use the functional introduced in \cite{Weinstein83}, for $0<\si<\si_*$,
  \begin{equation*}
    J_\si(f) = \frac{\|\nabla f\|_{L^2(\R^d)}^{d\si}
      \|f\|_{L^2(\R^d)}^{2-(d-2)\si}}{\|f\|_{L^{2\si+2}(\R^d)}^{2\si+2}}, 
  \end{equation*}
  and its limiting expression from \cite{Talenti76},
  \begin{equation*}
    J_{\si_*}(f) = \frac{\|\nabla f\|_{L^2(\R^d)}^{d\si_*}
     }{\|f\|_{L^{2\si_*+2}(\R^d)}^{2\si_*+2}}. 
  \end{equation*}
  It follows from \cite{Talenti76} (case $\si=\si_*$) and
  \cite[Theorem~B]{Weinstein83} (case $0<\si<\si_*$) that the
  minimizers of $I_\si$, $0<\si\le \si_*$, are ground state 
  solutions (hence positive) to
  \begin{equation}
    \label{eq:psi}
    \frac{d\si}{2}\Delta \psi +\psi^{2\si+1} = \frac{(d-2)}{2} (\si_*-\si)\psi. 
  \end{equation}
  There is uniqueness (up to translation) in the case $\si<\si_*$, but
  no longer in the case $\si=\si_*$ due to the additional scaling
  invariance. For any $\lambda>0$,
\begin{equation*}
  w_{*,\lambda}:= \lambda^{1/\si_*}w_*\(\lambda \rho\)
\end{equation*}
is the unique positive, radially symmetric solution to
\begin{equation*}
   w_{*,\lambda}''
   +\frac{d-1}{\rho}w_{*,\lambda}'+w_{*,\lambda}^{2\si_*+1}=0,\quad
   w_{*,\lambda}(0) = \lambda^{1/\si_*},\ w_{*,\lambda}'(0)=0.
\end{equation*}
Denote, for $\si\le \si_*$, 
 \begin{equation*}
    \ell_\si:=\inf_{f\in H^1(\R^d)} J_\si(f) .
  \end{equation*}
As recalled above, it follows from \cite{FGL21} (see also
\cite{FLW23}) that $\si\mapsto 
\ell_\si$ is continuous on $(0,\si_*]$. In addition, the value in
the endpoint case $\si_*$ is classical (see \cite{Talenti76} or
\cite[Theorem~2.49]{FLW23}), 
\begin{equation*}
  \ell_{\si_*}=
  \(\frac{d(d-2)}{4}\)^{d/(d-2)}
  2^{2/(d-2)}\pi^{(2d+2)/{d-2}}\Gamma\(\frac{d+1}{2}\)^{-2/(d-2)}. 
\end{equation*}
On the other hand, from \cite{Weinstein83}, for $\si<\si_*$, $J_\si$
is attained by $\psi^*$ which is the unique (from \cite{Kwong}) positive solution to
\begin{equation*}
  \frac{d\si}{2}\Delta \psi^*+\K_\si (\si+1)\(\psi^*\)^{2\si+1}=
  \frac{d-2}{2}(\si_*-\si)\psi^*. 
\end{equation*}
Moreover, it satisfies
\begin{equation*}
  \|\psi^*\|_{L^2(\R^d)}= \|\nabla \psi^*\|_{L^2(\R^d)}=1. 
\end{equation*}
We have the following result.

\begin{lemma} \label{lem:limit_mu_sigma}
The quantity $\Lambda_\si$ from \eqref{eq:mu_sigma} is continuous on
$\left( 0, \si_* \right]$. In particular,  
\begin{equation*}
  \Lambda_\si \underset{\si \to \si_*}{\longrightarrow} \left[ \left(
      \frac{d(d-2)}{4} \right)^{\frac{d}{d-2}} 2 ^{\frac{2}{d-2}}
    \pi^{\frac{2d+2}{d-2}} \Gamma\left( \frac{d+1}{2}
    \right)^{-\frac{2}{d-2}} \right]^{\frac{2}{ d \si_*}} =:
  \Lambda_{\si_*}
\end{equation*}
\end{lemma}
 \begin{proof}
 We first show that $\ell_\si = \Lambda_{\si}^{d \si /2}$. As $\|
 \varphi_\si \|_{L^{2\si+2}(\R^d)}= \| \varphi_\si \|_{L^2(\R^d)}=1$,
 we infer 
 \begin{equation*}
   J_\si(\varphi_\si)= \| \nabla \varphi_\si \|_{L^2(\R^d)}^{d \si /2}
   = \Lambda_{\si}^{d\si/2} \geq  \ell_\si
 \end{equation*}
by definition of $\ell_\si$. 

 On the other hand, as mentioned above we know that the minimum of
 $\ell_\si$ is attained for a positive, radially symmetric function
 $\psi^*$. We resume some arguments from the proof of
 \cite[Theorem~B]{Weinstein83}. 
 For any positive, radially symmetric function $f$,
 define $f_{a,b}(x)=a f(b x)$. We have 
 \begin{equation*}
   J_\si(f_{a,b}) = \frac{a^{d\si} b^{(1-d/2)d\si} \| \nabla f
     \|_{L^2(\R^d)}^{d\si} a^{2-(d-2)\si} b^{-\frac{d}{2}(2-(d-2)\si)}
     \| f \|_{L^2}^{2-(d-2)\si} }{a^{2\si+2} b^{-d} \| f
     \|_{L^{2\si+2}(\R^d)}^{2\si+2}} = J_\si(f).
 \end{equation*}
 Then we can find $a,b>0$ such that $\| f_{a,b} \|_{L^2(\R^d)} = \|
 f_{a,b} \|_{L^{2\si+2}(\R^d)}  =1$, so that $\| \nabla f_{a,b}
 \|_{L^2(\R^d)}^2 \ge \Lambda_\si$, and  
 \begin{equation*}
   J_\si(f_{a,b}) = \| \nabla f_{a,b} \|_{L^2(\R^d)}^{d\si} \ge
   \Lambda_{\si}^{d\si/2}.
 \end{equation*}
 Thus, $J_\si(f) \ge \Lambda_{\si}^{d\si/2}$, and we get by minimizing
 over $f \in H^1(\R^d)$ that $\ell_\si \ge \Lambda_{\si}^{d \si/2}$,
 so we get the equality. 
 
 We finally conclude by the continuity of $\ell_\si>0$.
 \end{proof}

 \begin{proposition}
 We have $\varphi_\si \underset{\si \to \si_*}{\longrightarrow}
 \varphi_*$ in $H^1_r$,  where $\varphi_*(r)=\ll w_*(\ll^{\si_*}
 \sqrt{a_{\si_*}} r)$ with  
 \begin{equation*}
   a_{\si_*}=\lim_{\si \to \si_*} a_\si \quad \text{and} \quad \ll
   = \(\frac{\| w_* \|_{L^2_r}}{\Lambda_{\si_*}^{d/4} }\)^{1/\si_*}.
 \end{equation*}
 \end{proposition}
 \begin{proof}
 By Lemma \ref{lem:a_si_b_si} and Lemma \ref{lem:varphi_zeta},
 $\varphi_\si$ satisfies the assumption of
 Lemma~\ref{lem:extraction_strong_convergence_L2}. Thus, there exists
 $\varphi_{_*} \in H^1_r$ radially symmetric decreasing function such
 that  $\varphi_{\si_n} \underset{n \to \infty}{\longrightarrow}
 \varphi_{*}$ in $L^2_r$. Moreover, using the fact that
 $\Lambda_\si \underset{\si \to \si_*}{\longrightarrow}
 \Lambda_{\si_*}$ from Lemma \ref{lem:limit_mu_sigma}, we infer: 
 \begin{align*}
 & \ \Delta \varphi_{\si_n} \underset{n \to \infty}{\longrightarrow}
   \Delta \varphi_{*} \quad   \text{in }  H^{-2}(\R^d), \\
 & \ a_{\si_n} \varphi_{\si_n}^{1+2{\si_n}} \underset{n \to
   \infty}{\longrightarrow} \Lambda_{\si_*} \varphi_{*}^{1+2\si_*}
   \quad   \text{in }  \mathcal{D}'(\R^d), \\
 & \ (\si_*-\si_n) b_{\si_n} \varphi_{\si_n} \underset{n \to
   \infty}{\longrightarrow} 0 \quad   \text{in }  L^2(\R^d). 
 \end{align*}
 For the second claimed convergence, write
 \begin{equation*}
  \varphi_{*}^{1+2\si_*} - \varphi_{\si_n}^{1+2\si_n}=
  \varphi_{*}^{1+2\si_*} - \varphi_{*}^{1+2\si_n} +
  \varphi_{*}^{1+2\si_n} - \varphi_{\si_n}^{1+2\si_n}, 
\end{equation*}
In view of the pointwise estimate $|\varphi_*|^{1+2\si_n} \le
|\varphi_*|\max\(1,|\varphi_*|^{2\si_*}\)$, Lebesgue Dominated
Convergence Theorem yields
\begin{equation*}
   \varphi_{*}^{1+2\si_*} - \varphi_{*}^{1+2\si_n}\Tend n {\infty} 0
   \quad\text{in }\mathcal{D}'(\R^d).
 \end{equation*}
 For the remaining difference, write similarly
 \begin{align*}
   |\varphi_{*}^{1+2\si_n} - \varphi_{\si_n}^{1+2\si_n}|
   &\lesssim |\varphi_{*}
     - \varphi_{\si_n}| \( |\varphi_{*}|^{2\si_n} + |\varphi_\si|^{2\si_n}\)\\
   &\lesssim |\varphi_{*}
     - \varphi_{\si_n}| \( \max\(1,|\varphi_{*}|^{2\si_*}\) +
     \max\(1,|\varphi_{\si_n}|^{2\si_*} \)\).
 \end{align*}
 Using the strong convergence $\varphi_{\si_n}\to \varphi_*$  in
 $L^2_r$, and the boundedness of $(\varphi_{\si_n})_n$ in
 $H^1_r\subset L^{2\si_*+2}_r$,
 Lebesgue Dominated
Convergence Theorem yields again
\begin{equation*}
   \varphi_{*}^{1+2\si_*} - \varphi_{\si_n}^{1+2\si_n}\Tend n {\infty} 0
   \quad\text{in }\mathcal{D}'(\R^d).
 \end{equation*}
Note that $a_{\si_*}=\Lambda_{\si_*}$ from the explicit
 expressions in Lemma~\ref{lem:a_si_b_si}.
 Again from Lemma~\ref{lem:a_si_b_si}, we know that $\varphi_\si$ satisfies 
 \begin{equation*}
   \Delta \varphi_\si + a_\si \varphi_\si^{1+2\si} = \frac{\si_*-\si}{\si} b_\si
   \varphi,
 \end{equation*}
 so we infer that $\varphi_{*}$ satisfies
 \begin{equation*}
   \Delta \varphi_{*} + \Lambda_{\si_*} \varphi_{*}^{1+2\si_*} = 0.
 \end{equation*}
 Invoking Lemma~\ref{lem:equation_sigma_star}, there exists $\ll
 \ge 0$ such that $\varphi_{*}(r)=\ll w_*(\ll^{\si_*}
 \sqrt{\Lambda_{\si_*}} r)$. Moreover, since $\| \varphi_{\si_n}
 \|_{L^2_r}=1$, we also get $\| \varphi_{*} \|_{L^2_r}=1$ by strong
 convergence in $L^2_r$. Thus,   
 \begin{equation*}
   1 =  \| \varphi_{*} \|_{L^2_r} = \ll \left\| w_*(\ll^{\si_*}
   \sqrt{\Lambda_{\si_*}} \cdot )\right\|_{L^2_r} = \ll^{1-\frac{\si_* d}{2}}
   \Lambda_{\si_*}^{-\frac{d}{4}} \| w_* \|_{L^2_r},
 \end{equation*}
 with $1-d\sigma_* /2=-\si_*$. Hence we can deduce that
 \begin{equation*}
   \ll = \left( \frac{\| w_* \|_{L^2_r}}{a_{\si_*}^{d/4}}
   \right)^{1/\si_*}= \left( \frac{\| w_* \|_{L^2_r}}{\Lambda_{\si_*}^{d/4}}
   \right)^{1/\si_*}.
 \end{equation*}
 Since the limit $\varphi_{*}$ is uniquely characterized, no
 subsequence is needed.  In order to
 infer that $\varphi_\si \underset{\si \to \si_*}{\longrightarrow}
 \varphi_{*}$ in $H^1_r$, we only have to prove that $\|
 \varphi_\si' \|_{L^2_r} \underset{\si \to \si_*}{\longrightarrow} \|
 \varphi_{*}' \|_{L^2_r}$. On the one hand, we know that $\|
 \varphi_\si' \|_{L^2_r}^2 \underset{\si \to \si_*}{\longrightarrow}
 \Lambda_{\si_*}$ from Lemma~\ref{lem:a_si_b_si}. On the other hand, we
 can explicitly compute 
\begin{align*}
  \| \varphi_{*}' \|_{L^2_r}
  & = \ll^{1+\si_*} \Lambda_{\si_*}^{1/2} \left\| w_*'(\ll^{\si_*}
    \sqrt{\Lambda_{\si_*}} \cdot ) \right\|_{L^2_r} = \ll^{1+\si_*- d\si_*/2}
    \Lambda_{\si_*}^{1/2-d/4} \| w_*' \|_{L^2_r} \\ 
& = \Lambda_{\si_*}^{1/2-d/4} \| w_*' \|_{L^2_r} =
  \Lambda_{\si_*}^{-\frac{1}{2\si_*}} \| w_*' \|_{L^2_r} .
\end{align*}
Since $w_*$ is the Aubin-Talenti soliton,
\begin{equation*}
  \Lambda_{\si_*} = \ell_{\si_*}^{2/(d\si_*)} =
\frac{\|w_*'\|_{L^2_r}^2}{\|w_*\|_{L^{2\si_*+2}_r}^2}. 
\end{equation*}
Proceeding like we did in Subsection~\ref{sec:extra}, we check that
$w_*$ satisfies the identity
\begin{equation*}
  \|w_*'\|_{L^2_r}^2 = \|w_*\|_{L^{2\si_*+2}_r}^{2\si_*+2},
\end{equation*}
and we infer $ \Lambda_{\si_*}^{-\frac{1}{2\si_*}} \| w_*' \|_{L^2_r}
= \Lambda_{\si_*}^{1/2}$, hence the result.
 \end{proof}
 
 There remains to prove that $\varphi_\si(0) \underset{\si \to
   \si_*}{\longrightarrow} \alpha$ to end the proof of the convergence
 towards the algebraic soliton. 
 
 \begin{lemma} \label{lem:superior_L_infty_bound}
 There exists $C_5>0$ such that $\beta_\si := \| \varphi_\si \|_
 {L^{\infty}} \le C_5$. 
 \end{lemma}
 \begin{proof}
 Let $\omega_\si$ be defined by $\varphi_\si(r)=\beta_\si \omega_\si(\rho)$
 where $\rho=\sqrt{a_\si} \beta_\si^{\si} r$. Then $\omega_\si$ is a
 solution to 
 \begin{equation*}
   \left\{
     \begin{aligned}
& \omega_{\si}'' + \frac{d-1}{\rho} \omega_{\si}' + \omega_{\si}^{1+2\si} =
  \frac{\si_*-\si}{\si} \frac{b_\si}{a_\si \beta_\si^{2 \si}} \omega_\si, \\ 
& \omega_\si(0)=1, \quad \omega_\si'(0)=0.
     \end{aligned}
   \right.
 \end{equation*}
 It is important to observe that by uniqueness of radially symmetric,
 positive solutions going to zero at infinity to \eqref{eq:ivp-w}
 (from \cite{Kwong}, see also \cite[Theorem
1.3]{FQTY08}), we have, since $w_\si(0)=\omega_\si(0)$,
 \begin{equation*}
   \omega_\si=w_\si,\quad \epsilon = \frac{\si_*-\si}{\si} \frac{b_\si}{a_\si
     \beta_\si^{2 \si}}. 
 \end{equation*}
Proposition~\ref{prop:unif-cv} implies for instance that
\begin{equation*}
  \| w_\si'\|_{L^2_r((0,1))} \underset{\si \to \si_*}{\longrightarrow}
  \| w_*' \|_{L^2_r((0,1))} >0.
\end{equation*}
Thus, writing $I_\si=\left( 0,1/(\sqrt{a_\si} \beta_\si^{\si})\right)$, we have
\begin{equation*}
  \| \varphi_\si'\|_{L^2_r(I_\si)} = a_\si^{\frac12 - \frac{d}{4}}
  \beta_\si^{1+\frac{\si}{2}-\frac{\si d}{2}} \|
  w_\si'\|_{L^2_r((0,1))} = a_\si^{-\frac{1}{2\si_*}}
  \beta_\si^{1-\frac{\si}{\si_*}} \| w_\si' \|_{L^2_r((0,1))} .
\end{equation*}
Therefore, we get that
\begin{equation}\label{eq:beta}
  \beta_\si = \left( \frac{a_\si^{1/(2\si_*)} \|
      \varphi_\si'\|_{L^2_r(I_\si)}}{\| w_\si' \|_{L^2_r((0,1))}}
  \right)^{\frac{1}{1-\si/\si_*}}.
\end{equation}
By contradiction, if $\beta_{\sigma_n} \underset{\si \to
  \si_*}{\longrightarrow} \infty$ for some $\si_n \underset{n \to
  \infty}{\longrightarrow} \si_*$, one would get that 
\begin{equation*}
  \| \varphi_{\si_n}'\|_{L^2_r(I_{\si_n})} \underset{n \to
    \infty}{\longrightarrow} 0
\end{equation*}
from the convergence of $\varphi_\si$ in $H^1_r$, and since
$I_\si \to \left\{ 0 \right\}$ as $\si \to\si_*$.
From \eqref{eq:beta}, we would then have that $\beta_{\si_n}
\underset{n \to \infty}{\longrightarrow} 0$, a contradiction. 
 \end{proof}
 
 \begin{proposition}
 We have $\varphi_\si(0) \underset{\si \to \si_*}{\longrightarrow} \ll$.
 \end{proposition}
 \begin{proof}
 We know that $\beta_\si = \varphi_\si(0)$ is bounded from
 Lemma~\ref{lem:superior_L_infty_bound}, and bounded away from 0 by
 Lemma~\ref{lem:inferior_L_infty_bound}. Take any converging subsequence of
 $\beta_\si$ denoted by $\beta_{\si_n}$, and denote by $\beta$ the
 limit. With the same notations from the previous lemma, we have once
 again that $\omega_{\si_n}$ converges in $W_{r,\mathrm{loc}}^{1,\infty}$
 to $w_*$. On the other hand, from the convergence of $\varphi_\si$ to
 $\varphi_*$ and of $\beta_{\si_n}$ to $\beta$, we also know that 
 \begin{equation*}
   w_{\si_n} \rightarrow  \frac{1}{\beta} \varphi_* \left(
     \frac{\cdot}{\sqrt{\Lambda_{\si_*}} \beta^{\si_*}} \right) =
   \frac{\ll}{\beta} w_* \left( \left( \frac{\ll}{\beta}
     \right)^{\si_*} \cdot \right) \quad \text{in} \ H^1_r.
 \end{equation*}
 By comparison, we thus get that $\beta=\ll$. Since the limit is
 unique, the conclusion holds for the whole sequence. 
 \end{proof}

\subsection{End of the proof of Theorem~\ref{theo:algebraic}}
\label{sec:epsilon-prime}

In view of \eqref{eps-expression} and \eqref{eq:eps-sublin},
\begin{equation*}
  \frac{\epsilon(\si)}{\si_*-\si}=
  \frac{\epsilon(\si)-\epsilon(\si_*)}{\si_*-\si}=
  \frac{\|w_\si'\|_{L^2_r}^2}{\si(1+\si_*)\|w_\si\|_{L^2_r}^2}. 
\end{equation*}
The convergence $w_\si\to w_*$ in $H^1_r$ implies
\begin{equation}
\label{quotient-eps}
 \lim_{\si\to \si_*}
  \frac{\epsilon(\si_*)-\epsilon(\si)}{\si_*-\si} = -
  \frac{\|w_*'\|_{L^2_r}^2}{\si_*(1+\si_*)\|w_*\|_{L^2_r}^2} ,
\end{equation}
and thus $\si\mapsto \epsilon(\si)$ is $\mathcal{C}^1$ on $(0,\si_*]$,
with $\epsilon'(\si_*)$ given by the above quantity.
To compute this ratio, we use Emden--Fowler transformation,
\begin{equation}
\label{eq:EF}
\rho = e^t, \quad W_*(t) = e^{t/\sigma_*} w_*(\rho).
\end{equation}
This transformation, applied to the expression \eqref{eq:alg-soliton},
yields
\begin{equation}
\label{eq:W}
W_*(t) = \frac{e^{\frac{t}{\sigma_*}}}{(1+a e^{2t})^{\frac{1}{\sigma_*}}},
\end{equation}
and thus
\begin{equation*}
\| w_* \|^2_{L^2_r} = \int_0^{\infty} \frac{\rho^{1 +
    \frac{2}{\sigma_*}} \dd \rho}{(1+a \rho^2)^{\frac{2}{\sigma_*}}} =
\int_{-\infty}^{\infty} \frac{e^{2t + \frac{2t}{\sigma_*}} }{(1 + a
          e^{2t})^{\frac{2}{\sigma_*}}} \dd t .
\end{equation*}
On the other hand, integration by parts gives
\begin{align*}
  \| w_*' \|^2_{L^2_r}
  &= \frac{4 a^2}{\sigma_*^2} \int_0^{\infty} \frac{\rho^{3 +
    \frac{2}{\sigma_*}} \dd \rho}{(1+a \rho^2)^{2 +  \frac{2}{\sigma_*}}}  
= \frac{4 a (1 + \sigma_*)}{\sigma_*^2 (2+\sigma_*)} \int_0^{\infty}
    \frac{\rho^{1 + \frac{2}{\sigma_*}} \dd \rho}{(1+a \rho^2)^{1 +
    \frac{2}{\sigma_*}}}.  
\end{align*}
Similarly, we get
\begin{align*}
\| w_*' \|^2_{L^2_r} = \frac{4 a (1 + \sigma_*)}{\sigma_*^2
  (2+\sigma_*)} \int_{-\infty}^{\infty}   
\frac{e^{2t + \frac{2t}{\sigma_*}} }{(1 + a
  e^{2t})^{1+\frac{2}{\sigma_*}}} \dd t.
\end{align*}
To proceed further, we use the identity
\begin{equation}
\label{identity-1}
\int_{-\infty}^{\infty}  
e^{\frac{2t}{\sigma_*}} \frac{(1-a e^{2t})}{(1 + a e^{2t})^{1 +
    \frac{2}{\sigma_*}}} \dd t = \frac{\sigma_*}{2}
\int_{-\infty}^{\infty} \frac{\dd}{\dd t}
\frac{e^{\frac{2t}{\sigma_*}}}{(1 + a  e^{2t})^{\frac{2}{\sigma_*}}}
\dd t = 0,  
\end{equation}
to further obtain 
\begin{align}
\int_{-\infty}^{\infty}  
\frac{e^{\frac{2t}{\sigma_*}}}{(1 + a e^{2t})^{1 +
  \frac{2}{\sigma_*}}} \dd t
  &= a \int_{-\infty}^{\infty}  
\frac{e^{2t + \frac{2t}{\sigma_*}} }{(1 + a e^{2t})^{1 +
    \frac{2}{\sigma_*}}} \dd t \notag \\ 
&= -\frac{\sigma_*}{4}  \int_{-\infty}^{\infty}  
e^{\frac{2t}{\sigma_*}} \frac{\dd}{\dd t} \frac{1}{(1 + a
  e^{2t})^{\frac{2}{\sigma_*}}} \dd t 
\notag \\
&= \frac{1}{2} \int_{-\infty}^{\infty}
  \frac{e^{\frac{2t}{\sigma_*}}}{(1 + a e^{2t})^{\frac{2}{\sigma_*}}}
  \dd t.
\label{identity-2}
\end{align}
Therefore,
\begin{equation*}
  \| w_*' \|^2_{L^2_r}
  = \frac{4 (1 + \sigma_*)}{\sigma_*^2  (2+\sigma_*)}
    \int_{-\infty}^{\infty}    
\frac{e^{\frac{2t}{\sigma_*}} }{(1 + a e^{2t})^{1+\frac{2}{\sigma_*}}}
    \dd t 
= \frac{2 (1 + \sigma_*)}{\sigma_*^2 (2+\sigma_*)}
  \int_{-\infty}^{\infty} \frac{e^{\frac{2t}{\sigma_*}} }{(1 + a
  e^{2t})^{\frac{2}{\sigma_*}}} \dd t.  
\end{equation*}
On the other hand, for $\sigma_* < 1$ (or, equivalently, $d\ge 5$), we have 
\begin{equation*}
\int_{-\infty}^{\infty}  
e^{\frac{2t}{\sigma_*}} \frac{(1-a (1-\sigma_*) e^{2t})}{(1 + a
  e^{2t})^{\frac{2}{\sigma_*}}} \dd t = \frac{\sigma_*}{2}
\int_{-\infty}^{\infty}  
\frac{\dd}{\dd t} \frac{e^{\frac{2t}{\sigma_*}}}{(1 + a
  e^{2t})^{\frac{2}{\sigma_*}-1}} \dd t = 0, 
\end{equation*}
which yields 
\begin{align*}
  \| w_* \|^2_{L^2_r}
  &= \int_{-\infty}^{\infty}  \frac{e^{2t + \frac{2t}{\sigma_*}} }{(1 + a
 e^{2t})^{\frac{2}{\sigma_*}}}\dd t 
= \frac{1}{a (1-\sigma_*)} \int_{-\infty}^{\infty}  
\frac{e^{\frac{2t}{\sigma_*}} }{(1 + a e^{2t})^{\frac{2}{\sigma_*}}}
    \dd t.
\end{align*}
Using the explicit expressions for  $\| w_*' \|^2_{L^2_r}$ and 
$\| w_* \|^2_{L^2_r}$ in \eqref{quotient-eps} yields
\begin{equation}\label{eq:deriv-eps}
\epsilon'(\sigma_*) = 
-\frac{2 a (1-\sigma_*)}{\sigma_*^3 (2+\sigma_*)} = \frac{(\sigma_*-1)}{2 \sigma_*(1 + \sigma_*) (2+\sigma_*)} < 0,
\end{equation}
where we have used $a = \frac{\sigma_*^2}{4 (1+\sigma_*)}$.
Since by definition,
\begin{equation}
\label{eps-si-again}
  \epsilon(\si) = \(\alpha(\si)\)^{-2\si}\Longleftrightarrow
  \alpha(\si) = \epsilon(\si)^{-1/(2\si)},
\end{equation}
the asymptotics $\epsilon(\si)\sim (\si-\si_*)\epsilon'(\si_*)$ as
$\si\to \si_*$ yields the final claim of Theorem~\ref{theo:algebraic}.

\subsection{Further properties of the ground state near the
  algebraic soliton}

Related to the algebraic soliton $w_*$, we introduce the linearized operator 
$\mathcal{M}_0 : H^2_r \subset L^2_r \to
L^2_r$ given by  
\begin{equation}
\label{eq:M0}
\mathcal{M}_0 = - \frac{\dd^2}{\dd\rho^2} - \frac{d-1}{\rho}
\frac{\dd}{\dd \rho} - \frac{1+2\sigma_*}{(1+a \rho^2)^2}. 
\end{equation} 
It is a self-adjoint operator in $L^2_r$ with the essential
spectrum located on $[0,\infty)$ by Weyl's theorem.  

Since $w_*$ is characterized variationally as a constrained minimizer of 
\eqref{eq:variation-alg} with a single constraint, the Morse index of
$\mathcal{M}_0$ (the number of negative eigenvalues in
$L^2_r$) is either $0$ or $1$. Since 
\begin{equation*}
\langle \mathcal{M}_0 w_*, w_* \rangle = -2 \int_0^{\infty} \rho^{d-1}
|w_*(\rho)|^{2 \sigma + 2} \dd \rho < 0, 
\end{equation*}
the Morse index is exactly one. To characterize solutions 
of the homogeneous equation $\mathcal{M}_0 \mathfrak{v} = 0$, 
we note that $\rho = 0$ is a regular singular point with two linearly
independent solution $1 + \mathcal{O}(\rho^2)$ and $\rho^{2-d} \left[
  1 + \mathcal{O}(\rho) \right]$. Since the second solution is
singular and does not belong to $L^2_r$, we define the
unique solution $\mathfrak{v} \in \mathcal{C}^2(0,\infty)$ of the
initial-value problem  
\begin{equation} 
\label{eq:ivp-lin-alg}
\left\{
  \begin{array}{l}
\mathfrak{v}''(\rho) + \frac{d-1}{\rho} \mathfrak{v}'(\rho) + 
\frac{1+2\sigma_*}{(1+a \rho^2)^2} \mathfrak{v}(\rho) = 0, \\
    \mathfrak{v}(0) = 1, \quad \mathfrak{v}'(0) = 0.
  \end{array}
\right.
\end{equation}
We invoke \cite[Theorem~8.1, p.~92]{CoddingtonLevinson}:
\begin{equation*}
  X =
  \begin{pmatrix}
    \mathfrak{v}\\
    \mathfrak{v}'
  \end{pmatrix}
\end{equation*}
solves $X'=(A+V(\rho)+R(\rho))X$ with
\begin{equation*}
  A=
  \begin{pmatrix}
    0 & 1 \\
   0& 0
  \end{pmatrix},
  \quad
  V(\rho) =
  \begin{pmatrix}
    0 & 0 \\
   0& -\frac{d-1}{\rho}
  \end{pmatrix},
  \quad
  R(\rho) =
  \begin{pmatrix}
    0 & 0 \\
   -\frac{1+2\sigma_*}{(1+a \rho^2)^2} & 0
  \end{pmatrix}.
\end{equation*}
Since
$V'$ and $R$ are integrable on $(1,\infty)$,
with $\mu_1=\mu_2=0$ in the notations of
\cite[Theorem~8.1, p.~92]{CoddingtonLevinson}, 
$\mathfrak{v}(\rho)$ does not diverge as $\rho \to \infty$: it  satisfies 
\begin{equation}
\label{eq:v-limit}
\mathfrak{v}(\rho) \to \mathfrak{v}_{\infty} \quad \mbox{\rm as} \quad
\rho \to \infty, 
\end{equation}
with uniquely defined $\mathfrak{v}_{\infty} \in \mathbb{R}$. Since
the Morse index is exactly one, Sturm's theorem implies that
$\mathfrak{v}(\rho)$ has a single node such that $\mathfrak{v}(\rho) >
0$ for $\rho \in [0,\rho_0)$ and $\mathfrak{v}(\rho) < 0$ for $\rho
\in (\rho_0,\infty)$ so that $\mathfrak{v}_{\infty} \le  0$. However,
due to degeneracy of the minimizers of \eqref{eq:variation-alg} by the
scaling transformation, we prove in the following lemma that
$\mathfrak{v}_{\infty} = 0$ so that $\mathfrak{v} \in H^2_r
\subset L^2_r$ if $d \ge 5$. Since $\mathcal{M}_0 :
H^2_r \subset L^2_r \to L^2_r$ is not
Fredholm due to $0$ being an embedded eigenvalue in the end point of
the essential spectrum, we also characterize solutions of the
inhomogeneous equation $\mathcal{M}_0 \mathfrak{g} = f$ for a given $f
\in L^2_r$. 

\begin{lemma}
  \label{lem-Fredholm}
The exact solution of \eqref{eq:ivp-lin-alg} is given by 
\begin{equation}
\label{eq:kernel}
\mathfrak{v}(\rho) = \frac{1 - a \rho^2}{(1+a \rho^2)^{1+1/\sigma_*}},
\end{equation}
hence $\mathfrak{v} \in H^2_r$ if $d \ge 5$. For every $f
\in L^2_r$ and $d \ge 5$, there exists a unique solution
$\mathfrak{g} = \mathcal{M}_0^{-1} f$ satisfying $\mathfrak{g}(0) =
0$, $\mathfrak{g}'(0) = 0$, and
\begin{equation*}
\mathfrak{g}_{\infty} := \lim_{\rho \to \infty} \mathfrak{g}(\rho) = 0
\end{equation*} 
if and only if $\langle \mathfrak{v}, f \rangle = 0$.
\end{lemma}

\begin{proof}
Differentiating $\alpha^{1/\sigma*} w_*(\alpha \rho)$ with
respect to $\alpha$ at $\alpha = 1$ yields  
\begin{equation*}
\partial_{\alpha} \alpha^{1/\sigma*} w_*(\alpha \rho) |_{\alpha = 1} =
\frac{1 - a \rho^2}{\sigma_* (1+a \rho^2)^{1+1/{\sigma_*}}}.  
\end{equation*}
Multiplying it by $\sigma_*$ yields \eqref{eq:kernel} which satisfies
the initial conditions $\mathfrak{v}(0) = 1$ and $\mathfrak{v}'(0) =
0$. Due to the decay $\mathfrak{v}(\rho) \sim
\rho^{-2/\sigma_*}$ as $\rho \to \infty$, we have
$\mathfrak{v}_{\infty} = 0$ in \eqref{eq:v-limit}. Furthermore,
$\mathfrak{v} \in L^2_r$ if $d \ge 5$, and due to
smoothness, we have $\mathfrak{v} \in H^2_r$ if $d \ge 5$.  
	 
The second, linearly independent solution $\mathfrak{w} \in
\mathcal{C}^2(0,\infty)$ of $\mathcal{M}_0 \mathfrak{w} = 0$ is given
by the Wronskian relation  
\begin{equation}
\label{eq:Wronskian}
\mathfrak{v}(\rho) \mathfrak{w}'(\rho) - \mathfrak{v}'(\rho)
\mathfrak{w}(\rho) = \rho^{-(d-1)}, \quad \rho \in (0,\infty), 
\end{equation}
where the norming factor is uniquely chosen. It is clear from
\eqref{eq:Wronskian} that $\mathfrak{w}(\rho) \sim \rho^{-(d-2)}$ as
$\rho \to 0$ with the singularity prescribed at the regular singular
point $\rho = 0$. It is also clear from \eqref{eq:Wronskian} that
$\mathfrak{w}(\rho) \to \mathfrak{w}_{\infty}$ as $\rho \to \infty$
with $\mathfrak{w}_{\infty} \neq 0$. Solving $\mathcal{M}_0
\mathfrak{g} = f$ by the variation of constant formula, we get  
\begin{equation}
\label{eq:var-constant}
\mathfrak{g}(\rho) = \mathfrak{v}(\rho) \int_0^{\rho} \varrho^{d-1}
\mathfrak{w}(\varrho) f(\varrho) \dd \varrho - \mathfrak{w}(\rho)
\int_0^{\rho} \varrho^{d-1} \mathfrak{v}(\varrho) f(\varrho) \dd
\varrho. 
\end{equation}
The lower limit of integration in \eqref{eq:var-constant} is chosen at
$0$ to satisfy the initial conditions $\mathfrak{g}(0) =
\mathfrak{g}'(0) = 0$, e.g. if $f$ is bounded at $\rho = 0$, then
$\mathfrak{g}(\rho) \sim \rho^2$ as $\rho \to 0$. On the other hand,
we use the Cauchy--Schwarz inequality and obtain for $\rho_0 \gg 1$,  
\begin{equation*}
\left| \int_{\rho_0}^{\rho}\varrho^{d-1} \mathfrak{w}(\varrho)
  f(\varrho) d \varrho \right| \le C |\mathfrak{w}_{\infty}| \| f
\|_{L^2_r} \| 1 \|_{L^2_r(\rho_0,\rho)}  \le C |\mathfrak{w}_{\infty}|
\| f \|_{L^2_r} \rho^{\frac{d}{2}}.  
\end{equation*}
Since $\mathfrak{v}(\rho) \sim \rho^{-(d-2)}$ as $\rho \to \infty$ and
$d \ge 5$, the first term in \eqref{eq:var-constant} has the zero
limit as $\rho \to \infty$. Then, we compute from the second term in
\eqref{eq:var-constant}  that  
\begin{equation*}
\mathfrak{g}_{\infty} := \lim_{\rho \to \infty} \mathfrak{g}(\rho) =
-\mathfrak{w}_{\infty} \int_0^{\infty} \varrho^{d-1}
\mathfrak{v}(\varrho) f(\varrho) \dd \varrho,	 
\end{equation*}
where the last term is equivalent to $\langle \mathfrak{v}, f
\rangle$, which is well-defined since $f, \mathfrak{v} \in
L^2_{r}$ for $d \ge 5$. Thus $\mathfrak{g}_{\infty} = 0$ if
and only if $\langle \mathfrak{v}, f \rangle = 0$. 
\end{proof}

\begin{remark}
Resuming the Emden--Fowler transformation in the inhomogeneous case,
\begin{equation*}
\rho = e^t, \quad W_*(t) = e^{t/\sigma_*} w_*(\rho), \quad 
\mathfrak{G}(t) = e^{t/\sigma_*} \mathfrak{g}(\rho), \quad 
F(t) = e^{t/\sigma_*} f(\rho),
\end{equation*}
the relation $\mathcal{M}_0 \mathfrak{g} = f$ is transformed to the
equivalent form 
\begin{equation}
\label{eq:G}
-\mathfrak{G}''(t) + \frac{1}{\sigma_*^2} \mathfrak{G}(t) -
(1+2\sigma_*) |W_*(t)|^{2\sigma_*} \mathfrak{G}(t) = e^{2t} F(t),  
\end{equation}
where $W_*$, given by \eqref{eq:W}, is now exponentially decaying with
the rate $e^{-|t|/\sigma_*}$ as $|t| \to \infty$. The homogeneous
solution $\mathfrak{v}$ in \eqref{eq:kernel}  is related to the
translational mode $W'_*(t)$ after the Emden--Fowler transformation,
whereas the constraint $\langle \mathfrak{v}, f \rangle = 0$
is equivalent to the Fredholm condition $\int_{-\infty}^{\infty}
e^{2t} W_*'(t) F(t) \dd t = 0$ required to solve the linear
inhomogeneous equation \eqref{eq:G}  to avoid the exponential growth
of solutions at $\infty$. 
\end{remark}

The following proposition provides an alternative approach in the study of
the asymptotic behavior of
$\epsilon(\sigma) \to 0$ as $\sigma \to \sigma_*$ for $d \ge 5$ and recovers exactly the same expression for $\epsilon'(\sigma_*)$ given by \eqref{eq:deriv-eps}. 

\begin{proposition}\label{th-limit}
For every $d \ge 5$, there exist unique solutions 
$\mathfrak{z}_*,\mathfrak{w}_* \in \mathcal{C}^2(0,\infty)$ of the
linear inhomogeneous equations  
\begin{equation}
  \label{eq:deriv-alg}
\mathcal{M}_0 \mathfrak{z}_* = -w_*
\end{equation}
and
\begin{equation}
\label{eq:inhom-alg}
\mathcal{M}_0 \mathfrak{w}_* = (\ln w_*^2) w_*^{1 + 2\sigma_*},
\end{equation}
satisfying $\mathfrak{z}_*(0) = \mathfrak{z}_*'(0) = 0$ and
$\mathfrak{w}_*(0) = \mathfrak{w}_*'(0) = 0$. If the mapping
$(0,\sigma_*) \ni \sigma \mapsto \epsilon(\sigma) \in (0,\infty)$ is
$\mathcal{C}^1$ at $\sigma = \sigma_*$, then $\epsilon(\sigma_*) =
0$ and $\epsilon'(\sigma_*)$ is given by  \eqref{eq:deriv-eps}.
\end{proposition}

\begin{proof}
For $d \ge 5$, we have 
\begin{equation*}
w_* \in L^2_r \quad \mbox{\rm and} \quad 
(\ln w_*^2) w_*^{1 + 2\sigma_*} \in L^2_r.
\end{equation*}
Hence, solutions 
$\mathfrak{z}_* \in \mathcal{C}^2(0,\infty)$ and $\mathfrak{w}_* \in
\mathcal{C}^2(0,\infty)$ of  
\eqref{eq:deriv-alg} and \eqref{eq:inhom-alg} are well defined by
Lemma \ref{lem-Fredholm}. However, we show that both
$\mathfrak{z}_*(\rho)$ and $\mathfrak{w}_*(\rho)$ do not decay to $0$
as $\rho \to \infty$. \\  

\underline{For $\mathfrak{z}_*$,} we check the Fredholm condition
\begin{align*}
  \langle \mathfrak{v}, w_* \rangle
  &= \int_0^{\infty} \rho^{1 + \frac{2}{\sigma_*}} \frac{(1-a
    \rho^2)}{(1 + a \rho^2)^{1 + \frac{2}{\sigma_*}}} \dd \rho 
= \int_{-\infty}^{\infty} e^{2t + \frac{2t}{\sigma_*}} \frac{(1-a
  e^{2t})}{(1 + a e^{2t})^{1 + \frac{2}{\sigma_*}}} \dd t \\ 
&=\sigma_* \int_{-\infty}^{\infty} e^{2t} 
  W_*'(t) W_*(t) \dd t = -\sigma_*  \int_{-\infty}^{\infty} e^{2t}
  W_*^2(t) \dd t < 0, 
\end{align*}
where we have used the Emden--Fowler transformation  \eqref{eq:EF}
with $W_*(t)$ given by \eqref{eq:W}, and integrated by parts with the
sufficient decay of $W_*^2(t) \sim e^{-2|t|/\sigma_*}$ at $\pm
\infty$ since $\sigma_* < 1$ if $d \ge 5$. Since $\langle
\mathfrak{v}, w_* \rangle \neq 0$, the  
unique solution $\mathfrak{z}_* \in \mathcal{C}^2(0,\infty)$ 
of \eqref{eq:deriv-alg} satisfying $\mathfrak{z}_*(0) =
\mathfrak{z}_*'(0) = 0$  
does not decay to $0$ as $\rho \to \infty$. \\

\underline{For $\mathfrak{w}_*$,} we use the Emden--Fowler
transformation  \eqref{eq:EF} and check the Fredholm condition  
\begin{align*}
  \langle \mathfrak{v}, (\ln w_*^2) w_*^{1 + 2\sigma_*} \rangle
&= -\frac{2}{\sigma_*} \int_0^{\infty} \rho^{1 + \frac{2}{\sigma_*}}
    \frac{(1-a \rho^2)}{(1 + a \rho^2)^{3 + \frac{2}{\sigma_*}}} \ln(1
    + a \rho^2) \dd \rho \\ 
&= -\frac{2}{\sigma_*}  \int_{-\infty}^{\infty} e^{2t +
  \frac{2t}{\sigma_*}} \frac{(1-a e^{2t})}{(1 + a e^{2t})^{3 +
  \frac{2}{\sigma_*}}} \ln(1 + a e^{2t}) \dd t. 
\end{align*}
Since 
\begin{equation*}
\frac{\dd}{\dd t} \frac{e^{2t + \frac{2t}{\sigma_*}}}{(1 + a e^{2t})^{2 +
    \frac{2}{\sigma_*}}} = \frac{2 (1+\sigma_*)}{\sigma_*} e^{2t +
  \frac{2t}{\sigma_*}} \frac{(1-a e^{2t})}{(1 + a e^{2t})^{3 +
    \frac{2}{\sigma_*}}}, 
\end{equation*}
integration by parts removes the logarithmic term and yields 
\begin{align*}
  \langle \mathfrak{v}, (\ln w_*^2) w_*^{1 + 2\sigma_*} \rangle
  &= -\frac{1}{(1+\sigma_*)}  \int_{-\infty}^{\infty} \ln(1 + a
    e^{2t}) \frac{\dd}{\dd t} \frac{e^{2t + \frac{2t}{\sigma_*}}}{(1 + a
    e^{2t})^{2 + \frac{2}{\sigma_*}}}  \dd t \\ 
&= \frac{2a}{(1+\sigma_*)}  \int_{-\infty}^{\infty} \frac{e^{4t +
  \frac{2t}{\sigma_*}}}{(1 + a e^{2t})^{3 + \frac{2}{\sigma_*}}}  \dd
  t > 0. 
\end{align*}
Since $\langle \mathfrak{v}, (\ln w_*^2) w_*^{1 + 2\sigma_*}
\rangle \neq 0$, the unique solution $\mathfrak{w}_* \in
\mathcal{C}^2(0,\infty)$  of \eqref{eq:inhom-alg} satisfying
$\mathfrak{w}_*(0) = \mathfrak{w}_*'(0) = 0$ does not decay to $0$ as
$\rho \to \infty$. \\ 

\underline{End of the proof.}
We have proved the existence and uniqueness of solutions 
$\mathfrak{z}_*  \in \mathcal{C}^2(0,\infty)$ and $\mathfrak{w}_* \in
\mathcal{C}^2(0,\infty)$ of the linear inhomogeneous equations
\eqref{eq:deriv-alg} and \eqref{eq:inhom-alg}. Let $w_{\sigma}(\rho)
=w(\rho;\epsilon(\sigma),\sigma) 
\in \mathcal{C}^2(0,\infty) \cap L^{\infty}(0,\infty)$ be defined from
the family of solutions of \eqref{eq:ivp-w}. Suppose that $\epsilon$
is $\mathcal{C}^1$ up to $\si=\si_*$. Differentiating
\eqref{eq:ivp-w} with respect to $\sigma$ yields  
\begin{equation}
\label{eq:der-w-sigma}
\frac{\dd w_{\sigma}}{\dd\sigma}(\rho) |_{\sigma = \sigma_*} =
\epsilon'(\sigma_*) \mathfrak{z}_*(\rho) +  \mathfrak{w}_*(\rho). 
\end{equation}
Since $w_{\sigma}(\rho) \to 0$ as $\rho \to \infty$ for every $\sigma
\in (0,\sigma_*)$, we require  
$\frac{\dd w_{\sigma}}{\dd\sigma}(\rho) \to 0$ as $\rho \to\infty$ for
every $\sigma \in (0,\sigma_*)$ including the limit $\sigma \to
\sigma_*^-$. By Lemma \ref{lem-Fredholm}, this is possible if and only
if $\epsilon'(\sigma_*)$  is chosen such that 
\begin{equation}
\label{eq:tech}
- \epsilon'(\sigma_*) \langle \mathfrak{v}, w_* \rangle + 
\langle \mathfrak{v}, (\ln w_*^2) w_*^{1 + 2\sigma_*} \rangle = 0.
\end{equation}
In order to derive the explicit expression \eqref{eq:deriv-eps}, we
integrate by parts with the use of the Emden--Fowler transformation
\eqref{eq:EF} and \eqref{eq:W}:  
\begin{align*}
  \langle \mathfrak{v}, (\ln w_*^2) w_*^{1 + 2\sigma_*} \rangle
  &= -\frac{\sigma_*}{2 (1+\sigma_*)^2}  \int_{-\infty}^{\infty}  
e^{2t + \frac{2t}{\sigma_*}} \frac{d}{dt} \frac{1}{(1 + a e^{2t})^{2 +
    \frac{2}{\sigma_*}}} \dd t \\ 
&= \frac{1}{(1+\sigma_*)}  \int_{-\infty}^{\infty}  
 \frac{e^{2t + \frac{2t}{\sigma_*}}}{(1 + a e^{2t})^{2 +
  \frac{2}{\sigma_*}}} \dd t \\ 
&= -\frac{\sigma_*}{2a (1+\sigma_*) (2+\sigma_*)}  \int_{-\infty}^{\infty}  
e^{\frac{2t}{\sigma_*}} \frac{\dd}{\dd t} \frac{1}{(1 + a e^{2t})^{1 +
  \frac{2}{\sigma_*}}} \dd t \\ 
&= \frac{1}{a (1+\sigma_*) (2+\sigma_*)}  \int_{-\infty}^{\infty}  
 \frac{e^{\frac{2t}{\sigma_*}}}{(1 + a e^{2t})^{1 + \frac{2}{\sigma_*}}} \dd t,
\end{align*}
where all integration by parts are justified due to the fast
exponential decay at $\pm \infty$. Recalling \eqref{identity-1} and
\eqref{identity-2},  we have 
\begin{equation*}
\langle \mathfrak{v}, w_* \rangle = -\sigma_*
\int_{-\infty}^{\infty} \frac{e^{2t +\frac{2t}{\sigma_*}}}{(1 + a
  e^{2t})^{\frac{2}{\sigma_*}}} \dd t,
\end{equation*}
and 
\begin{equation*}
\langle \mathfrak{v}, (\ln w_*^2) w_*^{1 + 2\sigma_*} \rangle = 
\frac{1}{2 a (1+\sigma_*) (2+\sigma_*)}  \int_{-\infty}^{\infty}  
\frac{e^{\frac{2t}{\sigma_*}}}{(1 + a e^{2t})^{\frac{2}{\sigma_*}}} \dd t.
\end{equation*}
In order to show that one expression is proportional to the other one,
we note that  
\begin{equation*}
\frac{\dd}{\dd t} \frac{e^{\frac{2t}{\sigma_*} - 2t}}{(1 + a
  e^{2t})^{\frac{2}{\sigma_*}-1}} = \frac{2(1-\sigma_*)}{\sigma_*}
\frac{e^{\frac{2t}{\sigma_*} - 2t}}{(1 + a
  e^{2t})^{\frac{2}{\sigma_*}}} - \frac{2a}{\sigma_*}
\frac{e^{\frac{2t}{\sigma_*}}}{(1 + a e^{2t})^{\frac{2}{\sigma_*}}}. 
\end{equation*}
Since $\sigma_* < 1$, integration by parts yields due to the
exponential decay at $\pm \infty$ that  
\begin{equation*}
\langle \mathfrak{v}, (\ln w_*^2) w_*^{1 + 2\sigma_*} \rangle = 
\frac{(1-\sigma_*)}{2 a^2 (1+\sigma_*) (2+\sigma_*)}  \int_{-\infty}^{\infty}  
\frac{e^{\frac{2t}{\sigma_*}-2t}}{(1 + a e^{2t})^{\frac{2}{\sigma_*}}} \dd t.
\end{equation*}
Replacing $t = -\tilde{t} - \frac{1}{2} \ln a^2$ yields finally 
\begin{align*}
  \langle \mathfrak{v}, (\ln w_*^2) w_*^{1 + 2\sigma_*} \rangle
 &= \frac{(1-\sigma_*)}{2 (1+\sigma_*) (2+\sigma_*)}
   \int_{-\infty}^{\infty}
   \frac{e^{-\frac{2\tilde{t}}{\sigma_*}+2\tilde{t}}
   a^{-\frac{2}{\sigma_*}}}{(1 + a^{-1}
   e^{-2\tilde{t}})^{\frac{2}{\sigma_*}}} \dd \tilde{t} \\ 
&= -\frac{(1-\sigma_*)}{2 \sigma_*  (1+\sigma_*) (2+\sigma_*)}\langle
  \mathfrak{v}, w_* \rangle. 
\end{align*}
Substituting this relation into \eqref{eq:tech} yields
\eqref{eq:deriv-eps}. 
\end{proof}

\begin{remark}
We show that equation \eqref{eq:inhom-alg} for $\mathfrak{w}_*$ can be
reduced to equation \eqref{eq:deriv-alg} for $\mathfrak{z}_*$ by using
an elementary transformation. To do so, we rewrite
\eqref{eq:inhom-alg} explicitly: 
\begin{equation*}
\mathfrak{w}_*''(\rho) + \frac{(2 + \sigma_*)}{\sigma_* \rho}
\mathfrak{w}_*'(\rho) + \frac{1+2\sigma_*}{(1+a \rho^2)^2}
\mathfrak{w}_*(\rho) = \frac{2}{\sigma_*} \frac{\ln(1+a
  \rho^2)}{(1+a\rho^2)^2} \frac{1}{(1+a \rho^2)^{\frac{1}{\sigma_*}}}. 
\end{equation*}	
Substitution 
\begin{equation*}
\mathfrak{w}_*(\rho) = \frac{f(\rho)}{(1+a \rho^2)^{\frac{1}{\sigma_*}}}
\end{equation*}	
brings this equation to the form 
\begin{equation*}
f''(\rho) +  \frac{(2 + \sigma_*)}{\sigma_* \rho} f'(\rho) - \frac{4a
  \rho}{\sigma_* (1+a \rho^2)} f'(\rho) + \frac{2\sigma_*}{(1+a
  \rho^2)^2} f(\rho) = \frac{2}{\sigma_*} \frac{\ln(1+a
  \rho^2)}{(1+a\rho^2)^2}. 
\end{equation*}
Transformation 
\begin{equation*}
f(\rho) = \frac{1}{\sigma_*^2} \ln(1+a \rho^2) + g(\rho) 
\end{equation*}
brings the right-hand-side to a rational function 
\[g''(\rho) +  \frac{(2 + \sigma_*)}{\sigma_* \rho} g'(\rho) - \frac{4a
  \rho}{\sigma_* (1+a \rho^2)} g'(\rho) + \frac{2\sigma_*}{(1+a
  \rho^2)^2} g(\rho) 
  = -\frac{1}{\sigma_* (1+\sigma_*)}
\frac{(1+\sigma_*) - a \rho^2}{(1+a\rho^2)^2}. \]
By using the substitution 
\begin{equation*}
g(\rho) = b + c \rho^2 + h(\rho), 
\end{equation*}
we obtain coefficients $(b,c)$ to reduce the right-hand side to the
constant function. Elementary computations give 
\begin{equation*}
b = -\frac{1}{2 \sigma_*^2} \left[ 1 + \frac{1}{(1+\sigma_*) (2 +
 \sigma_*)} \right], \quad c = \frac{1}{8 (1+\sigma_*) (2 + \sigma_*)}, 
\end{equation*}
and 
\begin{equation*}
h''(\rho) + \frac{d-1}{\rho} h'(\rho) - \frac{4a \rho}{\sigma_* (1+a
  \rho^2)} h'(\rho) + \frac{2\sigma_*}{(1+a \rho^2)^2} h(\rho) =
\frac{4c (1 - \sigma_*)}{\sigma_*}. 
\end{equation*}
To summarize, the transformation 
\begin{equation*}
\mathfrak{w}_*(\rho) = \frac{\ln(1+a \rho^2) + \sigma_*^2 b +
  \sigma_*^2 c \rho^2}{\sigma_*^2 (1+a \rho^2)^{\frac{1}{\sigma_*}}} +
\tilde{\mathfrak{w}}_*(\rho) 
\end{equation*}
with the uniquely defined $(b,c)$ reduces \eqref{eq:inhom-alg} to 
\begin{equation*}
\mathcal{M}_0 \tilde{\mathfrak{w}}_* = -\frac{(1 - \sigma_*)}{2
  \sigma_* (1+\sigma_*) (2 + \sigma_*)} w_*, 
\end{equation*}	
which coincides with \eqref{eq:deriv-alg} up to the scalar multiplication. 
The first term in $\mathfrak{w}_*$ is decaying as $\rho \to \infty$ if
$d \ge 5$ but does not satisfy the initial condition
$\mathfrak{w}_*(0) = 0$. To correct the solution, we use the
homogeneous solution $\mathfrak{v}$ given by \eqref{eq:kernel}, which
is also decaying as $\rho \to \infty$, and redefine $\mathfrak{w}_*$
in the equivalent form: 
\begin{equation}
\label{eq:w-p-alg}
\mathfrak{w}_*(\rho) = \frac{\ln(1+a \rho^2) + \sigma_*^2 b +
  \sigma_*^2 c \rho^2}{\sigma_*^2 (1+a \rho^2)^{\frac{1}{\sigma_*}}} -
b \mathfrak{v}(\rho) + \frac{(1 - \sigma_*)}{2 \sigma_* (1+\sigma_*)
  (2 + \sigma_*)}  \mathfrak{z}_*(\rho), 
\end{equation}
so that $\mathfrak{w}_*(0) = \mathfrak{w}_*'(0) = 0$ is satisfied. 
By using \eqref{eq:w-p-alg}, we can rewrite \eqref{eq:der-w-sigma}
explicitly as 	 
\begin{align*}
  \frac{\dd w_{\sigma}}{\dd\sigma}(\rho) |_{\sigma = \sigma_*}
  &= \epsilon'(\sigma_*) \mathfrak{z}_*(\rho) +  \mathfrak{w}_*(\rho), \\
&= \left[ \epsilon'(\sigma_*) + \frac{(1 - \sigma_*)}{2 \sigma_*
  (1+\sigma_*) (2 + \sigma_*)} \right] \mathfrak{z}_*(\rho) \\ 
& \qquad  + \frac{\ln(1+a \rho^2) + \sigma_*^2 b + \sigma_*^2 c
  \rho^2}{\sigma_*^2 (1+a \rho^2)^{\frac{1}{\sigma_*}}} - b
  \mathfrak{v}(\rho), 
\end{align*}
Since $\mathfrak{z}_*$ does not decay to $0$ as $\rho \to \infty$, 
we have $\frac{\dd w_{\sigma}}{\dd\sigma}(\rho) |_{\sigma = \sigma_*} \to
0$ as $\rho \to \infty$ if and only if $\epsilon'(\sigma_*)$ satisfies
\eqref{eq:deriv-eps}. Thus, both the explicit solution for \eqref{eq:der-w-sigma} and the Fredholm condition \eqref{eq:tech} result in the same expression 
\eqref{eq:deriv-eps}, which was found from the quotient
\eqref{quotient-eps}. 
\end{remark}

\begin{remark}
In view of \eqref{eq:deriv-eps} and \eqref{eps-si-again}, we obtain
the leading-order asymptotic divergence of $\alpha(\sigma)$ as $\sigma
\to \sigma_*^-$ as  
\begin{equation*}
\alpha(\sigma) \sim \left( |\epsilon'(\sigma_*)| (\sigma_* -
\sigma) \right)^{-\frac{1}{2\sigma_*}} \quad \mbox{\rm as}
 \;\; \sigma \to \sigma_*^-. 
\end{equation*}
Furthermore, we have 
\begin{equation*}
w_{\sigma}(\rho) \sim w_*(\rho) + (\sigma - \sigma_*) \frac{\dd
  w_{\sigma}}{\dd \sigma}(\rho) |_{\sigma = \sigma_*}, 
\end{equation*}
where the correction term
\begin{align} \label{eq:asymp_eps}
(\sigma - \sigma_*) \frac{\dd w_{\sigma}}{\dd\sigma}(\rho) |_{\sigma =
  \sigma_*} =  (\sigma - \sigma_*) \left[ \frac{\ln(1+a \rho^2) +
  \sigma_*^2 b + \sigma_*^2 c \rho^2}{\sigma_*^2 (1+a
  \rho^2)^{\frac{1}{\sigma_*}}} - b \mathfrak{v}(\rho) \right] 
\end{align}
is positive for $\rho \in (0,\rho_0)$ and negative for $\rho \in
(\rho_0,\infty)$ for some $\rho_0 > 0$. 	
\end{remark}

\begin{remark}
	Due to the term $c \rho^2$ in \eqref{eq:asymp_eps} with $c \neq 0$, the
	first term in \eqref{eq:asymp_eps} is not in $L^2_r$ for $5
	\le d \le 8$. This shows that the ground state near the algebraic 
	soliton cannot be generally expanded as powers of $(\sigma_* - \sigma)$ in  $L^2_r$. 
\end{remark}

\begin{remark}
  For $d = 4$, we have $\sigma_* = 1$ so that the exact solution
  $\mathfrak{w}_*$ given by \eqref{eq:w-p-alg} is independent of
  $\mathfrak{z}_*$. It is clear that the solution
  \eqref{eq:w-p-alg} is non-decaying as $\rho \to \infty$ due to
  the $c \rho^2$ term. Nevertheless, the balance with the term
  $\mathfrak{z}_*$ is impossible since $w_* \notin
  L^2_r$ for $d = 4$ and $\mathfrak{z}_*$ is
  logarithmically growing as $\rho \to \infty$. This shows that
the asymptotic behavior of $\epsilon(\sigma)$ as $\sigma \to \sigma^*$ is
  more complicated than the power expansion.  
  Similarly, we do not have a balance between $\mathfrak{z}_*$ growing as 
  $\mathcal{O}(\rho)$ and bounded $\mathfrak{w}_*$ for $d = 3$
  ($\sigma_* = 2$).  
  Modifications of the asymptotic behavior of the ground state
  near the algebraic soliton for $d = 4$ and $d = 3$ are discussed 
  within the Gross--Pitaevskii equation with a harmonic potential 
  in \cite{Pel1,Pel2}.
  The modified asymptotic formula for $\alpha(\si)$ for
  $d=3,4$ was found in \cite{GaPePuSe03}. The approach introduced in
  \cite{WeiWu2022}, based on  the projection of the Aubin-Talenti
  soliton under the Bessel operator, may provide an alternative
  approach for this case.
\end{remark}

\section{Numerical approximations of the ground state}
\label{sec:num}

\subsection{Radial finite differences}

We first recall the definition of the radial Lebesgue spaces $L^p_r(\R^d)$ associated with the norms
\[ \| u \|_{L^p_r} = \left( C(d) \int_0^{\infty} |u(r)|^{p} r^{d-1}
    \dd r  \right)^{1/p} , \] 
with the constants $C(d)$ given in Table \ref{table_d_sphere}.
\begin{table}[h] 
\centering
\begin{tabular}{ | c || c | c | c | c | c | }   
\hline
 $d$ &  1  & 2 & 3 & 4 & 5 \\
 \hline
 $C(d)$ & $2$ & $2 \pi$ & $4 \pi$ &  $2 \pi^2$ & $ \frac{8}{3}\pi^2$ \\ 
 \hline
\end{tabular}
\caption{Surface area of the unit sphere in $\mathbb{R}^d$ for $d = 1,\ldots,5$.}
\label{table_d_sphere}
\end{table}

The $d$-dimensional radial Laplace operator 
\[ \Delta_r u = \frac{1}{r^{d-1}}  \partial_r \left( r^{d-1} \partial_r u  \right)   \]
on the finite interval $\left[0,R\right]$, with Neumann boundary condition at $r=0$ and Dirichlet boundary condition at $R>0$, is then discretized as follows. We fix a mesh size $h=\frac{2R}{2M+1}$ with $M>0$ an integer, and define both \textit{regular} and \textit{staggered} grid points as
\[  r_j= j h \quad \text{and} \quad r_{j+\frac12}=
  \left(j+\frac12\right) h, \quad 0\leq j \leq M, \] 
so that we define the approximation of the radial Laplace operator $\Delta_r^h$ on the staggered grid as
\[ \Delta_r^h u_{j+\frac12} = \frac{1}{h^2}
  \frac{1}{r_{j+\frac12}^{d-1}} \left( r_{j+1}^{d-1} u_{j+\frac{3}{2}}
    - ( r_{j+1}^{d-1} + r_{j}^{d-1}) u_{j+\frac12} + r_{j}^{d-1}
    u_{j-\frac12} \right), \] 
for $0\leq j \leq M-1$, with Neumann and Dirichlet boundary conditions imposed by 
\[ u_{-\frac12}= u_{\frac12} \quad \text{and} \quad u_{M+\frac12}=0.  \]
Note that we have also introduced a ghost point $r_{-\frac12}$ to
approximate the Neumann boundary condition $u'(0)=0$ with the
second-order accuracy. 


\subsection{Gradient flow with $L^{2\sigma+2}$ normalization} \label{sec:num_L_2sigma_normalization}

The unique ground state with the profile $u_\si$ satisfying \eqref{eq:u} for $0 < \sigma < \sigma_*$ can be numerically approximated as follows. We first recall the definition of the \textit{Nehari manifold} for the variational problem \eqref{eq:variation}:
\[  
\mathcal{N}= \enstq{\phi \in H^1(\R^d)}{I(\phi) = \frac{1}{\sigma}\int_{\R^d} |\phi|^{2\sigma +2}, \quad \phi \neq 0}, 
\]
associated to the quadratic functional
\[ 
I(\phi) = \frac12 \| \nabla \phi \|^2_{L^2(\mathbb{R}^d)} + \frac{1}{\sigma} \| \phi \|^2_{L^2(\mathbb{R}^d)}. 
\]
We then perform, inspired by the method of \cite{Wang2022}, a normalized gradient flow scheme as follows. Starting from an initial radial state $\phi^0_{\sigma}(r)=e^{-r^2}$ for $r\in \R_+$, we realize a linearly implicit normalized gradient flow that writes for $n \in \N^*$ as
\begin{equation}
\label{numerical-method-1}
\left\{ \begin{aligned}
&\frac{\phi^{\ast,n+1}_{\sigma}-\phi^n_{\sigma}}{\tau} =  \Delta_r \phi^{\ast,n+1}_{\sigma}+\frac{1}{\sigma} ( |\phi^n_{\sigma}|^{2\sigma}-1 ) \phi^{\ast,n+1}_{\sigma}, \\
&\phi^{n+1}_{\sigma}=\lambda_{n+1} \phi^{\ast,n+1}_{\sigma}, \quad \lambda_{n+1}=\left( \frac{\sigma  I(\phi^{\ast,n+1}_{\sigma})}{\| \phi^{\ast,n+1}_{\sigma}\|_{L^{2\sigma+2}_{r}}^{2\sigma +2}}  \right)^{\frac{1}{2\sigma}}.
\end{aligned}  \right. 
\end{equation}
We stop the algorithm when 
\begin{equation} \label{eq:stopping_criterion}
\frac{\| \phi^{n+1}_{\sigma}-\phi^n_{\sigma} \|_{L^2_r}}{\tau} \leq \eps  
\end{equation}
for a given threshold $\eps>0$. A fixed point
$(\phi_\si,\phi_\si^\ast)$ of the iterative method \eqref{numerical-method-1} is then solution to  $\phi_\si=\lambda \phi_\si^\ast$ with 
$\lambda^{2\si} =  \sigma I(\phi^{\ast}_{\sigma})/\|
\phi^{\ast}_{\sigma}\|_{L^{2\sigma+2}_{r}}^{2\sigma +2}$ 
and  
\[  \frac{1-\lambda}{\tau} \phi_{\sigma}=  \Delta_r \phi_{\sigma}+\frac{1}{\sigma} ( |\phi_{\sigma}|^{2\sigma}-1 ) \phi_{\sigma}.
\]
Denoting $\gamma=1+\si(1-\lambda)/\tau$ and performing the rescaling
\[ 	
\mathrm{u}_{\sigma}(r)= \gamma^{-\frac{1}{1-2\si}}\phi_\si
  \left( \sqrt{\gamma} r \right),   
\] 
we get the numerical approximation $\mathrm{u}_{\sigma}$ of the profile
$u_{\sigma}$ of the ground state. In the following interpretation of numerical results, we identify $\mathrm{u}_\si$ for $\si \in (0,\si_*)$.

\begin{remark} \label{rem:stiffness}
The number of iterations needed in order to
achieve the stopping criterion \eqref{eq:stopping_criterion} greatly increases as $\sigma
\rightarrow 0$ or as $\sigma \rightarrow \sigma_*$ for $d \geq
3$. This suggests that our numerical scheme is stiff with respect to
both endpoint limits. In particular, we hardly go beyond $\sigma = 1.6$ for $d=3$ and $\sigma= 0.9$ for $d=4$ as $\sigma \rightarrow \sigma_*$.
\end{remark}

In Figure~\ref{fig:GS_2D_up_to_8_with_r_sigma.png} we plot
the (approximated) ground state profile $\mathrm{u}_{\sigma}$ for~$\sigma$ varying
between 0.1 and 8 in 2D, as well as the Gausson $u_0$ explicitly given
by~\eqref{eq:Gaussian} and the expected root $r_0=\sqrt{2+2\sqrt{2}}$
computed through equation \eqref{eq:root_r_0}, in both linear scale and logarithmic scale. We see that the successive ground states $u_{\sigma}$  do converge towards the Gausson $u_0$ as $\si \to 0$. The crossing point $r_\si>0$ between curves $u_\si$ and $u_0$ also tends towards the expected root $r_0$ of equation \eqref{eq:root_r_0} as $\si \to 0$. On the
other hand, we observe in the limit $\si \to \infty$ that the ground
state profile becomes steeper and steeper at the origin. 

\begin{figure}[htb!]
\centering
\hspace*{-1cm}
\includegraphics[width=1\textwidth,trim = 6cm 18cm 8cm 18cm, clip]{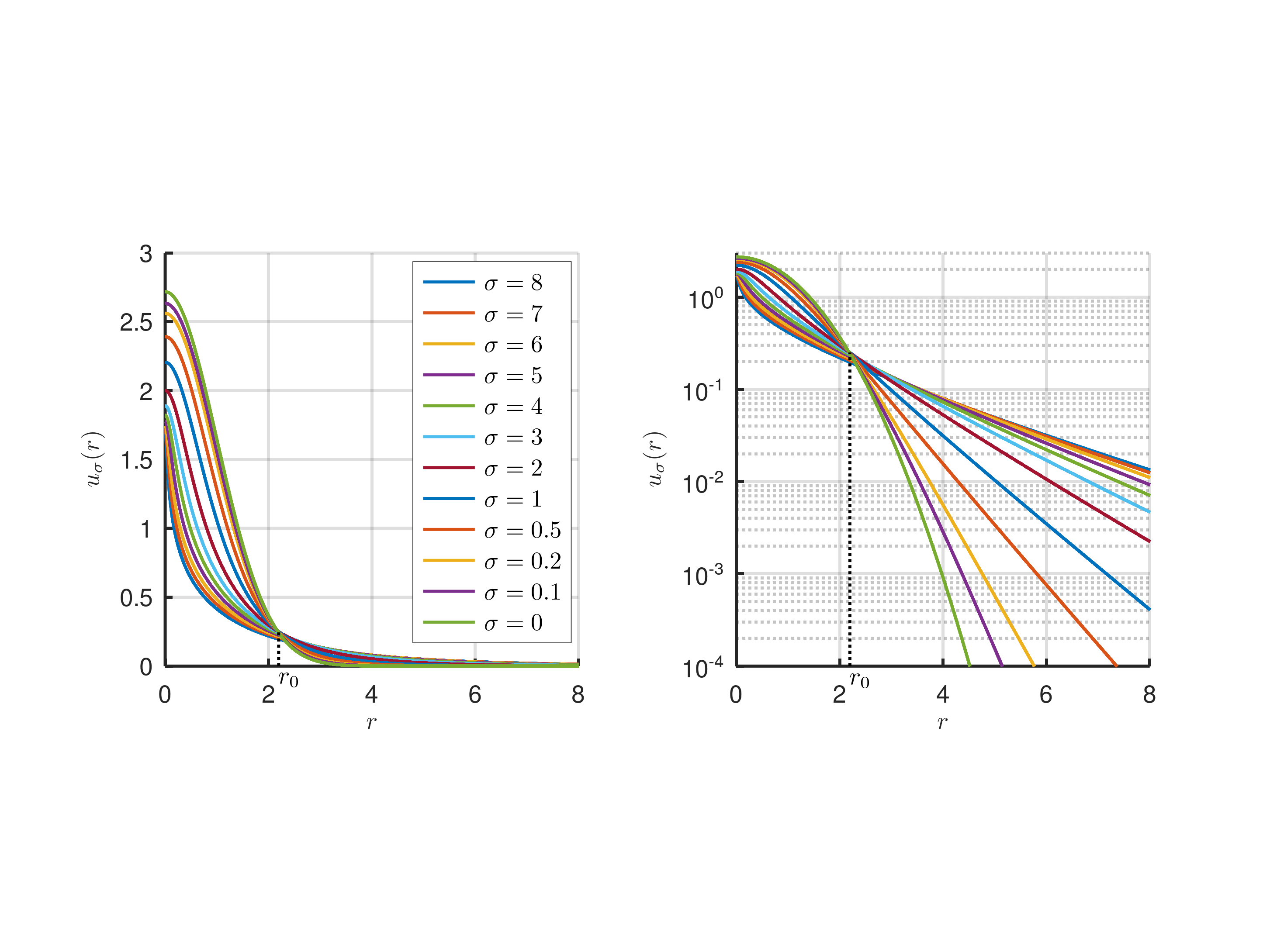}	  
\caption{Ground state profile $u_{\sigma}$ for $d=2$ in linear scale \textit{(left)} and logarithmic scale \textit{(right)}, for different values of $\si$.}
\label{fig:GS_2D_up_to_8_with_r_sigma.png}
\end{figure}

Recall that $\alpha(\sigma)=\|u_{\sigma}\|_{L^{\infty}}=u_{\sigma}(0)$
for $0< 
\sigma < \sigma_*$, and
$\alpha(0)=\|u_{0}\|_{L^{\infty}}=u_{0}(0)=e^{d/2}$. We plot in
Figure~\ref{fig:Evolution_d_3_4_5_expected_slopes.png} the dependence
of  $\alpha(\sigma)=u_{\sigma}(0)$ normalized by the value of
$\alpha(0)$ versus $\sigma \in (0,0.5)$ for $d=3, 4, 5$, as well as
the expected slopes at the origin 
explicitly given by \eqref{eq:slope_at_0}. Each curve matches its
respective slope at $\si=0$.

\begin{figure}[htb!]
	\centering
	\includegraphics[width=0.8\textwidth,trim = 25cm 0cm 25cm 2cm, clip]{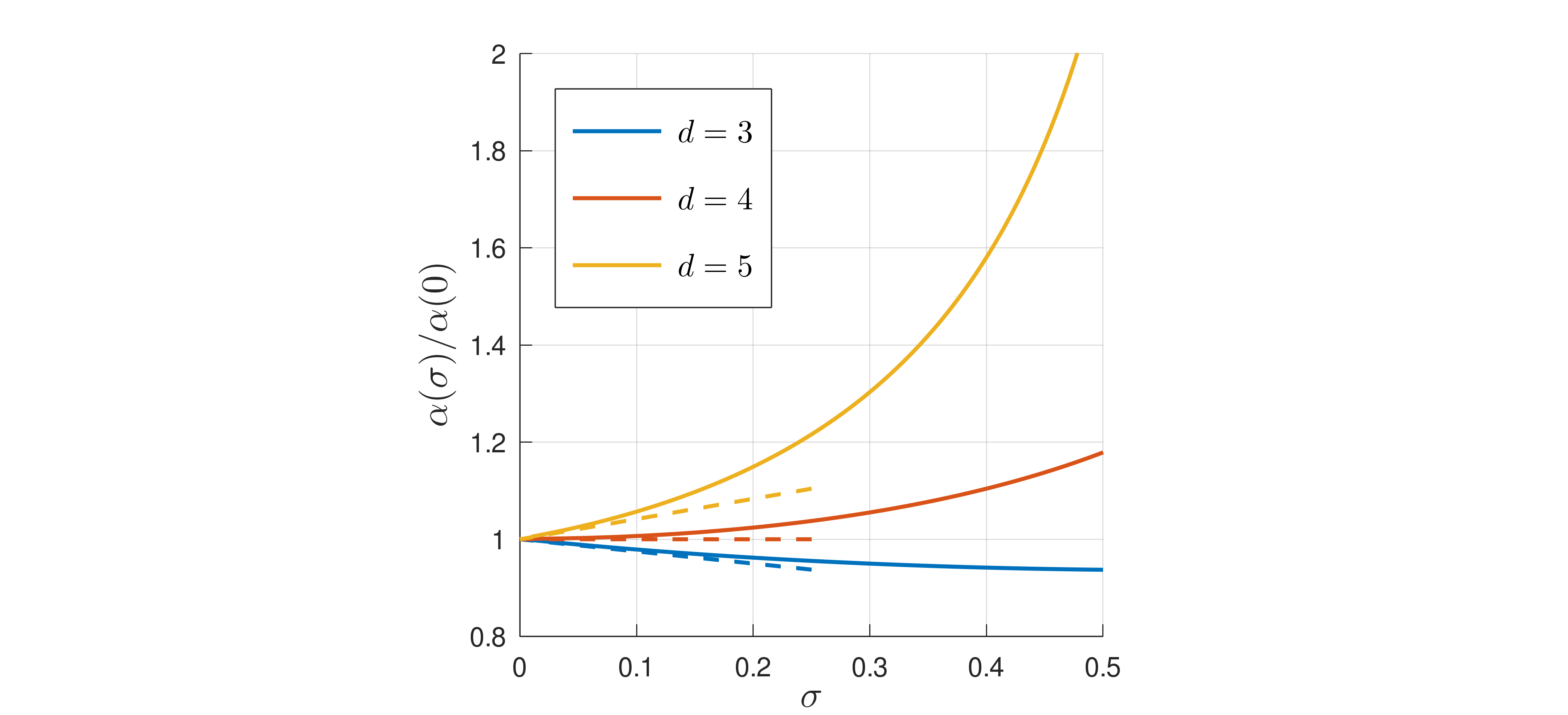}	   
	\caption{Dependence of $\alpha(\sigma)/\alpha(0)$ versus $\sigma$ and 
		the expected slopes \eqref{eq:slope_at_0} for $d=3,4,5$.}  
	\label{fig:Evolution_d_3_4_5_expected_slopes.png}
\end{figure}

\subsection{Gradient flow with $L^{\infty}$ normalization}

We now illustrate the limit $\sigma \to \sigma_*$ for
$d\geq 3$, which corresponds to the convergence to the algebraic
soliton. One should first mention that this limit is very stiff for
the $L^{2\sigma+2}$ normalization algorithm as the $L^\infty$-norm of
the ground states is unbounded when $\sigma \to \sigma_*$. This
motivates the use of a new gradient flow approach, based on the
formulation \eqref{eq:ivp-w}. We now perform, starting from the
explicit initial radial state $w^0_\si(\rho)=w_{*}(\rho)$ for $\rho\in
\R_+$, a linearly implicit normalized gradient flow that writes for $n
\in \N^*$ as 
\begin{equation}
\label{numerical-method-2}
\left\{ \begin{aligned}
&\frac{w^*_{\sigma}-w^n_{\sigma}}{\tau} =  \Delta_\rho w^*_{\sigma}+
  |w^n_{\sigma}|^{2\sigma} w^*_{\sigma} -\epsilon_0(\sigma)
  w^*_{\sigma}, \\ 
&w^{n+1}_{\sigma}= \frac{w^*_{\sigma}}{\| w^*_{\sigma}\|_{L^{\infty}}},
\end{aligned}  \right. 
\end{equation}
where we use the approximation 
\[ 
\epsilon_0(\sigma) :=  \frac{(\sigma_*-1)(\sigma_*-\sigma)}{2
    \sigma_*(1+\sigma_*)(2+\sigma_*)}  
\] 
from equation \eqref{eq:deriv-eps} (instead of the implicit constant
$\epsilon(\sigma)$). We stop the algorithm when 
\[ \frac{\| w^{n+1}_{\sigma}-w^n \|_{L^2_r}}{\tau} \leq \eta  \]
for a fixed threshold $\eta>0$. A fixed point
$( \mathrm{w}_\si, \mathrm{w}_\si^\ast)$ of the iterative method \eqref{numerical-method-2} 
is then solution to  $\mathrm{w}_\si^\ast = \mu \mathrm{w}_\si$ with 
$\mu =  \| \mathrm{w}_\si^*\|_{L^{\infty}}$ and  
\[  
\frac{\mu-1}{\mu \tau} \mathrm{w}_\si = \Delta_{\rho} \mathrm{w}_\si +
\mathrm{w}_\si^{2\sigma+1} - \epsilon_0(\sigma) \mathrm{w}_\si.
\]
Therefore, the numerical value of $\epsilon(\si)$ in the
formulation \eqref{eq:ivp-w} is adjusted as 
\[ 	
\underline{\epsilon}(\sigma) := \epsilon_0(\sigma) + \frac{\mu-1}{\mu \tau}, 
\] 
where $\mathrm{w}_{\si}$ provides the numerical approximation of the profile $w_{\si}$.

In order to illustrate the convergence of the ground state profile $w_{\si}$
to the algebraic soliton $w_*$ as $\sigma \to \sigma_* = \frac{2}{3}$ for $d=5$, we perform both
normalized gradient flow methods on the range $\sigma \in \left[0.54,
  0.66 \right]$. More precisely: 
\begin{itemize}
\item For $\sigma=0.54$ and $0.58$, we perform the gradient flow with
  $L^{2\sigma+2}_r$ normalization based on the numerical method \eqref{numerical-method-1}, after which we rescale the solution 
  to the profile $w_{\si}$ by using the scaling transformation 
  \eqref{eq:transformation}. 
\item For $\sigma =0.62$ and $0.66 $, we perform the gradient flow
  with $L^{\infty}$ normalization based on the numerical method \eqref{numerical-method-2}.
\end{itemize}

In Figure~\ref{fig:ground_state_sigma_star.png} we plot the
corresponding approximated ground states $\mathrm{w}_{\sigma}$ 
for different values of $\si$, as well as the algebraic soliton $w_*$ 
for $\si_*$ given by \eqref{eq:alg-soliton}, in both linear scale and logarithmic scale. Once again, this illustrates the convergence of the ground state $w_{\si}$ towards the algebraic soliton $w_*$. 

In Figure~\ref{fig:test_epsilon_sigma.png}, we plot the dependence of  $\underline{\epsilon}(\sigma)$ versus $\sigma \in
\left[0.5,0.66\right]$, along with the predicted evolution slope 
$\epsilon_0(\sigma)$. The dependence $\underline{\epsilon}(\sigma)$ 
matches $\epsilon_0(\sigma)$ as $\si \to \si_*$. 

\begin{figure}[htb!]
\centering
\includegraphics[width=1\textwidth,trim = 6cm 17cm 8cm 18cm, clip]{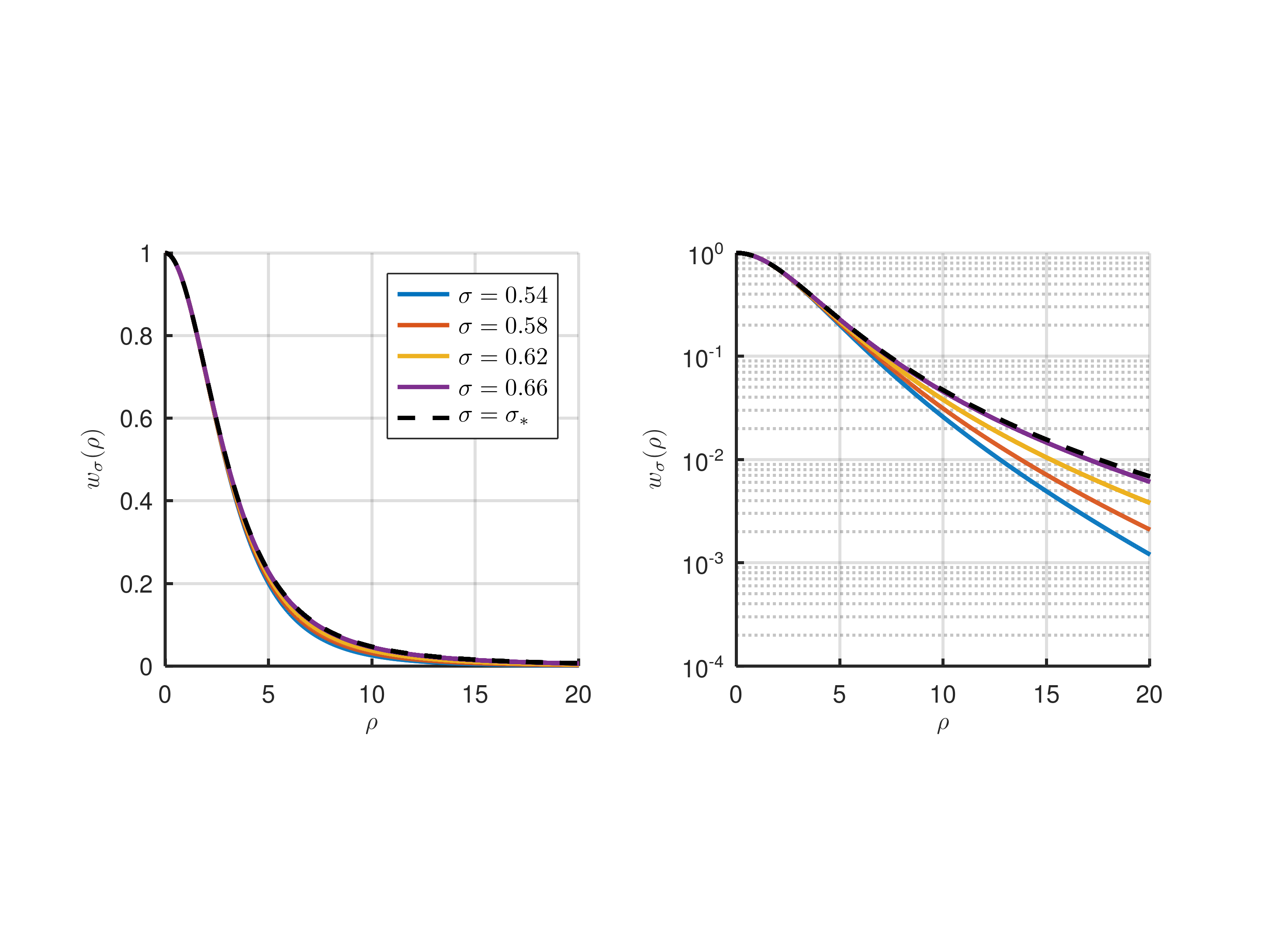} 
\caption{Ground state profile $w_{\sigma}$ for $d = 5$ in the linear scale
  \textit{(left)} and the logarithmic scale \textit{(right)}.} 
\label{fig:ground_state_sigma_star.png}
\end{figure}

\begin{figure}[htb!]
\centering
\includegraphics[width=0.8\textwidth,trim = 20cm 0cm 20cm 4cm, clip]{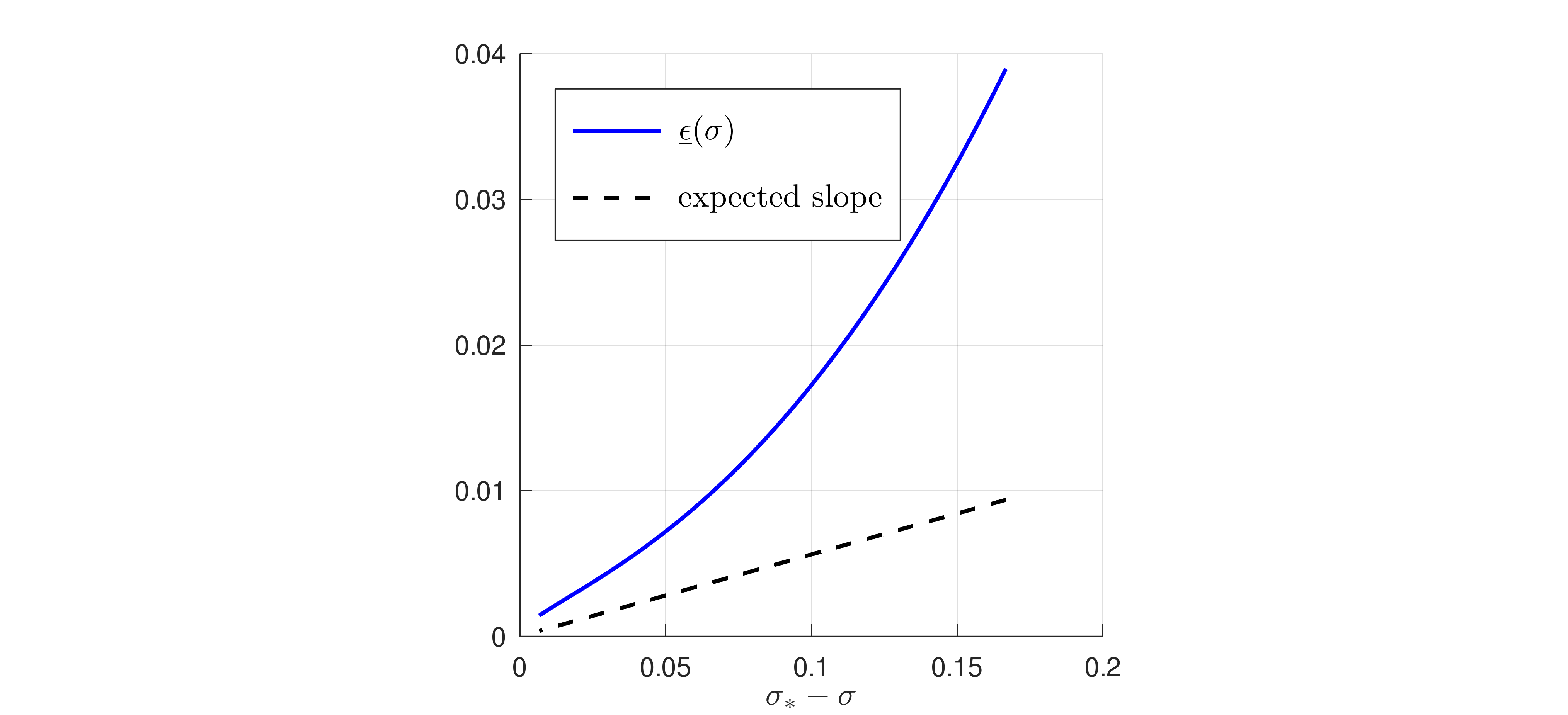}	  	 
\caption{Dependence of $\underline{\epsilon}(\sigma)$ (solid curve) 
	and the asymptotic approximation $\epsilon_0(\sigma)$ (dashed line)
versus $\sigma \in \left[0.5,0.66\right]$ for $d = 5$.} 
\label{fig:test_epsilon_sigma.png}
\end{figure}

Finally, we plot the difference $w_{\sigma}-w_*$ for several $\sigma$
in Figure~\ref{fig:crossing_point_5D.png}, as well as the expected
limit crossing point  
\[ \rho_0=\underset{\si \to \si_*}{\lim} \mathrm{argmin} \enstq{\rho
    >0}{w_{\sigma}(\rho)-w_*(\rho)}>0,  \] 
obtained from the asymptotic approximation \eqref{eq:asymp_eps}. The value of $\rho_0$ was computed numerically as the unique
positive root of the function given by \eqref{eq:asymp_eps}.

\begin{figure}[htb!]
\centering
\includegraphics[width=0.7\textwidth,trim = 12cm 3cm 14cm 8cm, clip]{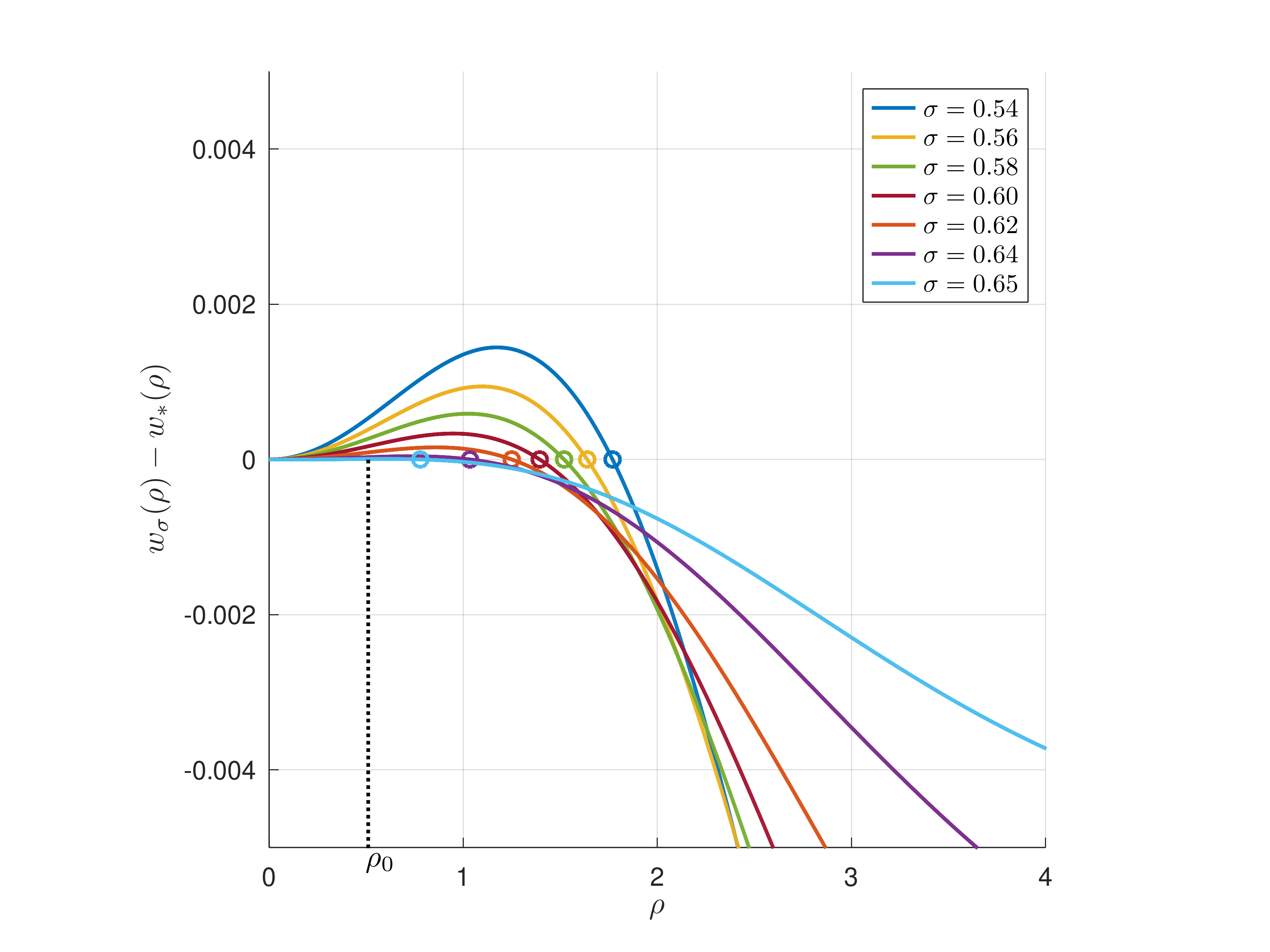}	  
\caption{Difference $w_{\sigma}-w_*$ and expected crossing point $\rho_0$.}
\label{fig:crossing_point_5D.png}
\end{figure}






\bibliographystyle{abbrv}
\bibliography{ground}

\end{document}